\documentclass{amsart}
\usepackage{latexsym}
\usepackage{amsmath}
\usepackage{amssymb}
\usepackage[all]{xy}

\usepackage{pstricks}
\newpsobject{showgrid}{psgrid}{subgriddiv=1,griddots=10,gridlabels=6pt}

\newtheorem{thm}{Theorem}[section]
\newtheorem{prop}[thm]{Proposition}
\newtheorem{conjecture}[thm]{Conjecture}
\newtheorem{cor}[thm]{Corollary}
\newtheorem{lem}[thm]{Lemma}
\theoremstyle{definition}

\theoremstyle{remark}
\newtheorem{rem}[thm]{Remark}
\newtheorem{ex}[thm]{Example}

\newcommand{\K}{{\mathbb K}}
\newcommand{\e}{\varepsilon}
\newcommand{\mapright}[1]{%
 \smash{\mathop{%
  \hbox to 1cm{\rightarrowfill}}\limits_{#1}}}
\newcommand{\maprightd}[2]{%
 \smash{\mathop{%
  \hbox to 1.2cm{\rightarrowfill}}\limits^{#1}\limits_{#2}}}
\newcommand{\mapleft}[1]{%
 \smash{\mathop{%
  \hbox to 1cm{\leftarrowfill}}\limits_{#1}}}
\newcommand{\mapleftu}[1]{%
 \smash{\mathop{%
  \hbox to 0.8cm{\leftarrowfill}}\limits^{#1}}}
\newcommand{\maprightu}[1]{%
 \smash{\mathop{%
  \hbox to 1cm{\rightarrowfill}}\limits^{#1}}}
\newcommand{\maprightud}[2]{%
 \smash{\mathop{%
  \hbox to 1cm{\rightarrowfill}}\limits^{#1}_{#2}}}
\newcommand{\mapleftud}[2]{%
 \smash{\mathop{%
  \hbox to 1cm{\leftarrowfill}}\limits^{#1}_{#2}}}

\newcounter{eqn}[section]

\def\theeqn{\textnormal{(\thesection.\arabic{eqn})}}

\def\eqnlabel#1{%
  \refstepcounter{eqn}%
  \label{#1}%
  \leqno{\theeqn}}

\begin{document}

\title[The BV algebra in String Topology of classifying spaces]{
The BV algebra in String Topology of classifying spaces
}


\footnote[0]{


Department of Mathematical Sciences, 
Faculty of Science,  
Shinshu University,   
Matsumoto, Nagano 390-8621, Japan   
e-mail:{\tt kuri@math.shinshu-u.ac.jp} 

D\'epartement de Math\'ematiques
Facult\'e des Sciences
Universit\'e d'Angers
49045 Angers
France 
e-mail:{\tt luc.menichi@univ-angers.fr}
}

\author{Katsuhiko KURIBAYASHI and Luc MENICHI}
\date{}

\vspace{-1cm}

\begin{abstract}
For almost any compact connected Lie group $G$ and any field $\mathbb{F}_p$, we compute the Batalin-Vilkovisky
algebra $H^{*+\text{dim }G}(LBG;\mathbb{F}_p)$ on the loop cohomology of the classifying space introduced by
Chataur and the second author.
In particular, if $p$ is odd or $p=0$, this Batalin-Vilkovisky algebra is isomorphic
to the Hochschild cohomology $HH^*(H_*(G),H_*(G))$. Over $\mathbb{F}_2$, such isomorphism of Batalin-Vilkovisky algebras
does not hold when $G=SO(3)$ or $G=G_2$.

\end{abstract}
\maketitle

\section{Introduction}

Let $M$ be a closed oriented smooth manifold and let $LM$ denote
the space of free loops on $M$. Chas and Sullivan~\cite{Chas-Sullivan:stringtop} have defined a product on the homology of $LM$,
called the {\em loop product},
$
H_*(LM)\otimes H_*(LM)\rightarrow H_{*-\text{dim }M}(LM)
$.
They showed that this loop product, together with the homological BV-operator
$\Delta:H_*(LM)\rightarrow H_{*+1}(LM)$, make the shifted free loop space homology
$\mathbb{H}_*(LM):=H_{*+\text{dim }M}(LM)$ into a Batalin-Vilkovisky algebra,
or BV algebra. Over $\mathbb{Q}$, when $M$ is simply-connected, this BV algebra
can be computed using Hochschild cohomology~\cite{Felix-Thomas:ratBVstringtop}.
In particular, if $M$ is formal over $\mathbb{Q}$, there is an isomorphism of
BV algebras between $\mathbb{H}_*(LM)$ and $HH^*(H^*(M;\mathbb{Q}),H^*(M;\mathbb{Q}))$,
the Hochschild cohomology of the symmetric Frobenius algebra $H^*(M;\mathbb{Q})$.
Over a field $\mathbb{F}_p$, if $p\neq 0$, this BV algebra $\mathbb{H}_*(LM)$
is hard to compute. It has been computed only for complex Stiefel manifolds~\cite{Tamanoi:BVLalg},
spheres~\cite{menichi:stringtopspheres}, compact Lie groups~\cite{Hepworth:stringLie,Menichi:BVmorphismdoublefree} and
complex projective spaces~\cite{Chataur-Leborgne:loophomologycomplexprojective,Hepworth:stringcomplexprojective}.

Let $G$ be a connected compact Lie group of dimension $d$ and let $BG$ its classifying space.
Motivated by Freed-Hopkins-Teleman twisted K-theory~\cite{FreedHopkinsTeleman:twistedKtheory3}
and by a structure of symmetric Frobenius algebra on $H_*(G)$, Chataur and the second
author~\cite{Chataur-Menichi:stringclass} have proved that the homology 
of the free loop space $LBG$ with coefficients in a field $\K$ admits the structure 
of a $d$-dimensional homological conformal field theory
(More generally, if $G$ acts smoothly on $M$, Behrend, Ginot, Noohi and Xu~\cite[Theorem 14.2]{BGNX:Stringstackspublie}
have proved that $H_*(L(EG\times_G M))$ is a $(d-\text{dim }M)$-homological conformal
field theory.).
In particular, the operation associated with 
a cobordism connecting one dimensional manifolds called the pair of pants, defined a product on the cohomology of $LBG$,
called the {\em dual of the loop coproduct},
$
H^*(LBG)\otimes H^*(LBG)\rightarrow H^{*-d}(LBG)
$.
Chataur and the second author showed that the dual of the loop coproduct,
together with the cohomological BV-operator
$\Delta:H^*(LBG)\rightarrow H^{*-1}(LBG)$, make the shifted free loop space cohomology
$\mathbb{H}^*(LBG):=H^{*+d}(LBG)$ into a BV algebra {\it up to} signs.
Over $\mathbb{F}_2$, Hepworth and Lahtinen~\cite{Hepworth-Lahtinen:stringclassifying}
have extended this result to non connected compact Lie group and
more difficult, they showed that this $d$-dimensional homological conformal field theory, in particular this algebra $\mathbb{H}^*(LBG)$, has an unit.
Our first result is to solve the sign issues and to show that indeed,
$\mathbb{H}^*(LBG)$ is a BV algebra (Corollary~\ref{thm:B-V_algebra}).
In fact, we show more generally that the dual of a $d$-homological field theory has a structure
of BV algebra (Theorem~\ref{dualHCFTisBV}).

In~\cite{Lahtinen:higherop}, Lahtinen computes some non-trivial higher operations
in the structure of this $d$-dimensional homological conformal field theory on the cohomology of $BG$ for some compact Lie groups $G$.
In this paper, we compute the most important part of this $d$-dimensional homological conformal field theory, namely
the BV-algebra $\mathbb{H}^*(LBG;\mathbb{F}_p)$ for almost any connected compact Lie group $G$ and any field $\mathbb{F}_p$. According to our knowledge, this BV-algebra $H^*(LBG;\mathbb{F}_p)$
has never been computed on any example.

Our method is completely different from the methods used to compute the BV algebra
$\mathbb{H}_*(LM)$ in the known cases recalled above.
Suppose that the cohomology algebra of $BG$ over $\mathbb{F}_p$,
$H^*(BG;\mathbb{F}_p)$, is a polynomial algebra $\mathbb{F}_p[y_1, ..., y_N]$
(few connected compact Lie groups do not satisfy this hypothesis).
Then the cup product on $H^*(LBG;\mathbb{F}_p)$ was first computed by the first author
in~\cite{Kuri:moduleadjoint}(see~\cite{Kishi-Kono:freetwisted} for a quick calculation).
In his paper~\cite{tamanoi:capproducts} entitled "cap products in String topology",
Tamanoi explains the relations between the cap product and the loop product
on $H_*(LM)$. Dually, in Theorem~\ref{Cup in string classifying}
entitled "cup products in String topology of classifying spaces",
 we give the relations between the cup product on $H^*(LBG)$
and the BV algebra $\mathbb{H}^*(LBG)$.
Knowing the cup product on $H^*(LBG)$, these relations give the dual of the loop coproduct, $m$ on
$\mathbb{H}^*(LBG)$ (Theorem~\ref{dlcop given by formulas}).
But now, since the cohomological BV-operator $\Delta$ (see section~\ref{Operateur Delta})
is a derivation with respect to the cup product, $\Delta$ is easy to compute.
So finally, on $H^*(LBG)$, we have computed at the same time, the cup product and the BV-algebra structure.
This has never be done for the BV algebra $\mathbb{H}_*(LM)$.

If there is no top degree Steenrod operation $\text{Sq}_1$ on $H^*(BG;\mathbb{F}_2)$, if $p$ is odd or $p=0$,
applying Theorem~\ref{dlcop given by formulas},
we give an explicit formula for the dual of the loop coproduct $m$ in Theorem~\ref{thm:BG_explicit_cal}
and we show in Theorem~\ref{iso BV with Hochschild when no steenrod operation} that there is an isomorphism of BV algebras between $\mathbb{H}^*(LBG;\mathbb{F}_p)$
and $HH^*(H_*(G;\mathbb{F}_p),H_*(G;\mathbb{F}_p))$, the Hochschild cohomology of the symmetric
Frobenius algebra $H_*(G;\mathbb{F}_p)$.

The case $p=2$ is more intriguing.
When $p=2$, we don't give in general an explicit formula for the dual of the loop coproduct $m$
(however, see Theorem~\ref{un terme en plus} for a general equation satisfied by $m$).
But for a given compact Lie group $G$, applying Theorem~\ref{dlcop given by formulas},
we are able to give an explicit formula.
As examples, in this paper, we compute the dual of the loop coproduct when $G=SO(3)$
(Theorem~\ref{thm:BSO_3}) or $G=G_2$ (Theorem~\ref{thm:BG_2}).
We show (Theorem~\ref{thm:non BV-iso modulo 2}) that the BV algebras $\mathbb{H}^*(LBSO(3);\mathbb{F}_2)$
and $HH^*(H_*(SO(3);\mathbb{F}_2),H_*(SO(3);\mathbb{F}_2))$, the Hochschild cohomology of the symmetric
Frobenius algebra $H_*(SO(3);\mathbb{F}_2)$, are not isomorphic although the underlying Gerstenhaber
algebras are isomorphic.
Such curious result was observed in~\cite{menichi:stringtopspheres} for the Chas-Sullivan
BV algebras $\mathbb{H}_*(LS^2;\mathbb{F}_2)$.

However, for any connected compact Lie group such that $H^*(BG;\mathbb{F}_p)$, is a polynomial algebra,
we show (Corollary~\ref{structure BV quand pas de sq1} and Theorem~\ref{iso of algebras between free loop space cohomology and Hochschild}) that as graded algebras
$$
\mathbb{H}^*(LBG;\mathbb{F}_p)\cong 
H_*(G;\mathbb{F}_p)\otimes H^*(BG;\mathbb{F}_p)\cong
HH^*(H_*(G;\mathbb{F}_p),H_*(G;\mathbb{F}_p)).
$$
Such isomorphisms of Gerstenhaber algebras should exist (Conjecture~\ref{conjecture iso Hochschild}).

 We give now the plan of the paper:

Section 2: We carefully recall the definition of the loop product and of the loop coproduct insisting on orientation (Theorem~\ref{product in degree d}).
 Theorem~\ref{Cup in string classifying} mentioned above is proved.

Section 3: When $H^*(X)$ is a polynomial algebra, following~\cite{Kuri:moduleadjoint} or~\cite{Kishi-Kono:freetwisted},
we give the cup product on $H^*(LX)$.
Therefore (Theorem~\ref{dlcop given by formulas}) the dual of the loop coproduct is completely given by
Theorems~\ref{product in degree d} and~\ref{Cup in string classifying}.

Section 4 is devoted to the simple case when the characteristic of the field is different from two or when there is no top degree Steenrod operation.

Section 5: The field is $\mathbb{F}_2$. We give some general properties of the dual
of the loop coproduct (Lemma~\ref{lem:trivial loop coproduct}, Theorem~\ref{un terme en plus}). In particular, we show that it has an unit (Theorem~\ref{unit for Dlcop}).
As examples, we compute the dual of the loop coproduct
on $\mathbb{H}^*(LBSO(3);\mathbb{F}_2)$ and on $\mathbb{H}^*(LBG_2;\mathbb{F}_2)$
(Theorems~\ref{thm:BSO_3} and~\ref{thm:BG_2}).
Up to an isomorphism of graded algebras,
$
\mathbb{H}^*(LX;\mathbb{F}_2)
$
is just the tensor product of algebras
$H^*(X;\mathbb{F}_2)\otimes H_{-*}(\Omega X;\mathbb{F}_2)
=\mathbb{F}_2[V]\otimes \Lambda(sV)^\vee$ (Theorem~\ref{iso of algebras between free loop space cohomology and Hochschild}).
As examples, we compute the BV-algebra
$
H^{*+3}(LBSO(3);\mathbb{F}_2)\cong \Lambda(u_{-1},u_{-2})\otimes \mathbb{F}_2[v_2,v_3]
$
(Theorem~\ref{thm: BV-algebra sur algebre commutative libre pour SO3})
and the BV-algebra
$
H^{*+14}(LBG_2;\mathbb{F}_2)\cong \Lambda(u_{-3},u_{-5},u_{-6})\otimes \mathbb{F}_2[v_4,v_6,v_7]
$
(Theorem~\ref{thm: BV-algebra sur algebre commutative libre pour G2}).

Section 6: After studying the formality and the coformality of $BG$,
we compare the associative algebras, the Gerstenhaber algebras,
the BV-algebras
$
\mathbb{H}^*(LBG)
$
and
$
HH^*(H_*(G),H_*(G))
$
under various hypothesis.

Section 7: We solve some sign problems in the results of Chataur and the second
author.
In particular, we correct the definition of integration along the fibre
and the main cotheorem of~\cite{Chataur-Menichi:stringclass}
concerning the prop structure on $H^*(LX)$.

Section 8: Therefore $\mathbb{H}^*(LX)$ is equipped with a graded associative
and graded commutative product $m$.

Section 9: In fact, $\mathbb{H}^*(LX)$ equipped with $m$ and the BV-operator
$\Delta$ is a BV-algebra since the BV identity arises from the lantern relation.

Section 10: This BV identity comes from seven equalities involving Dehn twists
and the prop structure on the mapping class group.

Section 11: We compare different definitions of the BV-operator
$\Delta:H^*(LX)\rightarrow H^{*-1}(LX)$.

Section 12: We compute the Gerstenhaber algebra structure on the Hochschild
cohomology $HH^*(S(V),S(V))$ of a free commutative graded algebra $S(V)$
(Theorem~\ref{thm: Hochschild cohomology exterior algebra}).
In particular, we give the BV-algebra structure on the Hochschild
cohomology $HH^*(\Lambda(V),\Lambda(V))$ of a graded exterior algebra $\Lambda(V)$.

Section 13: In this last section independent of the rest of the paper,
we show that the loop product on $H_*(LBG;\mathbb{F}_p)$ is trivial
if and only if the inclusion of the fibre
$i:\Omega BG\hookrightarrow LBG$ induces a surjective map in cohomology
if and only if $H^*(BG;\mathbb{F}_p)$ is a polynomial algebra
if and only if $BG$ is $\mathbb{F}_p$-formal (when $p$ is odd).

\section{The dual of the Loop coproduct}
In this paper, all the results are stated for simplicity for a connected compact Lie group $G$. But they are also valid for an exotic $p$-compact group.
Indeed, following~\cite{Chataur-Menichi:stringclass}, we only require
that $G$ is a connected topological group (or a pointed loop space) with finite dimensional cohomology
$H^*(G;\mathbb{F}_p)$.
This is the main difference with~\cite{Hepworth-Lahtinen:stringclassifying},
where Hepworth and Lahtinen require the smoothness of $G$.

Let $X$ be a simply-connected space satisfying the condition that 
$H^*(\Omega X; \K)$ is of finite dimension. 
Then there exists an unique integer $d$ such that $H^i(\Omega X; \K) = 0$ for $i > d$
and $H^d(\Omega X; \K) \cong \K$. 
In order to describe our results, we first recall the definitions of the product 
$Dlcop$ on $H^{*+d}(LX; \K)$ and of the loop product on $H_{*-d}(LX; \K)$
defined by Chataur and the second author in~\cite{Chataur-Menichi:stringclass}.

Let $F$ be the pair of pants regarded as a cobordism between one ingoing
circle and two outgoing circles.
The ingoing map $in:S^1\hookrightarrow F$
and outgoing map $out:S^1\coprod S^1\hookrightarrow F$
give the correspondence
$$
\xymatrix@C40pt@R20pt{
LX &  \ar@{->>}[l]_{map(in,X)} map(F,X) \ar@{->>}[r]^{map(out,X)} &  LX\times LX
}
$$
where $map(in,X)$ and $map(out,X)$ are orientable fibrations.
After orienting them, the integration along the fibre induces a map 
in cohomology
$$
map(in,X)^!:H^{*+d}(map(F,X))\rightarrow H^*(LX)
$$
and a map in homology
$$
map(out,X)_!:H_*(LX)^{\otimes 2}\rightarrow H_{*+d}(map(F,X)).
$$
See Section~\ref{review of chataurmenichi} for the definition of the integration along the fibre. 
By definition, the loop product is the composite
\begin{eqnarray*}
H_*(map(in,X))\circ map(out,X)_!:H_{p-d}(LX)\otimes H_{q-d}(LX) \!\!\!\! &\rightarrow& \!\!\!\! H_{p+q-d}(map(F,X)) \\
&\rightarrow& H_{p+q-d}(LX).
\end{eqnarray*}
By definition, the dual of the loop coproduct, $Dlcop$ is the composite
\begin{eqnarray*}
map(in,X)^!\circ H^*(map(out,X)):H^{p+d}(LX)\otimes H^{q+d}(LX) \!\!\!\! &\rightarrow& \!\!\!\! H^{p+q+2d}(map(F,X)) \\
&\rightarrow& H^{p+q+d}(LX).
\end{eqnarray*}
The pair of pants $F$ is the mapping cylinder of
$c\coprod \pi:S^1\coprod (S^1\coprod S^1)\rightarrow S^1\vee S^1$
where $c:S^1\rightarrow S^1\vee S^1$ is the pinch map and
$\pi:S^1\coprod S^1\rightarrow S^1\vee S^1$ is the quotient map.
Therefore the wedge of circles $S^1\vee S^1$ is a strong deformation retract of the pair of pants $F$.
The retract $r:F\buildrel{\thickapprox}\over\twoheadrightarrow S^1\vee S^1$
 corresponds to lower his pants and tuck up his trouser legs at the
same time:

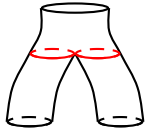
\begin{figure}[ht]
  \centering
  \psset{unit=1ex} 
  \begin{pspicture}(-6,-6)(6,4) 

    \pscurve
    (6,-6)(5.5,-3)(4,0)(3,2)(3,4)
    \pscurve
    (-6,-6)(-5.5,-3)(-4,0)(-3,2)(-3,4)    
    
     \pscurve
    (2,-6)(1.5,-3)(0,0)
      \pscurve
    (-2,-6)(-1.5,-3)(0,0)



    \rput(0,4){%
      \psellipse
      (0,0)(3,.5)
    }

    \rput(2,0){%
      \psellipse[linestyle=dashed,linecolor=red](0,0)(2,.5)
      \psellipticarc[linecolor=red]
      (0,0)(2,.5){180}{360}
    }
    \rput(-2,0){%
      \psellipse[linestyle=dashed,linecolor=red](0,0)(2,.5)
      \psellipticarc[linecolor=red]
      (0,0)(2,.5){180}{360}
    }
     \rput(4,-6){%
      \psellipse[linestyle=dashed]
      (0,0)(2,.5)
      \psellipticarc
      (0,0)(2,.5){180}{360}
    }
    
    \rput(-4,-6){%
      \psellipse[linestyle=dashed]
      (0,0)(2,.5)
      \psellipticarc
      (0,0)(2,.5){180}{360}
    }

  \end{pspicture}  
    \caption{the homotopy between the pairs of pants and the figure eight.}
  \label{fig:1}
\end{figure}
\noindent Thus we have the commutative diagram
$$
\xymatrix@C40pt@R20pt{
LX
&map(F,X)\ar[r]^{map(out,X)}\ar[l]_{map(in,X)}
& LX^{\times 2}\\
& LX\times_X LX\ar[ul]^{Comp}\ar[ur]_{q}\ar[u]_{map(r,X)}^{\thickapprox}
}$$
where $Comp$ is the composition of loops and $q$ is the inclusion.
If $X$ was a closed manifold $M$ of dimension $d$,
$Comp$ and $q$ would be embeddings.
And the Chas-Sullivan loop product is the composite
$$
H_*(Comp)\circ q_!:H_{p+d}(LM)\otimes H_{q+d}(LM)\rightarrow H_{p+q+d}(LM\times_M LM))\rightarrow H_{p+q+d}(LM).
$$
while the dual of the loop coproduct is the composite
$$
Comp^!\circ H^*(q):H^{p-d}(LM)\otimes H^{q-d}(LM)\rightarrow H^{p+q-2d}(LM\times_M LM)\rightarrow H^{p+q-d}(LM).
$$
Therefore although $Comp$ and $q$ are not fibrations, by an abuse of notation,
sometimes, we will say that in the case of string topology of
classifying spaces~\cite{Chataur-Menichi:stringclass}, the loop product on $H_{*-d}(LX)$ is still
$H_*(Comp)\circ q_!$ while $Dlcop$ is $Comp^!\circ H^*(q)$.


The shifted  cohomology $\mathbb{H}^*(LX):=H^{*+d}(LX)$ together with the dual of the loop coproduct $Dlcop$ defined by Chataur and the second author in~\cite{Chataur-Menichi:stringclass}
is a Batalin-Vilkovisky algebra, in particular a graded commutative associative algebra,
only up to signs for two reasons:

-First, the integration along the fibre defined in~\cite{Chataur-Menichi:stringclass} as usually does not
satisfy the usual property with respect to the product.
We have corrected this sign mistake of~\cite{Chataur-Menichi:stringclass} in section~\ref{review of chataurmenichi}.

-Second, as explained in section~\ref{review of chataurmenichi}, this is also due
to the non-triviality of the prop $\text{det}H_1(F,\partial_{out};\mathbb{Z})^{\otimes d}$
(if $d$ is odd).

Nevertheless, we show Theorem~\ref{dualHCFTisBV}.
In particular,
we have that ${\mathbb H}^*(LX)$ equipped with the operator  $\Delta$
induced by the action of the circle on $LX$ (See our definition in section~\ref{Operateur Delta}) 
is a 
Batalin-Vilkovisky algebra with respect to the product $m$ defined by 
$$
m(a\otimes b)=(-1)^{d(d-\vert a\vert)}Dlcop(a\otimes b)
$$
for $a\otimes b \in H^*(LX)\otimes H^*(LX)$; see Corollary~\ref{thm:B-V_algebra} below.

In order to investigate $Dlcop$ more precisely, we need to know how the fibration
$map(in,X)$ is oriented. As explained in~\cite[section 11.5]{Chataur-Menichi:stringclass}, we have to choose
a pointed homotopy equivalence $f:F/\partial_{in}\buildrel{\thickapprox}\over\rightarrow S^1$.
Then the fibre $map_*(F/\partial_{in},X)$ of $map(in,X)$ is oriented by
the composite
$$
\tau\circ H^d(map_*(f,X)):
H^d(map_*(F/\partial_{in},X))\rightarrow H^d(\Omega X)\rightarrow\K.
$$
where $\tau$ is the orientation on $\Omega X$ that we choose.
In this paper, we choose $f$ such that we have the following
homotopy commutative diagram
$$
\xymatrix{
map_*(F/\partial_{in},X)\ar[r]^{incl}
&map(F,X)\\
\Omega X\ar[u]^{map_*(f,X)}_\thickapprox\ar[r]_j
& LX\times_X LX\ar[u]_{map(r,X)}^\thickapprox
}
$$
where $incl$ is the inclusion of the fibre of $map(in,X)$
and $j$ is the map defined by $j(\omega)=(\omega,\omega^{-1})$.
\begin{thm}\label{product in degree d}
Let $i:\Omega X\hookrightarrow LX$ be the inclusion
of pointed loops into free loops.
Let $S$ be the antipode of the Hopf algebra $H^*(\Omega X)$.
Let $\tau:H^d(\Omega X)\rightarrow\K$ be the chosen orientation
on $\Omega X$.
Let $a\in H^p(LX)$ and $b\in H^q(LX)$ such that $p+q=d$.
Then with the above choice of pointed homotopy equivalence
$f:F/\partial_{in}\buildrel{\thickapprox}\over\rightarrow S^1$,
$$
m(a\otimes b)=(-1)^{d(d-p)}\tau\left(H^p(i)(a)\cup S\circ H^q(i)(b)\right)1_{H^*(LX)}.
$$
\end{thm}
\begin{proof}
Let $F\buildrel{incl}\over\hookrightarrow E
\buildrel{p}\over\twoheadrightarrow B$ be an oriented fibration
with orientation $\tau:H^d(F)\rightarrow\K$.
By definition or by naturality with respect to pull-backs,
the integration along the fibre $p^!$ is in degree $d$ the composite
$$
H^d(E)\buildrel{H^d(incl)}\over\rightarrow H^d(F)\buildrel{\tau}\over\rightarrow
\K\buildrel{\eta}\over\rightarrow H^0(B)
$$
where $\eta$ is the unit of $H^*(B)$.
Therefore Dlcop is given by the commutative diagram
$$
\xymatrix{
& H^d(LX\times LX)\ar[ld]_{H^dmap(out,X)}\ar[d]^{H^d(q)}\ar[rd]^{H^d(i\times i)}\\
H^d(map(F,X))\ar[r]_{H^dmap(r,X)}\ar[d]^{H^d(incl)}\ar@/_4pc/[dd]_{map(in,X)^!}
& H^d(LX\times_X LX)\ar[d]^{H^d(j)}\ar[r]^{H^d(incl)}
& H^d(\Omega X\times\Omega X)\ar[d]^{H^d(Id\times Inv)}\\
H^d(map_*(F/\partial_{in}))\ar[r]_{H^dmap_*(f,X)}
& H^d(\Omega X)\ar[d]_\tau
& H^d(\Omega X\times\Omega X)\ar[l]^{H^d(\Delta)}\\
H^0(LX)
&\K\ar[l]_\eta
}
$$
where $incl:\Omega X\times \Omega X\hookrightarrow LX\times_X LX$ is
the inclusion and $Inv:\Omega X\rightarrow \Omega X$ maps a loop $\omega$ to its inverse
$\omega^{-1}$.
Therefore
$$
Dlcop(a\otimes b)=\tau\left(H^p(i)(a)\cup S\circ H^q(i)(b)\right)1_{H^*(LX)}.
$$
\end{proof}

We define a bracket $\{ \ , \}$ on $H^*(LX)$ with the product $m$ and the Batalin-Vilkovisky operator $\Delta : H^*(LX) \to H^{*-1}(LX)$ by 
$$
\{a, b\} = (-1)^{\vert a \vert}\Delta (m(a\otimes b)) - (-1)^{\vert a \vert}m(\Delta (a)\otimes b) - m(a \otimes \Delta (b)) 
$$
for $a , b$ in  $H^*(LX)$. By Theorem~\ref{thm:B-V_algebra}, this bracket is exactly
a Lie bracket. 
The following theorem is the analogue for string topology of
classifying spaces~\cite{Chataur-Menichi:stringclass} of the theorems of Tamanoi in~\cite{tamanoi:capproducts} for Chas-Sullivan string topology~\cite{Chas-Sullivan:stringtop}.
This analogy is quite surprising and complete. For our calculations, in the rest of the
paper, we use only parts (1) and (2) of this theorem.
Let $ev: LX\twoheadrightarrow X$ be the evaluation map defined by
$ev(\gamma)=\gamma(0)$ for $\gamma\in LX$.
\begin{thm}\label{Cup in string classifying}(Cup products in string topology of classifying spaces)
Let $X$ be a simply-connected space such that
$H_*(\Omega X;\K)$ is finite dimensional.
Let $P$, $Q\in H^*(X)$ and $a$ and $b\in H^*(LX)$.

(1) (Compare with~\cite[Theorem A (1.2)]{tamanoi:capproducts})
The dual of the loop coproduct $m:\mathbb{H}^*(LX)\otimes\mathbb{H}^*(LX)\rightarrow\mathbb{H}^*(LX)$ is
a morphism of left $H^*(X)\otimes H^*(X)$-modules:
$$m(ev^*(P)\cup a\otimes ev^*(Q)\cup b)
=(-1)^{(\vert a\vert-d)\vert Q\vert}ev^*(P)\cup ev^*(Q)\cup m(a\otimes b).$$

(2) (Compare with~\cite[Theorem A (1.3)]{tamanoi:capproducts})
The cup product with $\Delta(ev^*(P))$ is a derivation with
respect to the algebra ($\mathbb{H}^*(LX),m$):
\begin{eqnarray*}
&&\Delta(ev^*(P))\cup m(a\otimes b) \\
&=& m(\Delta(ev^*(P))\cup  a\otimes b)
+(-1)^{(\vert P\vert-1)(\vert a\vert-d)}m(a\otimes \Delta(ev^*(P))\cup  b). 
\end{eqnarray*}

(3)  (Compare with~\cite[Theorem A(1.4) ]{tamanoi:capproducts})
The cup product with $\Delta(ev^*(P))$ is a derivation with
respect to the bracket
$$
\Delta(ev^*(P))\cup \{a,b\}=
 \{\Delta(ev^*(P))\cup  a, b\}+ (-1)^{(\vert P \vert -1)(\vert a \vert-d -1)}\{a,\Delta(ev^*(P))\cup  b\}. 
$$

(4) (Compare with~\cite[formula p. 16, line -3]{tamanoi:capproducts}) The following formula gives a relation for the cup product of $ev^*(P)$ with the bracket
$$
\{ev^*(P)\cup a, b\}=
ev^*(P)\cup \{a,b\}+
(-1)^{\vert P\vert(\vert a\vert -d-1)}
m( a\otimes \Delta(ev^*(P))\cup b)
$$

(5) (Compare with~\cite[Theorem B]{tamanoi:capproducts})
The direct sum
$H^*(X)\oplus\mathbb{H}^*(LX)$ is a Batalin-Vilkovisky algebra
where the dual of the loop coproduct $m$, the bracket and the $\Delta$ operator are
extended by
$m(P\otimes a):=ev^*(P)\cup a$, $m(P\otimes Q)=P\cup Q$,
$\{P,a\}=(-1)^{\vert P\vert}\Delta(ev^*(P))\cup a$,
$\{P,Q\}=0$ and $\Delta(P)=0$.

(6) (Compare with~\cite[Theorem C]{tamanoi:capproducts})
Suppose that the algebra ($\mathbb{H}^*(LX),m$) has an unit $\mathbb{I}$.
Let $s^!:H^*(X)\rightarrow H^{*+d}(LX)$, be the map 
mapping $P$ to $ev^*(P)\cup \mathbb{I}$. Then $s^!$ is a morphism of Batalin-Vilkovisky algebras
with respect to the trivial BV-operator on $H^*(X)$ and
$$
ev^*(P)\cup a=m(s^!(P)\otimes a)\quad\text{and}\quad
(-1)^{\vert P\vert}\Delta(ev^*(P))\cup a=\{s^!(P), a\}.
$$

(7) Let $r\geq 0$. Let $P_1$, \dots, $P_r$ be $r$ elements of $H^*(LX)$.
Denote by $X_i:=\Delta(ev^*(P_i))$. Then
\begin{align*}
m(ev^*(P)\cup a\otimes ev^*(Q)\cup X_1\cup\cdots\cup X_r\cup b)=(-1)^{(\vert a\vert-d)(\vert Q\vert+\vert X_1\vert+\dots+\vert X_r\vert)}\times\\
\sum_{0\leq j_1,\dots, j_r\leq 1}\pm
ev^*(P)\cup ev^*(Q)
\cup X_1^{1-j_1}\cup\cdots\cup X_r^{1-j_r}
\cup m(X_1^{j_1}\cup\cdots\cup X_r^{j_r}\cup a\otimes b)
\end{align*}
where $\pm$ is the sign $(-1)^{j_1+\cdots+j_r+\sum_{k=1}^r (1-j_k)\vert X_k\vert (j_1\vert X_1\vert +\cdots+j_{k-1}\vert X_{k-1}\vert)}$.
\end{thm}
To prove parts (1) and (2), it is shorter to use the following Lemma.
This Lemma is just the cohomological version
of~\cite[Theorem 8.2]{Chas-Sullivan:stringtop} when we replace the correspondence
$$
LM\times LM\buildrel{q}\over\hookleftarrow LM\times_M LM
\buildrel{Comp}\over\rightarrow LM
$$
by its opposite
$$
LX\buildrel{Comp}\over\leftarrow LX\times_X LX
\buildrel{q}\over\hookrightarrow  LX\times LX.
$$
Similarly, it would have been shorter for Tamanoi to prove parts (1.2) and (1.3) of
~\cite[Theorem A]{tamanoi:capproducts} using ~\cite[Theorem 8.2]{Chas-Sullivan:stringtop}.

\begin{lem}\label{relations cup and loop product}
Let $a=\sum a_1\otimes a_2\in H^{*}(LX\times LX)$ and $A\in H^{*}(LX)$
such that $H^*(Comp)(A)=H^*(q)(a)$.
Then for any $z\in H^*(LX\times LX)$,
$$
A\cup m(z)=\sum (-1)^{d\vert a_2\vert}m(a_1\otimes a_2\cup z).
$$
\end{lem}
\begin{proof}
Since the integration along the fibre $Comp^!$ is exactly with signs,
a morphism of left $H^*(LX)$-modules (See our definition of integration along the fibre
in cohomology in section~\ref{review of chataurmenichi})
$$
Comp^!(H^*(Comp)(A)\cup y)=(-1)^{d\vert A\vert}A\cup Comp^!(y).
$$
Since $H^*(q)$ is a morphism of algebras,
\begin{multline*}
(-1)^{d\vert A\vert}Dlcop(a\cup z)=
(-1)^{d\vert A\vert}Comp^!\circ H^*(q)(a\cup z)\\
=(-1)^{d\vert A\vert}Comp^!(H^*(Comp)(A)\cup H^*(q)(z))
=A\cup Comp^!\circ H^*(q)(z)
=A\cup Dlcop(z).
\end{multline*}
By linearity, we can suppose that $z=z_1\otimes z_2$.
Then the previous equation is
$$
A\cup (-1)^{d(\vert z_1\vert-d)}m(z_1\otimes z_2)
=\sum (-1)^{d(\vert a_1\vert+\vert a_2\vert)}
(-1)^{d(\vert a_1\vert+\vert z_1\vert-d)}
m(a_1\otimes a_2\cup z_1\otimes z_2).
$$
\end{proof}
\begin{proof}[Proof of Theorem~\ref{Cup in string classifying}]
(1) We have the commutative diagram
$$
\xymatrix@C25pt@R20pt{
LX \ar[dr]_{ev}
&LX\times_X LX\ar[l]_(0.6){Comp}\ar[r]^q\ar[d]
&LX\times LX\ar[d]^{ev\times ev}\\
&X\ar[r]_{\Delta}
&X\times X\\
}
$$
Therefore by applying Lemma~\ref{relations cup and loop product}
to $a:=H^*(ev\times ev)(P\otimes Q)$ and $A:=H^*(\Delta\circ
ev)(P\otimes Q)$,
we obtain
$$
H^*(ev)(P)\cup H^*(ev)(Q)\cup m( a\otimes b)=
(-1)^{d\vert Q\vert} m(H^*(ev)(P)\otimes H^*(ev)(Q)\cup a\otimes b).
$$

(2) By~\cite[Proof of Theorem 4.2 (4.5)]{tamanoi:capproducts}
$$
Comp^*(\Delta(ev^*(P)))=q^*(\Delta(ev^*(P))\times 1+1\times \Delta(ev^*(P))).
$$
So we can apply Lemma~\ref{relations cup and loop product}
to $a:=\Delta(ev^*(P))\times 1+1\times \Delta(ev^*(P))$ and $A:=\Delta(ev^*(P))$,
we obtain
\begin{multline*}
\Delta(ev^*(P))\cup(m(a\otimes b)=\\
m\left((\Delta(ev^*(P))\otimes 1\right)
+(-1)^{d(\vert P\vert-1)}m\left(1\otimes \Delta(ev^*(P)))\cup (a\otimes b)\right).
\end{multline*}

(3) By using the formula (2), the same argument 
as in~\cite[Proof of Theorem 4.5]{tamanoi:capproducts} deduces the derivation formula on the bracket. 

(4) Again, the arguments are identical as those given by Tamanoi: see~\cite[end of proof of Theorem 4.7]{tamanoi:capproducts}.

(5) As explained in~\cite[proof of Theorem 4.7]{tamanoi:capproducts} by Tamanoi,
(2), (3) and (4) are equivalent to the Poisson and Jacobi identities in the
Gerstenhaber algebra $H^*(X)\oplus\mathbb{H}^*(LX)$.
By definition of the bracket, this Gerstenhaber algebra is a Batalin-Vilkovisky algebra:
see~\cite[proof of Theorem 4.8]{tamanoi:capproducts}.

(6) Since $H^{*+d}(LX;\mathbb{F}_2)$ is a $H^*(X)$-algebra
(formula (1) of Theorem~\ref{Cup in string classifying}),
the map $s^!:H^*(X)\rightarrow H^{*+d}(LX)$, $P\mapsto ev^*(P)\cup \mathbb{I}$,
is a morphism of unital commutative graded algebras
(we denote this map $s^!$ because this map should coincide with some Gysin
map of the trivial section $s:X\hookrightarrow LX$~\cite{Chataur-Menichi:stringclass}).

Since the cup product with $\Delta(ev^*(P))$ is a derivation with respect to the dual of the loop coproduct,
$\Delta(ev^*(P))\cup \mathbb{I}=0$. Since $\mathbb{H}^{*}(LX)$ is a Batalin-Vilkovisky
algebra, $\Delta(\mathbb{I})=0$. Therefore, since $\Delta$ is a derivation with respect to the cup product,
$$\Delta(s^!(P))=\Delta(ev^*(P))\cup \mathbb{I}+(-1)^{\vert P\vert} ev^*(P)\cup \Delta(\mathbb{I})=0+0.$$
Now we can conclude using the same arguments as in~\cite[proof of Theorem 5.1]{tamanoi:capproducts}.

(7) The case $r=0$ is just (1). Now, by induction on $r$,

\begin{align*}
m(ev^*(P)\cup a\otimes ev^*(Q)\cup X_1\cup\cdots\cup X_{r-1}\cup (X_r\cup b))=(-1)^{(\vert a\vert-d)(\vert Q\vert+\vert X_1\vert+\dots+\vert X_{r-1}\vert)}\times\\
\sum_{0\leq j_1,\dots, j_{r-1}\leq 1}\pm
ev^*(P)\cup ev^*(Q)
\cup X_1^{1-j_1}\cup\cdots\cup X_{r-1}^{1-j_{r-1}}
\cup m(X_1^{j_1}\cup\cdots\cup X_{r-1}^{j_{r-1}}\cup a\otimes X_r\cup b)
\end{align*}

But by (2),
%
\begin{align*}
m(X_1^{j_1}\cup\cdots\cup X_{r-1}^{j_{r-1}}\cup a\otimes X_r\cup b)=\\
\sum_{j_r=0}^1 (-1)^{\vert X_r\vert(\vert a\vert -d)+j_r+ (1-j_r)\vert X_r\vert\sum_{l=1}^{r-1}j_{l}\vert X_{l}\vert}
X_{r}^{1-j_{r}} m(X_1^{j_1}\cup\cdots\cup X_{r}^{j_{r}}\cup a\otimes b)
\end{align*}
\end{proof}
\begin{rem}\label{dlcop when generated by Delta}
Suppose that the algebra $H^*(LX)$ is generated by $H^*(X)$ and $\Delta(H^*(X))$.
Then by formula (7) of Theorem~\ref{Cup in string classifying} in the case $b=1$,
we see that the dual of the loop coproduct $m$ is completely given by the cup product, by the
$\Delta$ operator and by its restriction on $\mathbb{H}^*(LX)\otimes 1$.
In the following section, we show that this is the case when
$H^*(X)$ is a polynomial (see remark~\ref{dlcop for polynomial}).
\end{rem}
\section{The cup product on free loops and the main theorem}\label{cup product main theorem}
Let $X$ be a simply-connected space with polynomial cohomology: $H^*(X)$ is a polynomial
algebra $\K[y_1, ..., y_N]$.
The cup product on the free loop space cohomology $H^*(LX;\K)$ was first computed
by the first author in~\cite[Theorem 1.6]{Kuri:moduleadjoint}.
We now explain how to recover simply this computation following~\cite[p. 648]{Kishi-Kono:freetwisted}.

By Borel theorem~\cite[Chapter VII. Corollary 2.8(2)]{Mimura-Toda:topliegroups}
(which can be easily proved using the Eilenberg-Moore spectral sequence associated to
the path fibration $\Omega X\hookrightarrow PX\twoheadrightarrow X$ since
$E_2^{*,*}\cong\Lambda(\sigma(y_1),\dots,\sigma(y_N))$),
$$
H^*(\Omega X;\K)=\Delta(\sigma(y_1),\dots,\sigma(y_N))
$$
where $\Delta \sigma(y_i)$ denotes an algebra with simple system of generators $\sigma(y_i)$~\cite[Definition p. 367]{Mimura-Toda:topliegroups}.
If $ch(\K) \neq 2$, $\Delta \sigma(y_i)$ is just the exterior algebra
$\Lambda \sigma(y_i)$.

Let $\Delta:H^*(LX)\rightarrow H^{*-1}(LX)$ be the operator induced
by the action of the circle on $LX$ (See section~\ref{Operateur Delta}).
Let ${\mathcal D}:=\Delta \circ ev^*$ denotes the module derivation of the first author
in~\cite{Kuri:moduleadjoint}.
Since $\Delta$ is a derivation with respect to the cup product,
$\mathcal{D}$ is a $(ev^*,ev^*)$-derivation~\cite[Proposition 3.3]{Kuri:moduleadjoint}.
Since $\Delta$ and $H^*(ev)$ commutes with the Steenrod operations,
$\mathcal{D}$ also~\cite[Proposition 3.3]{Kuri:moduleadjoint}.
Since the composite $i^*\circ \mathcal{D}$ is the suspension homomorphism $\sigma$~\cite[Proposition 2(1)]{Kishi-Kono:freetwisted}, $i^*$ is surjective and so by
Leray-Hirsch theorem,
$$
H^*(LX;\K)=H^*(X)\otimes \Delta \left(\mathcal{D}(y_1),\dots, \mathcal{D}(y_N)\right)
$$
as $H^*(X)$-algebra. Modulo $2$, it follows from above that
$H^*(LX;{\mathbb Z}/2))$ is the polynomial algebra $${\mathbb Z}/2[ev^*(y_i),\mathcal{D} y_i]$$
quotiented by the relations $$(\mathcal{D}y_i)^2=\mathcal{D}(\text{Sq}^{\vert y_i\vert-1}y_i).$$
In particular, we have $\Delta(ev^*(y_i))=\mathcal{D} y_i$
and $\Delta(\mathcal{D} y_i)=0$ since $\Delta\circ\Delta=0$.
Therefore, we know the cup product and the
$\Delta$ operator on $H^*(LX;\K)$. The following theorem claims that we also know
the dual of the loop coproduct.

\begin{thm}\label{dlcop given by formulas}
Let $X$ be a simply-connected space such that
$H^*(X;\K)$ is the polynomial algebra $\K[y_1,\dots,y_N]$.
Denote again by $y_i$, the element of $H^*(LX)$, $ev^*(y_i)$,
and by $x_i$, $\Delta\circ ev^*(y_i)$.
With respect to the cup product, as algebras
$$
H^*(LX)=\K[y_1,\dots,y_N]\otimes\Delta(x_1,\dots,x_N).
$$
Let $d$ be the degree of $x_1\dots x_N$.
Then the dual of the loop coproduct
$$
m:H^i(LX)\otimes H^j(LX)\rightarrow H^{i+j-d}(LX)
$$
is given inductively (see remark~\ref{dlcop for polynomial}) by the following four formulas

(1) For any $a$ and $b\in H^*(LX)$, $\forall 1\leq i\leq N$,
$$
m(a\otimes x_ib)=(-1)^{\vert x_i\vert(\vert a\vert-d)}x_i m(a\otimes b)-(-1)^{d\vert x_i\vert} m(ax_i\otimes b)
$$

(2) Let $\{i_1,\dots,i_l\}$ and $\{j_1,\dots,j_m\}$ be two disjoint subsets
of $\{1,\dots, N\}$ such that $\{i_1,\dots, i_l\}\cup\{j_1,\dots, j_m\}=\{1,\dots, N\}$.
If we orient $\tau:H^d(\Omega X)\buildrel{\cong}\over\rightarrow\K$
by $\tau\circ H^*(i)(x_1\dots x_N)=1$
then
$$m(x_{i_1}\dots x_{i_l}\otimes x_{j_1}\dots x_{j_m})=(-1)^{Nm+m}\varepsilon$$
where $\varepsilon$ is the signature of the permutation
$
\left(
\begin{array}{lr}
1\dots &l+m\\
i_1\dots i_lj_1\dots &j_m
\end{array}
\right)
$.

(3) Let $\{i_1,\dots,i_l\}$ and $\{j_1,\dots,j_m\}$ be two disjoint subsets
of $\{1,\dots, N\}$ such that $\{i_1,\dots, i_l\}\cup\{j_1,\dots, j_m\}\neq \{1,\dots, N\}$.
Then
$$m(x_{i_1}\dots x_{i_l}\otimes x_{j_1}\dots x_{j_m})=0.$$



(4) $m$ is a morphism of left $H^*(X)\otimes H^*(X)$-modules:
$\forall P$, $Q\in H^*(X)$, $\forall a$ and $b\in H^*(LX)$,
$m((-1)^{\vert Q\vert(\vert a\vert-d)}Pa\otimes Qb)=PQ m(a\otimes b)$.
\end{thm}
\begin{proof}
Note that if $y_i$ is of odd degree then $2=0$ in $\K$.
(1) and (4) are particular cases of (1) and (2) of Theorem~\ref{Cup in
  string classifying}.
Since $x_{i_1}\dots x_{i_l}\otimes x_{j_1}\dots x_{j_m}$ is of degree less than $d$, for degree
reasons, we have (3).

(2) Since $i^*(x_i)=i^*\circ \Delta\circ ev^*(y_i)$ is the suspension of $y_i$, denoted $\sigma(y_i)$,   by Theorem~\ref{product in degree d},
$$m(x_{i_1}\dots x_{i_l}\otimes x_{j_1}\dots x_{j_m})
=(-1)^{Nm}\tau\left(\sigma(y_{i_1})\dots \sigma(y_{i_l})\cup S(\sigma(y_{j_1})\dots \sigma(y_{j_m})\right)1.$$
Since $\sigma(y_i)$ is a primitive element, $S(\sigma(y_i))=-\sigma(y_i)$.
Since also the antipode $S:H^*(\Omega X)\rightarrow H^*(\Omega X)$
is a morphism of commutative graded algebras,
$$m(x_{i_1}\dots x_{i_l}\otimes x_{j_1}\dots x_{j_m})=(-1)^{Nm+m}\varepsilon
\tau(\sigma(y_1)\dots \sigma(y_N)).$$
\end{proof}
\begin{rem}\label{dlcop for polynomial}
We explain now why the four formulas of Theorem~\ref{dlcop given by formulas}
determine inductively the dual of the loop coproduct $m$.
For $P\in H^*(X)$ and $\{i_1,\dots,i_l\}$ a strict subset of $\{1,\dots, N\}$, by (2), (3) and (4),
$m(Px_{i_1}\dots x_{i_l}\otimes 1)=0$ 
and $m(Px_{1}\dots x_{N}\otimes 1)=P$.
Therefore, we know the restriction of $m$ on $\mathbb{H}^*(LX)\otimes 1$.
Since the algebra $H^*(LX)$ is generated by
$H^*(X)$ and $\Delta(H^*(X))$, $m$ is now given inductively by (1) and (4) (see remark~\ref{dlcop when generated by Delta}).

The restriction of $m: \mathbb{H}^*(LX)\otimes 1\rightarrow H^*(X)$ looks similar
to the intersection morphism $i_!:\mathbb{H}_*(LM)\rightarrow H_*(\Omega M)$ for manifold given by
the loop product with the constant pointed loop.
\end{rem}
\section{Case $p$ odd or no $Sq_1$}
Let $Sq_1$ be  the operator $H^*(BG; {\mathbb Z}/2) \to H^*(BG; {\mathbb Z}/2)$ is defined by 
$Sq_1(x)=Sq^{\deg x-1}x$ for $x \in H^*(BG; {\mathbb Z}/2)$. 

Suppose that $H^*(BG)$ is a polynomial algebra, say ${\mathbb K}[V]$ and that 

\medskip
\noindent
(H) :  \hspace{2cm} $Sq_1\equiv 0$ on $H^*(BG)$ if $p=2$ or  $p$  is odd or $p=0$
(Since $Sq(xy)=x^2Sq_1(y)+Sq_1(x)y^2$, it suffices to check that $Sq_1\equiv 0$ on $V$).
\medskip

\noindent 
Then it follows that 
$$
H^*(LBG; {\mathbb Z}/p)\cong \wedge (sV) \otimes \K[V]
$$ 
as an algebra; see  \cite[Remark 3.4] {Kono-Kuri:modulesplitting} for example.  
We moreover have  

\begin{thm}
\label{thm:BG_explicit_cal}
Under the hypothesis  (H), 
an explicit form of the dual of the loop coproduct 
$
m: H^*(LBG; {\mathbb Z}/p)\otimes H^*(LBG; {\mathbb Z}/p) \to H^{*-\dim G}(LBG; {\mathbb Z}/p)
$
is given by 
$$
m(sv_{i_1}\cdots sv_{i_l}a\otimes sv_{j_1}\cdots sv_{j_m}b)
= (-1)^{\e' + \e + m + u + lu+Nm} sv_{k_1}\cdots sv_{k_u}ab
$$ 
if $\{i_1, ...i_l\}\cup \{j_1, ..., j_m\}=\{1, ..., N\}$ and 
$m(sv_{i_1}\cdots sv_{i_l}a\otimes sv_{j_1}\cdots sv_{j_m}b)=0$ otherwise, where 
 $\{i_1, ...i_l\}\cap \{j_1, ..., j_m\}=\{k_1, .., k_u\}$, $a, b \in H^*(BG)$, 
 $$
(-1)^\e= 
\text{sgn}\left(
\begin{array}{ccc}
j_1  ....  \ \  \ ....\  \  \ .... \  \  \ ...  j_m   \\
k_1  ... k_u   j_1 ... \widehat{k_1} . ..   \widehat{k_u} ...   j_m   
\end{array}
\right) 
 \text{and} \ 
(-1)^{\e'}= 
\text{sgn}\left(
\begin{array}{ccc}
i_1 ... i_l j_1 ... \widehat{k_1} . ..   \widehat{k_u} ...   j_m\\
1  ....  \ \  \ ....\  \  \ .... \  \  \ ...  N   \\
\end{array}
\right).
 $$
\end{thm}
Over $\mathbb{R}$, \cite[17.23]{BGNX:Stringstackspublie} have a similar formula
(surprisingly without any signs) for their dual hidden loop product
on $H^*([G/G])$.

\noindent
{\it Proof of Theorem \ref{thm:BG_explicit_cal}.} By (4) of Theorem~\ref{dlcop given by formulas}
to prove Theorem \ref{thm:BG_explicit_cal}, 
it suffices to show that the formula for the element $x_{i_1}\cdots x_{i_l}\otimes x_{j_1}\cdots x_{j_m}$, namely in the case where 
$a=b=1$. 

Since $x_{k_1}^2=0$, $
m(x_{i_1}\cdots x_{i_l}x_{k_1}\otimes x_{j_1}\cdots\widehat{x_{k_1}}\cdots x_{j_m})=0$.
So by (1) of Theorem~\ref{dlcop given by formulas},
\begin{eqnarray*}
&&m(x_{i_1}\cdots x_{i_l}\otimes x_{j_1}\cdots x_{j_m}) \\
&=&(-1)^{\vert x_{k_1}\vert(\vert x_{i_1}\cdots x_{i_l} x_{j_1}\cdots\widehat{x_{k_1}}\vert-d)}
x_{k_1}m(x_{i_1}\cdots x_{i_l}\otimes x_{j_1}\cdots\widehat{x_{k_1}}\cdots x_{j_m}).
\end{eqnarray*}
By induction on $u$,
\begin{eqnarray*}
&&m(x_{i_1}\cdots x_{i_l}\otimes x_{j_1}\cdots x_{j_m}) \\ 
&=&(-1)^{u(l-d)+\varepsilon}
x_{k_1}\dots x_{k_u}m(x_{i_1}\cdots x_{i_l}\otimes x_{j_1}\cdots\widehat{x_{k_1}}\cdots\widehat{x_{k_u}}\cdots x_{j_m}).
\end{eqnarray*}
By (2) and (3) of Theorem~\ref{dlcop given by formulas},
\begin{eqnarray*}
&&m(x_{i_1}\cdots x_{i_l}\otimes x_{j_1}\cdots\widehat{x_{k_1}}\cdots\widehat{x_{k_u}}\cdots x_{j_m}) \\ 
&=&
\begin{cases}
(-1)^{N(m-u)+m-u+\varepsilon'}&\text{If $\{i_1,\dots, i_l\}\cup\{j_1,\dots, j_m\}=\{1,\dots, N\}$},\\
0& \text{otherwise}.
\end{cases}
\end{eqnarray*}
Here $\widehat{x}$ means that the element $x$ disappears from the presentation. 
\hfill\qed

\begin{cor} \label{cor:unital p odd}
Under the hypothesis H, the graded associative commutative algebra 
$({\mathbb H}^*(LBG),m)$ of Corollary \ref{cor:shifted_alg} is unital. 
\end{cor}

\begin{proof}
We see that $x_1\cdots x_N$ is the unit. Theorem \ref{thm:BG_explicit_cal} yield that
\begin{multline*}
m(x_1\cdots x_N\otimes x_{j_1}\cdots x_{j_m}b)=\\
\text{sgn}\left(
\begin{array}{ccc}
j_1 ......  j_m   \\
j_1 ......  j_m   
\end{array}
\right)  
\text{sgn}\left(
\begin{array}{ccc}
1  ....   N    \\
1  ....   N   \\
\end{array}
\right)
(-1)^{m+m+mN+Nm}x_{j_1}\cdots x_{j_m}b.
\end{multline*}
\begin{multline*}
m(ax_{i_1}\cdots x_{i_l}\otimes x_1\cdots x_N)=
\text{sgn}\left(
\begin{array}{ccc}
1  ....  \ \  \ ....\  \  \ .... \  \  \ ...  N   \\
i_1  ... i_l   1 ... \widehat{i_1} . ..   \widehat{i_l} ...   N   
\end{array}
\right) \\
\text{sgn}\left(
\begin{array}{ccc}
i_1  ... i_l   1 ... \widehat{i_1} . ..   \widehat{i_l} ...   N  \\
1  ....  \ \  \ ....\  \  \ .... \  \  \ ...  N   \\
\end{array}
\right)
(-1)^{N+l+l^2+N^2}ax_{i_1}\cdots x_{i_l}.
\end{multline*}
\end{proof}

\begin{thm}\label{structure BV quand pas de sq1}
Under the hypothesis $(H)$,
${\mathbb H}^*(LBG) =H^{*+\text{dim }G}(LBG;\K)$
is isomorphic as BV algebras to
the tensor product of algebras
$$
H^*(BG;\K)\otimes H_{-*}(G;\K)\cong\K[V]\otimes\wedge (sV)^\vee
$$

equipped with the BV-operator $\Delta$ given by
$\Delta(x_i^\vee\wedge x_j^\vee) = \Delta(y_i y_j)= \Delta(x_j^\vee) = \Delta(y_i)=0$ for any $i, j$ and 

$$
\Delta(y_i\otimes x_j^\vee) = 
\left\{
\begin{array}{ll}
0 & \text{if} \ i\neq j,  \\
1  & \text{if} \  i= j. 
\end{array}
\right.
$$ 
\end{thm}
\begin{proof}
Since $H^*(G)$ is the Hopf algebra $\Lambda x_i$ with $x_i=\sigma(y_i)$ primitive, its dual
is the Hopf algebra $\Lambda x_i^\vee$.
By Corollary \ref{cor:shifted_alg} and Corollary \ref{cor:unital p odd}, we see that the shifted cohomology
${\mathbb H}^*(LBG)$ is a graded commutative algebra with unit $x_1\dots x_N$. 
This enables us to define a morphism of algebras $\Theta$
from  $$H^*(BG;\K)\otimes H_{-*}(G;\K)
=\K[y_1,\cdots,y_n]\otimes \Lambda(x_1^\vee,\cdots,x_N^\vee)$$
to $${\mathbb H}^*(LBG)=\K[y_1,\cdots,y_n]\otimes \Lambda(x_1,\cdots,x_N)$$ by
$$
\Theta(1\otimes x_{j}^\vee)=(-1)^{j-1}1\otimes (x_{1}\wedge\dots\wedge\widehat{x_{j}}\wedge\dots\wedge x_{N} )  \ \ \text{and} \  \ 
\Theta(a\otimes 1)=a\otimes (x_1\wedge\dots\wedge x_N)
$$
for any $a$ in  $\K[V]$.
By induction on $p$, using Theorem \ref{thm:BG_explicit_cal}, we have that
$$
\Theta(a\otimes (x_{j_1}^\vee\wedge\dots \wedge x_{j_p}^\vee))
=\pm a\otimes (x_{1}\wedge\dots\wedge\widehat{x_{j_{1}}}\wedge\dots\wedge\widehat{x_{j_p}}
\wedge\dots \wedge x_{N}) 
$$for any $a \in \K[V]$.
Therefore the map $\Theta$ is an isomorphism.

The isomorphism $\Theta$ sends $1\otimes \Lambda(x_1^\vee,\cdots,x_N^\vee)$ on
$1\otimes \Lambda(x_1,\cdots,x_N)$
and $\K[y_1,\cdots,y_N]\otimes 1$ on $\K[y_1,\cdots,y_N]\otimes x_1\cdots x_N$.
Since $\Delta$ is null on $1\otimes \Lambda(x_1,\cdots,x_N)$ and $\K[y_1,\cdots,y_N]\otimes x_1\cdots x_N$, $\Delta$ is null on $1\otimes \Lambda(x_1^\vee,\cdots,x_N^\vee)$ and $\K[y_1,\cdots,y_N]\otimes 1$:
we have the first equalities. Moreover, we see that 
$\Theta (y_i\otimes x_j^\vee)=(-1)^{j-1}y_i x_1\wedge \cdots \wedge \widehat{x_j}\wedge \cdots \wedge x_N$ 
and hence $\Delta \Theta  (y_i\otimes x_j^\vee)=0$ if $i\neq j$. 
The equalities $\Delta((-1)^{i-1}y_i x_1\wedge \cdots \wedge \widehat{x_i}\wedge 
\cdots \wedge x_N) = 
x_1\wedge  \cdots \wedge x_N= \Theta (1)$ enable us to obtain the second formula. 
\end{proof}

\section{mod 2 case}

In the case where the operation $Sq_1$ is non-trivial on $H^*(BG; {\mathbb Z}/2)$, the loop coproduct structure 
on $H^*(LBG; {\mathbb Z}/2)$ is more complicated in general. For example, we compute the dual to the loop coproduct 
on $H^*(LBG_2; {\mathbb Z}/2)$, where $G_2$ is the simply-connected compact exceptional Lie group of rank $2$. 
Recall that 
\begin{eqnarray*}
H^*(LBG_2; {\mathbb Z}/2)&\cong& \Delta(x_3, x_5, x_6)\otimes {\mathbb Z}/2[y_4, y_6, y_7] \\
&\cong & {\mathbb Z}/2[x_3, x_5]\otimes {\mathbb Z}/2[y_4, y_6, y_7] \left/ \right.
\left(
\begin{array}{cc}
  x_3^4+x_5y_7 + x_3^2 y_6   \\
  x_5^2 + x_3y_7 + x_3^2y_4   
\end{array}
\right)
\end{eqnarray*}
as algebras over $H^*(BG_2; {\mathbb Z}/2)\cong {\mathbb Z}/2[y_4, y_6, y_7]$, where 
$\deg x_i=i$ and $\deg y_j = j$; see  \cite[Theorem 1.7]{Kuri:moduleadjoint}. 
\begin{thm}
\label{thm:BG_2} The dual to the loop coproduct 
$$
Dlcop: H^*(LBG_2; {\mathbb Z}/2)\otimes H^*(LBG_2; {\mathbb Z}/2) \to H^{*-14}(LBG_2; {\mathbb Z}/2)
$$
is commutative strictly and the only non-trivial  forms restricted to the submodule 
$\Delta(x_3, x_5, x_6)\otimes \Delta(x_3, x_5, x_6)$ are  given by 
$Dlcop(x_3x_5x_6\otimes 1) = Dlcop(x_3x_5\otimes x_6) 
= Dlcop(x_3x_6\otimes x_5) = Dlcop(x_5x_6\otimes x_3)=1$,
$$
\begin{array}{l}
Dlcop(x_3x_5x_6\otimes x_3)=Dlcop(x_3x_5\otimes x_3x_6) = x_3,\\
Dlcop(x_3x_5x_6\otimes x_5)=Dlcop(x_3x_5\otimes x_5x_6) = x_5,\\
Dlcop(x_3x_5x_6\otimes x_6) =Dlcop(x_3x_6\otimes x_5x_6) = x_6+y_6, \\
Dlcop(x_3x_5x_6\otimes x_3x_5) =  x_3x_5, \\
Dlcop(x_3x_5x_6\otimes x_3x_6) =  x_3x_6+x_3y_6,\\ 
Dlcop(x_3x_5x_6\otimes x_5x_6) = x_5x_6 + x_5y_6 + y_4y_7, \\
Dlcop(x_3x_5x_6\otimes x_3x_5x_6) = x_3x_5x_6 + x_3x_5y_6 + x_3y_4y_7 + y_7^2. 
\end{array}
$$
\end{thm}

\begin{lem}\label{product in free loop cohomology}
Let $k:\{1,\dots, q\}\rightarrow \{1,\dots, N\}$,
$j\mapsto k_j$ be a map such that $\forall 1\leq i\leq N$,
the cardinality of the inverse image $k^{-1}(\{ i \})$ is $\leq 2$.
In $H^*(LX;\mathbb{F}_2)=\mathbb{F}_2 [y_1,\dots,y_N]\otimes\Delta(x_1,\dots,x_N)$,
the cup product satisfies the equality
$$x_{k_1}\cdots x_{k_q}=
\sum_{\substack{ 0\leq l\leq\text{cardinal of }\{k_1,\dots,k_q\},\\1 \leq i_1<\dots<i_l\leq N}} P_{i_1,\dots,i_l}x_{i_1}\cdots x_{i_l}
$$
where $P_{i_1,\dots,i_l}$ are elements of $\mathbb{F}[y_1,\dots,y_N]$.
\end{lem}
\begin{proof}
Suppose by induction that the lemma is true for $q-1$.
If the elements $k_1$, \dots, $k_q$ are pairwise distinct,
take $\{i_1,\dots,i_l\}=\{k_1,\dots,k_q\}$.
Otherwise by permuting the elements $x_{k_1}$, \dots, $x_{k_q}$,
suppose that $k_{q-1}=k_q$.
$$
x_{k_q}^2=\Delta\circ ev^*\circ\text{Sq}^{\vert y_{k_q}\vert-1} (y_{k_q})
=\sum_{i=1}^N x_i P_i
$$
where $P_1$,\dots,$P_N$  are elements of $\mathbb{F}[y_1,\dots,y_N]$.
So $$
x_{k_1}\cdots x_{k_q}
=\sum_{i=1}^N x_{k_1}\cdots x_{k_{q-2}}x_i P_i
.$$
Since $k_q=k_{q-1}$, by hypothesis, $k_q\notin \{k_1,\dots,k_{q-2}\}$.
Therefore the cardinal of $\{k_1,\dots,k_{q-2},i\}$ is less or equal to the cardinal of $\{k_1,\dots,k_{q}\}$.
By our induction hypothesis,
$$
x_{k_1}\cdots x_{k_{q-2}}x_i
=\sum_{\substack{ 0\leq l\leq\text{cardinal of }\{k_1,\dots,k_{q-2},i\},\\1 \leq i_1<\dots<i_l\leq N}} P_{i_1,\dots,i_l}x_{i_1}\cdots x_{i_l}.
$$
\end{proof}
\begin{lem}\label{lem:trivial loop coproduct}
Let $k:\{1,\dots, q+r\}\rightarrow \{1,\dots, N\}$,
$j\mapsto k_j$ be a non-surjective map such that $\forall 1\leq i\leq N$,
the cardinality of the inverse image $k^{-1}(\{ i \})$ is $\leq 2$.
Then 
$$Dlcop(x_{k_1}\cdots x_{k_q}\otimes x_{k_{q+1}}\cdots x_{k_{q+r}}) = 0.
$$ 
\end{lem}
\begin{proof}
We do an induction on $r\geq 0$.

\noindent Case $r=0$:
By Lemma~\ref{product in free loop cohomology},
since the cardinal of $\{ k_1,\dots,k_q\}<N$,
$$Dlcop(x_{k_1}\cdots x_{k_q}\otimes 1)=
\sum_{\substack{ 0\leq l< N,\\1 \leq i_1<\dots<i_l\leq N}}Dlcop(P_{i_1,\dots,i_l}x_{i_1}\cdots x_{i_l}\otimes 1)
$$
where $P_{i_1,\dots,i_l}$ are elements of $\mathbb{F}[y_1,\dots,y_N]$.
By (3) and (4) of Theorem~\ref{dlcop given by formulas}, since $l<N$,
$$Dlcop(P_{i_1,\dots,i_l}x_{i_1}\cdots x_{i_l}\otimes 1)=0.$$
Suppose now by induction that the Lemma is true for $r-1$. Then by (1) of Theorem~\ref{dlcop given by formulas},
\begin{align*}
Dlcop(x_{k_1}\cdots x_{k_q}\otimes x_{k_{q+1}}\cdots x_{k_{q+r}}) =
x_{k_{q+1}}Dlcop(x_{k_1}\cdots x_{k_q}\otimes x_{k_{q+2}}\cdots x_{k_{q+r}})\\
+Dlcop(x_{k_1}\cdots x_{k_{q+1}}\otimes x_{k_{q+2}}\cdots x_{k_{q+r}})
=x_{k_{q+1}}\times 0+0.
\end{align*}
\end{proof}

Let $I=\{i_1, ..,i_l\}\subset\{1, ..., N\}$. In $\Delta(x_1,\cdots,x_N)$, denote by
$x_I$ the generator $x_{i_1}\cup x_{i_2}\cup\dots \cup x_{i_l}$. Since mod $2$, the cup product is
strictly commutative, we don't need to assume that $i_1<i_2<\cdots<i_l$. 

\begin{thm}\label{un terme en plus}
Let $I$ and $J$ be two subsets
of $\{1,\dots, N\}$.
Then 
$$
 Dlcop (x_I\otimes x_J)= 
\left\{
\begin{array}{ll}
Dlcop(x_1\dots x_N\otimes x_{I\cap J}) & \text{if} \ I\cup J=\{1,\dots, N\},  \\
0  & \text{otherwise}.
\end{array}
\right.
$$ 
In particular,
$\{x_I,x_J\}=\Delta(Dlcop(x_I\otimes x_J))=
\Delta(Dlcop(x_{I\cup J}\otimes x_{I\cap J}))=\{x_{I\cup J},x_{I\cap J}\}$.
\end{thm}
\begin{proof}
Let $\{i_1, ..,i_l\}$ denote the elements of the relative complement $I-J$.
Let $\{j_1, ..,j_m\}$ denote the elements of the relative complement $J-I$.
Let $\{k_1, ..,k_u\}$ denote the elements of the intersection $I\cap J$.

By Lemma~\ref{lem:trivial loop coproduct},
$Dlcop(x_{i_1}\dots x_{i_l}x_{k_1}\dots x_{k_u}\otimes x_{j_2}\dots x_{j_m}x_{k_1}\dots x_{k_u})=0$.
So by (1) of Theorem~\ref{dlcop given by formulas},
\begin{align*}
Dlcop(x_{i_1}\dots x_{i_l}x_{k_1}\dots x_{k_u}\otimes x_{j_1}\dots x_{j_m}x_{k_1}\dots x_{k_u})
=x_{j_1}\times 0\\
+Dlcop(x_{i_1}\dots x_{i_l}x_{j_1}x_{k_1}\dots x_{k_u}\otimes x_{j_2}\dots x_{j_m}x_{k_1}\dots x_{k_u}).
\end{align*}
By induction on $m$, this is equal to
$$Dlcop(x_{i_1}\dots x_{i_l}x_{j_1}\dots x_{j_m}x_{k_1}\dots x_{k_u}\otimes x_{k_1}\dots x_{k_u}).$$
So we have proved that
$Dlcop (x_I\otimes x_J)=Dlcop(x_{I\cup J}\otimes x_{I\cap J})$.
By Lemma~\ref{lem:trivial loop coproduct}, if $I\cup J\neq \{1,\dots, N\}$
then $Dlcop (x_I\otimes x_J)=0$.
\end{proof}
\begin{thm}\label{unit for Dlcop}
Let $X$ be a simply-connected space such that
$H^*(X;\mathbb{F}_2)$ is the polynomial algebra $\mathbb{F}_2[y_1,\dots,y_N]$.
The dual of the loop coproduct admits
$Dlcop(x_1\dots x_N\otimes x_1\dots x_N)\in H^d(LX;\mathbb{F}_2)$
as unit.
\end{thm}
\begin{lem}\label{square of Dlcop}
Let $a\in H^*(LX;\mathbb{F}_2)$

(1) For $1\leq i\leq N$, $x_iDlcop(a\otimes a)=0$.

(2) For any $b\in H^*(LX;\mathbb{F}_2)$,
$$
Dlcop(Dlcop(a\otimes a)\otimes b)=bDlcop(Dlcop(a\otimes a)\otimes 1).
$$
\end{lem}
\begin{proof}[Proof of Lemma~\ref{square of Dlcop}]
(1) By (1) of Theorem~\ref{dlcop given by formulas},
$$Dlcop(a\otimes ax_i)=x_iDlcop(a\otimes a)+Dlcop(ax_i\otimes a).$$
Since $Dlcop$ is graded commutative~\cite{Chataur-Menichi:stringclass},
$
Dlcop(a\otimes ax_i)=Dlcop(ax_i\otimes a)
$.
So $x_iDlcop(a\otimes a)=0$.

(2) By (1) and (1) of Theorem~\ref{dlcop given by formulas},
$$
Dlcop(Dlcop(a\otimes a)\otimes bx_i)=x_iDlcop(Dlcop(a\otimes a)\otimes b)+0.
$$
Therefore by induction
$$
Dlcop(Dlcop(a\otimes a)\otimes x_{i_1}\dots x_{i_l})=
x_{i_1}\dots x_{i_l}Dlcop(Dlcop(a\otimes a)\otimes 1).
$$
Using (4) of Theorem~\ref{dlcop given by formulas}, we obtain (2).
\end{proof}
\begin{proof}[Proof of Theorem~\ref{unit for Dlcop}]
Since $Dlcop$ is graded associative~\cite{Chataur-Menichi:stringclass}
and using (2) of Theorem~\ref{dlcop given by formulas} twice,
\begin{eqnarray*}
&&Dlcop(Dlcop(x_1\dots x_N\otimes x_1\dots x_N)\otimes 1)\\
&=&Dlcop(x_1\dots x_N\otimes Dlcop(x_1\dots x_N\otimes 1))
=Dlcop(x_1\dots x_N\otimes 1)=1.
\end{eqnarray*}
Therefore using (2) of Lemma~\ref{square of Dlcop},
\begin{eqnarray*}
&&Dlcop(Dlcop(x_1\dots x_N\otimes x_1\dots x_N)\otimes b)\\
&=&bDlcop(Dlcop(x_1\dots x_N\otimes x_1\dots x_N)\otimes 1) = b\times 1=b.
\end{eqnarray*}
\end{proof}
The simplest connected Lie group with non-trivial Steenrod operation $Sq_1$ in the cohomology of its classifying space is $SO(3)$.
\begin{thm}\label{thm:BSO_3}
The cup product and the dual of the loop coproduct on the mod $2$ free loop cohomology
of the classifying space of $SO(3)$ are given by
\begin{eqnarray*}
H^*(LBSO(3); {\mathbb Z}/2)&\cong& 
\Delta(x_1, x_2)\otimes {\mathbb Z}/2[y_2, y_3] \\
&\cong & {\mathbb Z}/2[x_1, x_2]\otimes {\mathbb Z}/2[y_2, y_3] \left/ \right.
\left(
\begin{array}{cc}
  x_1^2+x_2   \\
  x_2^2 + x_2y_2 + y_3x_1   
\end{array}
\right)
\end{eqnarray*}
as algebras over $H^*(BSO(3); {\mathbb Z}/2)\cong {\mathbb Z}/2[y_2, y_3]$, where 
$\deg x_i=i$ and $\deg y_j = j$.
$$
\begin{array}{l}
Dlcop(x_1x_2\otimes 1) = Dlcop(x_1\otimes x_2)=1,\\
Dlcop(x_1x_2\otimes x_1)= x_1,\quad
Dlcop(x_1x_2\otimes x_2)= x_2+y_2,\\
Dlcop(x_1x_2\otimes x_1x_2) = x_1x_2 + x_1y_2 + y_3,
\end{array}
$$
\end{thm}
\begin{proof}
The cohomology $H^*(BSO(3); {\mathbb Z}/2)$ is the polynomial algebra on the Stiefel-Whitney classes $y_2$ and $y_3$ of the tautological bundle $\gamma^3$ (\cite[Theorem 7.1]{Milnor-Stasheff}
or~\cite[III.Corollary 5.10]{Mimura-Toda:topliegroups}). By Wu formula~\cite[III.Theorem 5.12(1)]{Mimura-Toda:topliegroups},
$Sq^1 y_2=y_3$ and $Sq^2 y_3=y_2y_3$.
Now the computation of the cup product and of the dual of the loop coproduct follows from Theorem~\ref{dlcop given by formulas}.
\end{proof}
In the following proof, we detail the computation of the cup
product and the dual of the loop coproduct following Theorem~\ref{dlcop given by formulas}
for a more complicated example of Lie group.

\medskip
\noindent
{\it Proof of Theorem \ref{thm:BG_2}.}  
 Observe that $Sq^2y_4=y_6$, $Sq^1y_6=y_7$~\cite[VII.Corollary 6.3]{Mimura-Toda:topliegroups}
and hence 
$Sq^3y_4=Sq^1Sq^2y_4=y_7$. From~\cite[Proof of Theorem 1.7]{Kuri:moduleadjoint},
$Sq^5y_6=y_4y_7$ and $Sq^6y_7=y_6y_7$.
Therefore the computation of the cup product on $H^*(LBG_2;\mathbb{Z}/2)$ follows from Theorem~\ref{dlcop given by formulas}: $x_3^2=x_6$, $x_5^2=x_3y_7+y_4x_6$ and $x_6^2=x_5y_7+y_6x_6$. 

Lemma \ref{lem:trivial loop coproduct} implies that monomials with non-trivial loop coproduct 
are ones only listed in the theorem.

By (2) of Theorem~\ref{dlcop given by formulas},
$$Dlcop(x_3x_5x_6\otimes 1) = Dlcop(x_3x_5\otimes x_6) = Dlcop(x_3x_6\otimes x_5)= Dlcop(x_5x_6\otimes x_3)=1.$$
By~Lemma \ref{lem:trivial loop coproduct},
$Dlcop(x_3x_5^2\otimes 1)=0$.
By (1) of Theorem~\ref{dlcop given by formulas},
$$Dlcop(x_3x_5x_6\otimes x_6)=x_6 Dlcop(x_3x_5x_6\otimes 1)
+Dlcop(x_3x_5x_6^2\otimes 1).$$
Since $x_3x_5x_6^2=x_3x_5(x_5y_7+y_6x_6)$,
by (4) of Theorem~\ref{dlcop given by formulas},
$$Dlcop(x_3x_5x_6^2\otimes 1)=
y_7Dlcop(x_3x_5^2\otimes 1)
+y_6Dlcop(x_3x_5x_6\otimes 1)=y_7\times 0+y_6\times 1$$
 So finally
$Dlcop(x_3x_5x_6\otimes x_6)=x_6+ y_6$.

By Theorem~\ref{un terme en plus}, 
$Dlcop(x_3x_6\otimes x_5x_6)=Dlcop(x_3x_5x_6\otimes x_6)$.
 
Since 
$
x_3x_5^2x_6=x_5y_7^2+x_6y_6y_7+x_3x_5y_7y_4+x_3x_6y_6y_4
$,
by (1) of Theorem~\ref{dlcop given by formulas} and Lemma \ref{lem:trivial loop coproduct},
\begin{align*}
Dlcop(x_3x_5x_6\otimes x_5x_6)=
x_5 Dlcop(x_3x_5x_6\otimes x_6)+Dlcop(x_3x_5^2x_6\otimes x_6)
\\=x_5(x_6+y_6)+y_7^2\times 0+y_6y_7\times 0+y_7y_4\times 1+y_6y_4\times 0.
\end{align*}
The other computations are left to the reader. 
\hfill\qed

\medskip
\noindent
We would like to emphasize that Theorem~\ref{thm:BG_2} gives at the same time,
the cup product and the dual of the loop coproduct on $H^{*}(LBG_2;Z/2)$.
As mentioned in Introduction, if we forget the cup product, then the following Theorem shows that the
dual of the loop coproduct is really simple:

\begin{thm}\label{iso of algebras between free loop space cohomology and Hochschild}
Let $X$ be a simply-connected space such that $H^*(X;\mathbb{F}_2)$
is the polynomial algebra $\mathbb{F}_2[V]$.
Then with respect to the dual of the loop coproduct,
there is an isomorphism of graded algebras between 
$H^{*+d}(LX;\mathbb{F}_2)$ and the tensor product of algebras
$H^*(X;\mathbb{F}_2)\otimes H_{-*}(\Omega X;\mathbb{F}_2)\cong
\mathbb{F}_2[V]\otimes\Lambda(sV)^\vee$.
\end{thm}

\medskip

\begin{lem}~\label{morphism of algebras from free loop space to pointed loop space}
Let $X$ be a simply-connected space such that $H^*(X;\mathbb{F}_2)=\mathbb{F}_2[V]$. Let $x_1,\dots,x_N$ be a basis of $sV$.

1) Suppose that $\{i_1, ..,i_l\}\cup\{j_1, ..., j_m\}=\{1, ..., N\}$. Let $\{k_1, .., k_u\}:=\{i_1, ...,i_l\}\cap \{j_1, ..., j_m\}$.
Then $$
H^*(i)\circ Dlcop(x_{i_1}\cdots x_{i_l}\otimes x_{j_1}\cdots x_{j_m})
= x_{k_1}\cdots x_{k_u}.$$

2 ) 
Let $\Theta:H_{-*}(\Omega X)=\wedge (sV)^\vee\buildrel{\cong}\over\rightarrow H^{*+d}(\Omega X)=\Delta (sV)$
be the linear isomorphism defined by
$$
\Theta(x_{j_1}^\vee\wedge\dots \wedge x_{j_p}^\vee)
= x_{1}\cup\dots\cup\widehat{x_{j_{1}}}\cup\dots\cup\widehat{x_{j_p}}\cup\dots \cup x_{N}.
$$
Here $^\vee$ denote the dual and $\widehat{\quad}$ denotes omission.
Then the composite $\Theta^{-1}\circ H^*(i):H^{*+d}(LX)\rightarrow
H^{*+d}(\Omega X)\buildrel{\cong}\over\rightarrow H_{-*}(\Omega X)$
is a morphism of graded algebras preserving the units.
\end{lem}
\begin{proof}[Proof of Lemma~\ref{morphism of algebras from free loop space to pointed loop space}]
1) Suppose that $\vert x_{k_1}\vert \geq \cdots\geq  \vert x_{k_u}\vert$.
There exists polynomials $P_1$,\dots,$P_N\in \mathbb{F}[y_1,\dots,y_N]$
possibly null such that
$$
x_{k_1}^2=\Delta\circ ev^*\circ\text{Sq}^{\vert y_{k_1}\vert-1} (y_{k_1})
=\sum_{i=1}^N x_i P_i.
$$
If $P_i$ is of degree $0$, since $\vert x_i\vert > \vert x_{k_1}\vert$,
$x_i$ is not one of the elements $x_{k_1}$,\dots, $x_{k_u}$
and so by Lemma~\ref{lem:trivial loop coproduct}
$
Dlcop(x_{i_1}\cdots\widehat{x_{k_{1}}}\cdots x_{i_l}x_i
\otimes x_{j_1}\cdots \widehat{x_{k_{1}}}\cdots x_{j_m})=0.
$

\noindent If $P_i$ is of degree $\geq 1$, by (4) of Theorem~\ref{dlcop given by formulas},
$H^*(i)\circ Dlcop(P_i x_{i_1}\cdots\widehat{x_{k_{1}}}\cdots x_{i_l}x_i
\otimes x_{j_1}\cdots \widehat{x_{k_{1}}}\cdots x_{j_m})=0$

\noindent Therefore $H^*(i)\circ Dlcop(x_{i_1}\cdots\widehat{x_{k_{1}}}\cdots x_{i_l}x_{k_{1}}^2
\otimes x_{j_1}\cdots \widehat{x_{k_{1}}}\cdots x_{j_m})=0$.
Now the same proof as the proof of Theorem~\ref{thm:BG_explicit_cal} shows 1).

2) Since $H^*(\Omega X;\mathbb{F}_2)$ is generated by the $x_i:=\sigma (y_i)$, $1\leq i\leq N$
which are primitives, by~\cite[4.20 Proposition]{Milnor-Moore}, all squares vanish in $H_*(\Omega X;\mathbb{F}_2)$. Therefore $H_*(\Omega X;\mathbb{F}_2)$ is the exterior algebra $\Lambda \sigma(y_i)^\vee$.

Let $I=\{i_1, ..,i_l\}\subset\{1, ..., N\}$. Recall from Theorem~\ref{un terme en plus}
that in $\Delta(x_1,\cdots,x_N)$, $x_I$ 
denotes the generator $x_{i_1}\cup x_{i_2}\cup\dots \cup x_{i_l}$.
Denote also in the exterior algebra $\Lambda(x_1^\vee,\cdots,x_N^\vee)$
by $x_I^\vee$ the element $x_{i_1}^\vee\wedge x_{i_2}^\vee\wedge\dots \wedge x_{i_l}^\vee$.
Then with this notation,
$\Theta(x_I^\vee)=x_{I^c}$ where $I^c$ is the complement of $I$ in $\{1, ..., N\}$.
Let $comp^!:H^{*+d}(\Omega X)\otimes H^{*+d}(\Omega X)\rightarrow  H^{*+d}(\Omega X)$
be the multiplication defined by
$comp^!(x_I\otimes x_J)=x_{I\cap J}$ if $I\cup J=\{1, ..., N\}$ and $0$ otherwise.
By (1) and Lemma~\ref{lem:trivial loop coproduct}, $H^*(i):H^{*+d}(LX)\rightarrow H^{*+d}(\Omega X)$ commutes with the products $Dlcop$ and $comp^!$.
Since $x_{(I\cup J)^c}=x_{I^c\cap J^c}$, $\Theta:H_{-*}(\Omega X)\rightarrow H^{*+d}(\Omega X)$
commutes with the exterior product and $comp^!$.

By Theorem~\ref{unit for Dlcop}, $Dlcop(x_1\dots x_N\otimes x_1\dots x_N)$ is the unit
of Dlcop. By (1),
$$\Theta^{-1}\circ H^*(i)\circ Dlcop(x_1\dots x_N\otimes x_1\dots x_N)= \Theta^{-1}(x_1\dots x_N)=1.$$
Therefore  $\Theta^{-1}\circ H^*(i)$ preserves also the units.
\end{proof}
\begin{proof}[Proof of Theorem~\ref{iso of algebras between free loop space cohomology and Hochschild}]
Denote by $\mathbb{I}:=Dlcop(x_1\dots x_N\otimes x_1\dots x_N)$
the unit of $H^{*+d}(LX;\mathbb{F}_2)$ (Theorem~\ref{unit for Dlcop}).
By (6) of Theorem~\ref{Cup in string classifying},
the map $s^!:H^*(X)\rightarrow H^{*+d}(LX)$, $a\mapsto ev^*(a)\mathbb{I}$,
is a morphism of unital commutative graded algebras.

By Lemma~\ref{lem:trivial loop coproduct}, we have 
$Dlcop(x_1\dots\widehat{x_i}\dots x_N\otimes x_1\dots\widehat{x_i}\dots x_N)=0$.
So let $\sigma:H^{*+d}(\Omega X)\rightarrow H^{*+d}(LX)$ be the unique linear
map such that for $\forall 1\leq i\leq N$,
$\sigma(x_1\dots\widehat{x_i}\dots x_N)=x_1\dots\widehat{x_i}\dots x_N$
and such that
$\sigma\circ\Theta:H_{-*}(\Omega X)=\Lambda(sV)^\vee\rightarrow H^{*+d}(LX)$
is a morphism of unital commutative graded algebras.
For $1\leq i\leq N$,
$\Theta^{-1}\circ H^*(i)\circ\sigma\circ\Theta (x_i^\vee)=x_i^\vee$.
By Lemma~\ref{morphism of algebras from free loop space to pointed loop space},
the composite $\Theta^{-1}\circ H^*(i):H^{*+d}(LX)\rightarrow
H^{*+d}(\Omega X)\buildrel{\cong}\over\rightarrow H_{-*}(\Omega X)$
is a morphism of graded algebras.
So the composite $\Theta^{-1}\circ H^*(i)\circ \sigma\circ\Theta$
is the identity map and $\sigma$ is a section of $H^*(i)$.
So by Leray-Hirsch theorem, the linear morphism of $H^*(X)$-modules
$
H^*(X)\otimes H^*(\Omega X)\rightarrow H^*(LX)
$,
$a\otimes g\mapsto ev^*(a)\sigma(g)$,
is an isomorphism.

The composite $$\varphi:H^*(X)\otimes H_{-*}(\Omega X)
\buildrel{s^!\otimes \sigma\circ\Theta}\over\rightarrow
H^{*+d}(LX)\otimes H^{*+d}(LX)\buildrel{Dlcop}\over\rightarrow H^{*+d}(LX)$$
is a morphism of commutative graded algebras with respect to the dual of the
loop coproduct. By (4) of Theorem~\ref{dlcop given by formulas} and since
$\mathbb{I}$ is an unit for $Dlcop$,
$\varphi(a\otimes g)=Dlcop(ev^*(a)\mathbb{I}\otimes \sigma\circ\Theta(g))
=ev^*(a)\sigma\circ\Theta(g)$.
Therefore $\varphi$ is an isomorphism.
\end{proof}
\begin{ex}\label{conjecture verifie pour SO3}
With respect to the dual of the loop coproduct,
there is an isomorphism of algebras between
$H^{*+3}(LBSO(3);Z/2)$ and 
$$H_{-*}(SO(3);Z/2)\otimes H^*(BSO(3);Z/2)\cong \wedge (u_{-1},u_{-2})\otimes Z/2[v_2,v_3].$$
\end{ex}
\begin{proof}
By Theorem~\ref{unit for Dlcop},
$Dlcop(x_1x_2\otimes x_1x_2) = x_1x_2 + x_1y_2 + y_3$
is an unit for the dual of the 
loop coproduct on $H^{*+3}(LBSO(3);Z/2)$.
By Lemma~\ref{lem:trivial loop coproduct},
$$Dlcop(x_1\otimes x_1)=Dlcop(x_2\otimes x_2)=0.$$
So let $\varphi:\wedge (u_{-1},u_{-2})\otimes Z/2[v_2,v_3]\rightarrow
H^{*+3}(LBSO(3);Z/2)$ be the unique morphism of algebras such
that $\varphi(u_{-2})=x_1$, $\varphi(u_{-1})=x_2$,
$\varphi(v_2)=y_2(x_1x_2 + x_1y_2 + y_3)$ and
$\varphi(v_3)=y_3(x_1x_2 + x_1y_2 + y_3)$.

For all $i$, $j\geq 0$, we see that 
$\varphi(v_2^iv_3^j)=y_2^iy_3^j(x_1x_2 + x_1y_2 + y_3)$, 
$\varphi(u_{-1}u_{-2}v_2^iv_3^j)=y_2^iy_3^j$,
$\varphi(u_{-1}v_2^iv_3^j)=x_2y_2^iy_3^j$ and
$\varphi(u_{-2}v_2^iv_3^j)=x_1y_2^iy_3^j$.
Therefore $\varphi$ sends a linear basis of $\wedge (u_{-1},u_{-2})\otimes Z/2[v_2,v_3]$ to a linear basis $H^{*+3}(LBSO(3);Z/2)$.
So $\varphi$ is an isomorphism.
\end{proof}
\begin{ex}\label{conjecture verifie pour G2}
With respect to the dual of the loop coproduct,
there is an isomorphism of algebras between
$H^{*+14}(LBG_2;Z/2)$ and $H_{-*}(G_2;Z/2)\otimes H^*(BG_2;Z/2)\cong \wedge (u_{-3},u_{-5},u_{-6})\otimes Z/2[v_4,v_6,v_7]$.
\end{ex}
\begin{proof}
By Theorem~\ref{unit for Dlcop}, $
Dlcop(x_3x_5x_6\otimes x_3x_5x_6)=
x_3x_5x_6+x_3x_5y_6+x_3y_4y_7+y_7^2$ is an unit for the dual of the 
loop coproduct on $H^{*+14}(LBG_2;Z/2)$.
By Lemma~\ref{lem:trivial loop coproduct},
$$Dlcop(x_5x_6\otimes x_5x_6)=Dlcop(x_3x_6\otimes x_3x_6)=Dlcop(x_3x_5\otimes x_3x_5)=0.$$
So let $\varphi:\wedge (u_{-3},u_{-5},u_{-6})\otimes Z/2[v_4,v_6,v_7]\rightarrow
H^{*+14}(LBG_2;Z/2)$ be the unique morphism of algebras such
that $\varphi(u_{-3})=x_5x_6$, $\varphi(u_{-5})=x_3x_6$, $\varphi(u_{-6})=x_3x_5$,
$\varphi(v_4)=y_4(x_3x_5x_6+x_3x_5y_6+x_3y_4y_7+y_7^2)$,
$\varphi(v_6)=y_6(x_3x_5x_6+x_3x_5y_6+x_3y_4y_7+y_7^2)$ and
$\varphi(v_7)=y_7(x_3x_5x_6+x_3x_5y_6+x_3y_4y_7+y_7^2)$.

For all $i$, $j$ and $k\geq 0$, we see that 
$\varphi(v_4^iv_6^jv_7^k)=y_4^iy_6^jy_7^k(x_3x_5x_6+x_3x_5y_6+x_3y_4y_7+y_7^2)$,
$\varphi(u_{-3}u_{-5}u_{-6}v_4^iv_6^jv_7^k)=y_4^iy_6^jy_7^k$,
$\varphi(u_{-3}u_{-5}v_4^iv_6^jv_7^k)=(x_6+y_6)y_4^iy_6^jy_7^k$,
$\varphi(u_{-3}u_{-6}v_4^iv_6^jv_7^k)=x_5y_4^iy_6^jy_7^k$,
$\varphi(u_{-5}u_{-6}v_4^iv_6^jv_7^k)=x_3y_4^iy_6^jy_7^k$,
$\varphi(u_{-3}v_4^iv_6^jv_7^k)=x_5x_6y_4^iy_6^jy_7^k$,
$\varphi(u_{-5}v_4^iv_6^jv_7^k)=x_3x_6y_4^iy_6^jy_7^k$
and $\varphi(u_{-6}v_4^iv_6^jv_7^k)=x_3x_5y_4^iy_6^jy_7^k$.
Therefore $\varphi$ sends a linear basis of $\wedge (u_{-3},u_{-5},u_{-6})\otimes Z/2[v_4,v_6,v_7]$ to a linear basis $H^{*+14}(LBG_2;Z/2)$.
So $\varphi$ is an isomorphism.
\end{proof}
\begin{lem}\label{multiplication par inverse}
Let $(A,\bullet)$ be a commutative unital associative graded algebra such that
$x\bullet x=1$. Let $\psi:A\rightarrow A$ be the linear morphism mapping $a$ to $x\bullet a$. Then $\psi$ is an involutive isomorphism such that $\psi(a)\bullet\psi(b)=a\bullet b$.
\end{lem}
\begin{proof}
$\psi(a)\bullet\psi(b)=(x\bullet a)\bullet (x\bullet b)=(x\bullet x)\bullet (a\bullet b)=1\bullet (a\bullet b)=a\bullet b.$
\end{proof}
\begin{proof}[Second proof of Theorem~\ref{iso of algebras between free loop space cohomology and Hochschild} which gives another (better?) algebra isomorphism]

By commutativity and associativity of Dlcop and Theorem~\ref{unit for Dlcop}, applying Lemma~\ref{multiplication par inverse},
$
\psi:H^*(X)\otimes H^{*+d}(\Omega X)\rightarrow H^{*+d}(LX)
$
defined by $$\psi(a\otimes x_{k_1}\dots x_{k_u})=Dlcop(x_1\dots x_N\otimes ax_{k_1}\dots x_{k_u})$$ is an involutive isomorphism such that
$$
Dlcop(\psi(a\otimes x_I)\otimes\psi(b\otimes x_J))=Dlcop(ax_I\otimes bx_J)
$$ for any subsets $I$ and $J$ of $\{1, ..., N\}$.

Case $I\cup J=\{1, ..., N\}$. By Theorem~\ref{un terme en plus},
$$Dlcop(ax_I\otimes bx_J)=Dlcop(x_1\dots x_N\otimes abx_{I\cap J})=\psi(ab\otimes x_{I\cap J})=
\psi(ab\otimes comp^!(x_I\otimes x_J)).$$

Case $I\cup J\neq\{1, ..., N\}$.
By Theorem~\ref{un terme en plus}, $Dlcop(ax_I\otimes bx_J)=0$ and
$comp^!(x_I\otimes x_J)=0$.

Therefore $\psi$ is a morphism of graded algebras.

One can shows that $\{\psi(1\otimes \Theta(x_i^\vee)),\psi(1\otimes\Theta(x_j^\vee))\}=0$.
That is why this isomorphism might be better.
\end{proof}
\begin{thm}\label{thm: BV-algebra sur algebre commutative libre pour SO3}
As Batalin-Vilkovisky algebra, $$H^{*+3}(LBSO(3);Z/2)\cong \wedge (u_{-1},u_{-2})\otimes Z/2[v_2,v_3]$$ where for all $i,j\geq 0$,
$\Delta(v_2^iv_3^j)=0$,
$\Delta(u_{-1}u_{-2}v_2^iv_3^j)=i u_{-2}v_2^{i-1}v_3^j+ju_{-1}v_2^iv_3^{j-1}$,
$$\Delta(u_{-2}v_2^iv_3^j)=i u_{-1}v_2^{i-1}v_3^j+jv_2^iv_3^{j-1}+ju_{-2}v_2^{i+1}v_3^{j-1}+ju_{-1}u_{-2}v_2^iv_3^j\text{ and}$$
$$\Delta(u_{-1}v_2^iv_3^j)=iv_2^{i-1}v_3^j+(i+j)u_{-2}v_2^iv_3^j+iu_{-1}u_{-2}v_2^{i-1}v_3^{j+1}
+ju_{-1}v_2^{i+1}v_3^{j-1}.$$
In particular $1\notin\text{Im }\Delta$.
\end{thm}
\begin{proof}
Theorem~\ref{thm:BSO_3} gives the BV-algebra $H^{*+3}(LBSO(3);Z/2)$ since $\Delta$ is a derivation with respect to the cup product.
In the proof of Example~\ref{conjecture verifie pour SO3},
the isomorphism of algebras $\varphi: \wedge (u_{-1},u_{-2})\otimes Z/2[v_2,v_3]\rightarrow
H^{*+3}(LBSO(3);Z/2)$ of Theorem~\ref{iso of algebras between free loop space cohomology and Hochschild} is 
made explicit on generators. We now transport the operator $\Delta$ using $\varphi$.

In degree $1$, the $\Delta$ operator is given by
$\Delta( u_{-1}u_{-2}v_2^2)=0$ and  $$\Delta(u_{-2}v_3)=\Delta(u_{-1}v_2)=1+u_{-2}v_2+u_{-1}u_{-2}v_3.$$
\end{proof}
\begin{thm}\label{thm: BV-algebra sur algebre commutative libre pour G2}
As Batalin-Vilkovisky algebra, $$H^{*+14}(LBG_2;Z/2)\cong \wedge (u_{-3},u_{-5},u_{-6})\otimes Z/2[v_4,v_6,v_7]$$
where for all $i,j,k\geq 0$,
$\Delta(v_4^iv_6^jv_7^k)=0$,
\begin{multline*}
\Delta(u_{-3}u_{-5}u_{-6}v_4^{i}v_6^{j}v_7^{k})
=iu_{-5}u_{-6}v_4^{i-1}v_6^jv_7^k
+ju_{-3}u_{-6}v_4^{i}v_6^{j-1}v_7^{k}\\
+ku_{-3}u_{-5}v_4^{i}v_6^{j}v_7^{k-1}
+ku_{-3}u_{-5}u_{-6}v_4^{i}v_6^{j+1}v_7^{k-1}
,
\end{multline*}
\begin{multline*}
\Delta(u_{-5}u_{-6}v_4^{i}v_6^{j}v_7^{k})
=iu_{-3}u_{-5}v_4^{i-1}v_6^{j}v_7^{k}
+iu_{-3}u_{-5}u_{-6}v_4^{i-1}v_6^{j+1}v_7^{k}\\
+ju_{-6}v_4^{i}v_6^{j-1}v_7^{k}
+ku_{-5}v_4^{i}v_6^{j}v_7^{k-1},
\end{multline*}
\begin{multline*}
\Delta(u_{-3}u_{-6}v_4^{i}v_6^{j}v_7^{k})
=iu_{-6}v_4^{i-1}v_6^{j}v_7^{k}
+ju_{-5}u_{-6}v_4^{i}v_6^{j-1}v_7^{k+1}\\
+ju_{-3}u_{-5}v_4^{i+1}v_6^{j-1}v_7^{k}
+ju_{-3}u_{-5}u_{-6}v_4^{i+1}v_6^{j}v_7^{k}
+ku_{-3}v_4^{i}v_6^{j}v_7^{k-1},
\end{multline*}
\begin{multline*}
\Delta(u_{-3}u_{-5}v_4^{i}v_6^{j}v_7^{k})
=iu_{-5}v_4^{i-1}v_6^{j}v_7^{k}
+iu_{-5}u_{-6}v_4^{i-1}v_6^{j+1}v_7^{k}\\
+ju_{-3}v_4^{i}v_6^{j-1}v_7^{k}
+(j+1+k)u_{-3}u_{-6}v_4^{i}v_6^{j}v_7^{k}
\end{multline*}
\begin{multline*}
\Delta(u_{-6}v_4^{i}v_6^{j}v_7^{k})
=iu_{-3}v_4^{i-1}v_6^{j}v_7^{k}
+ju_{-5}v_4^{i+1}v_6^{j-1}v_7^{k}
+ju_{-3}u_{-5}v_4^{i}v_6^{j-1}v_7^{k+1}\\
+(j+k)u_{-3}u_{-5}u_{-6}v_4^{i}v_6^{j}v_7^{k+1}
+kv_4^{i}v_6^{j}v_7^{k-1}
+ku_{-6}v_4^{i}v_6^{j+1}v_7^{k-1}
+ku_{-5}u_{-6}v_4^{i+1}v_6^{j}v_7^{k},
\end{multline*}
\begin{multline*}
\Delta(u_{-3}v_4^{i}v_6^{j}v_7^{k})
=iv_4^{i-1}v_6^{j}v_7^{k}
+iu_{-6}v_4^{i-1}v_6^{j+1}v_7^{k}
+(i+k)u_{-5}u_{-6}v_4^{i}v_6^{j}v_7^{k+1}\\
+iu_{-3}u_{-5}u_{-6}v_4^{i-1}v_6^{j}v_7^{k+2}
+ju_{-5}v_4^{i}v_6^{j-1}v_7^{k+1}
+ju_{-3}u_{-6}v_4^{i+1}v_6^{j-1}v_7^{k+1}\\
+(j+k)u_{-3}u_{-5}v_4^{i+1}v_6^{j}v_7^{k}
+(j+k)u_{-3}u_{-5}u_{-6}v_4^{i+1}v_6^{j+1}v_7^{k}
+ku_{-3}v_4^{i}v_6^{j+1}v_7^{k-1}\text{ and}
\end{multline*}
\begin{multline*}
\Delta(u_{-5}v_4^{i}v_6^{j}v_7^{k})
=iu_{-3}u_{-5}v_4^{i-1}v_6^{j+1}v_7^{k}
+iu_{-3}u_{-5}u_{-6}v_4^{i-1}v_6^{j+2}v_7^{k}
+jv_4^{i}v_6^{j-1}v_7^{k}\\
+(j+k)u_{-6}v_4^{i}v_6^{j}v_7^{k}
+ju_{-5}u_{-6}v_4^{i+1}v_6^{j-1}v_7^{k+1}
+ju_{-3}u_{-5}u_{-6}v_4^{i}v_6^{j-1}v_7^{k+2}
+ku_{-5}v_4^{i}v_6^{j+1}v_7^{k-1}.
\end{multline*}
In particular $1\notin\text{Im }\Delta$.
\end{thm}
\begin{proof}
Theorem~\ref{thm:BG_2} gives the BV-algebra $H^{*+14}(LBG_2;Z/2)$ since $\Delta$ is a derivation with respect to the cup product.
In the proof of Example~\ref{conjecture verifie pour G2},
the isomorphism of algebras $\varphi: \wedge (u_{-3},u_{-5},u_{-6})\otimes Z/2[v_4,v_6,v_7]\rightarrow
H^{*+14}(LG_2;Z/2)$ of Theorem~\ref{iso of algebras between free loop space cohomology and Hochschild} is 
made explicit on generators. We now transport the operator $\Delta$ using $\varphi$.

In degree $1$, the $\Delta$ operator is given by $\Delta( u_{-5}u_{-6}v_6^2)=0$,
$$\Delta(u_{-3}u_{-5}u_{-6}v_4^{2}v_7)=\Delta( u_{-5}u_{-6}v_4^3)=u_{-3}u_{-5}v_4^{2}+u_{-3}u_{-5}u_{-6}v_4^{2}v_6,$$
$$
\Delta(u_{-3}u_{-6}v_4v_6)=u_{-6}v_6+u_{-5}u_{-6}v_4v_7
+u_{-3}u_{-5}v_4^{2}+u_{-3}u_{-5}u_{-6}v_4^{2}v_6\text{ and}
$$
$$\Delta(u_{-6}v_7)=\Delta(u_{-5}v_6)=\Delta(u_{-3}v_4)=1+u_{-6}v_6+u_{-5}u_{-6}v_4v_7
 +u_{-3}u_{-5}u_{-6}v_7^{2}.$$
\end{proof}
Note that
$\varphi^{-1}\circ\Delta\circ\varphi(y_i\otimes x_i^\vee)=\varphi^{-1}(x_1\dots x_N)$
is independent of $i$.
%

\section{Relation with Hochschild cohomology}\label{Relation Hochschild cohomology}
Let $\K$ be any field.
Let $G$ be a connected compact Lie group of dimension $d$.
\begin{conjecture}\cite[Conjecture 68]{Chataur-Menichi:stringclass}\label{conjecture iso Hochschild}
There is an isomorphism of Gerstenhaber algebras
$$H^{*+d}(LBG)\buildrel{\cong}\over\rightarrow HH^{*}(S_{*}(G),S_{*}(G)).$$
\end{conjecture}

Suppose that $H^*(BG;\K)$ is a polynomial algebra $\K[V]=K[y_1,\dots, y_N]$.
It follows from~\cite[Theorem 9, p. 572]{Halperin-Stasheff:diff} (See also~\cite[Proposition 8.21]{McCleary:usersguide}) that $BG$ is $\K$-formal.
Then $BG$ is $\K$-coformal and $H_*(G;\K)$ is
the exterior algebra $\wedge (sV)^\vee$.
Indeed, since $BG$ is $\K$-formal, the Cobar construction
$\Omega H_*(BG)$ is weakly equivalent as algebras to $S_*(G)$.
Let $A_i$ denote the exterior algebra $\Lambda s^{-1}(y_i^\vee)$.
Then $EZ$, the Eilenberg-Zilber map and $\varepsilon$, the counit of the adjonction between the Bar and the Cobar construction give
the quasi-isomorphims of algebras
$$
\xymatrix@1{
&\Omega H_*(BG)=\Omega(\otimes_{i=1}^N BA_i)
& \otimes_{i=1}^N \Omega BA_i\ar[l]^-\simeq_-{EZ}\ar[r]_-\simeq^-{\otimes_{i=1}^N \varepsilon_i}
&\otimes_{i=1}^N A_i=\Lambda s^{-1}V^\vee.
}
$$
Alternatively, since $BG$ is $\K$-formal, you can use the implication $(2)\Rightarrow (1)$ in Theorem 2.14 of~\cite{Berglund-Borjeson:freeloopmanifold}.

Therefore, we have the isomorphism of Gerstenhaber algebras
$$
HH^{*}(S_{*}(G),S_{*}(G))\cong HH^{*}(H_*(G;\K),H_*(G;\K))\cong HH^{*}(\wedge (sV)^\vee,\wedge (sV)^\vee).
$$
By Theorem~\ref{thm: Hochschild cohomology exterior algebra} as graded algebras
$$
HH^{*}(\wedge (sV)^\vee,\wedge (sV)^\vee)\cong \wedge (sV)^\vee\otimes \K[V]\cong H_{-*}(G;\K)\otimes H^*(BG;\K).
$$

So in Theorem~\ref{iso of algebras between free loop space cohomology and Hochschild},
we have checked Conjecture~\ref{conjecture iso Hochschild} only for the algebra structure when $\K=\mathbb{F}_2$.
When $\K=\mathbb{F}_2$,
we would like to check conjecture~\ref{conjecture iso Hochschild} also for the Gerstenhaber algebra structure.

The following theorem shows that the conjecture is true for the Gerstenhaber algebra structure when $\K$ is a 
field of characteristic different
from $2$.
\begin{thm}\label{iso BV with Hochschild when no steenrod operation}
Under the hypothesis $(H)$,
the free loop space cohomology of the classifying space of $G$, $H^{*+\text{dim} G}(LBG;\mathbb{F})$
is  isomorphic as Batalin-Vilkovisky algebra to the Hochschild cohomology of $H_*(G;\mathbb{F})$, $HH^*(H_*(G;\mathbb{F});H_*(G;\mathbb{F}))$. In particular the underlying Gerstenhaber algebras are isomorphic.
\end{thm}
\begin{proof}
By hypothesis, $H^*(BG)\cong \K[V]=K[y_i]$ as algebras.
Then $H_*(G)\cong \Lambda (sV)^\vee=\Lambda x_j^\vee$ as algebras.

Let $\Psi:sV\rightarrow (sV)^{\vee\vee}$ be the canonical isomorphism
of the graded vector space $sV$ into its bidual.
By definition,
$\Psi(sv)(\varphi)=(-1)^{\vert\varphi\vert\vert sv\vert}\varphi(sv)$
for any linear form $\varphi$ on $sV$.

By Theorem~\ref{thm: Hochschild cohomology exterior algebra},
we have the BV-algebra isomorphism
$HH^*(H_*(G);H_*(G))\cong \Lambda (sV)^\vee\otimes \K[s^{-1}(sV)^{\vee\vee}]$
where for any $v\in V$ and $\varphi\in (sV)^\vee$,
$$
\Delta((1\otimes s^{-1}\Psi(sv))(\varphi\otimes 1))=
(-1)^{\vert v\vert}\{s^{-1}\Psi(sv),\varphi\}=
-\Psi(sv)(\varphi)
=-(-1)^{\vert\varphi\vert\vert sv\vert}\varphi(sv)
$$
and where $\Delta$ is trivial on $\Lambda (sV)^\vee$ and on $\K[s^{-1}(sV)^{\vee\vee}]$.

The isomorphism of algebras
$Id\otimes \K[s^{-1}\Psi]:
\Lambda (sV)^\vee\otimes \K[V]\rightarrow
\Lambda (sV)^\vee\otimes \K[s^{-1}(sV)^{\vee\vee}]$
is a isomorphism of BV-algebras if for any $v\in V$ and $\varphi\in (sV)^\vee$,
$\Delta((1\otimes v)(\varphi\otimes 1))=-(-1)^{\vert\varphi\vert\vert sv\vert}\varphi(sv)
$
and if $\Delta$ is trivial on $\Lambda (sV)^\vee$ and on $\K[V]$.

Taking $v=y_i$ and $\varphi=\sigma(y_j)^\vee=x_j^\vee$, we obtained that
$\Delta (y_i\otimes x_j^\vee)=1$ if $i=j$ and $0$ otherwise like in Theorem~\ref{structure BV quand pas de sq1}.
\end{proof}
\begin{thm}\label{thm:non BV-iso modulo 2}
For $G=SO(3)$ or $G=G_2$,
the free loop space modulo 2 cohomology of the classifying space of $G$, $H^{*+\text{dim} G}(LBG;\mathbb{F}_2)$
is not isomorphic as Batalin-Vilkovisky algebra to the Hochschild cohomology of $H_*(G;\mathbb{F}_2)$, $HH^*(H_*(G;\mathbb{F}_2);H_*(G;\mathbb{F}_2))$ although when $G=SO(3)$
the underlying Gerstenhaber algebras are isomorphic.
\end{thm}
The main result of~\cite{menichi:stringtopspheres} is that the same phenomenon appears for
Chas-Sullivan string topology even in the simple case of the two dimensional sphere $S^2$.
\begin{lem}\label{lem:iso BV implique unite et image de Delta}
Let $A$ and $B$ two unital BV-algebras.
Let $\varphi:A\rightarrow B$ be a linear map preserving the units and commuting with
the BV-operators $\Delta$ (For example if $\varphi$ is an isomorphism preserving the multiplications and the $\Delta$'s).
If $1_A\in\text{Im } \Delta$ then $1_B\in\text{Im } \Delta$.
\end{lem}
\begin{proof}
There exists $a\in A$ such that $\Delta(a)=1_A$. So $$1_B=\varphi(1_A)=\varphi(\Delta(a))=\Delta(\varphi(a))\in\text{Im } \Delta.$$
\end{proof}
\begin{lem}\label{surjectivity morphism of algebras}
Let $d\in\mathbb{N}$ be a non-negative integer.
Let $f:A\rightarrow B$ be a morphism of augmented graded algebras
such that $B=B_{\geq -d}$, i. e. $B$ is concentrated in degrees greater or equal than $-d$
and such that $B_0=\mathbb{F}$.
Then $f$ is surjective iff $Q(f)$ is surjective.
\end{lem}
\begin{proof}
When $d=0$, this Lemma is Proposition 3.8 of~\cite{Milnor-Moore}. But the proof of~\cite{Milnor-Moore} cannot be easily generalized. Therefore we provide a proof.

Filter $A$ by wordlength: $F^n(A):=\overline{A}\cdot \overline{A} \dots \cdot \overline{A}$
for any $n\geq 0$.
The sequence
$$
\oplus_{i=1}^n \overline{A}^{\otimes i-1}\otimes \overline{A}\cdot \overline{A}\otimes \overline{A}^{\otimes n-i}
\rightarrow \overline{A} ^{\otimes n}\twoheadrightarrow Q(A)^{\otimes n}\rightarrow 0
$$
is exact.
Alternatively, since over a field $\mathbb{F}$, $\overline{A}=\overline{A}\cdot \overline{A}\oplus Q(A)$,
$$
0\rightarrow +_{i=1}^n \overline{A}^{\otimes i-1}\otimes \overline{A}\cdot \overline{A}\otimes \overline{A}^{\otimes n-i}
\hookrightarrow \overline{A} ^{\otimes n}\twoheadrightarrow Q(A)^{\otimes n}\rightarrow 0
$$
is a short exact sequence. Therefore the iterated multiplication of $\overline{A}$
induces a natural map
$
Q(A)^{\otimes n}\twoheadrightarrow F^n(A)/F^{n+1}(A)
$
obviously surjective.

Assume that $Q(f)$ is surjective.
Then
$
Q(f)^{\otimes n}:Q(A)^{\otimes n}\twoheadrightarrow Q(B)^{\otimes n}
$
is also surjective.
Since the following square is commutative by naturality,
$$
\xymatrix{
Q(A)^{\otimes n}\ar[d]_{Q(f)^{\otimes n}}\ar[r]
&F^n(A)/F^{n+1}(A)\ar[d]^{Gr_nf}\\
Q(B)^{\otimes n}\ar[r]
&F^n(B)/F^{n+1}(B),
}
$$
the map induced by $f$, $Gr_nf$, is also surjective.
In a fixed degree, consider the commutative diagram
$$
\xymatrix{
0\ar[r]
&F^{n+1}(A)\ar[r]\ar[d]_{f|F^{n+1}(A)}
&F^{n}(A)\ar[r]\ar[d]^{f|F^{n}(A)}
&F^n(A)/F^{n+1}(A)\ar[d]^{Gr_nf}\ar[r]
&0\\
0\ar[r]
&F^{n+1}(B)\ar[r]
&F^{n}(B)\ar[r]
&F^n(B)/F^{n+1}(B)\ar[r]
&0
}
$$
with exact rows.
Suppose by induction that the restriction of $f$ to $F^{n+1}(A)$, $f|F^{n+1}(A)$, is surjective.
Then by the five Lemma, $f|F^{n}(A)$, is also surjective.
Since $F^{n}(B)$ is concentrated in degrees greator or equal than $n-2d$, in a fixed degree,
for large $n$, $F^{n}(B)$ is trivial and we can start the induction.
Therefore $f=f|F^{0}(A)$ is surjective.
\end{proof}
\begin{proof}[Proof of Theorem~\ref{thm:non BV-iso modulo 2}]
Since $H_*(G)$ is an exterior algebra, by Example~\ref{ex: unit belongs to image of Delta} b),
$1\in\text{Im }\Delta$ in the BV-algebra $HH^*(H_*(G);H_*(G))$.
On the contrary, by Theorems~\ref{thm: BV-algebra sur algebre commutative libre pour SO3} and~\ref{thm: BV-algebra sur algebre commutative libre pour G2}, the unit $1$ does not belong to the image of $\Delta$ in the BV-algebra
$H^{*+\text{dim} G}(LBG;\mathbb{F}_2)$. So by Lemma~\ref{lem:iso BV implique unite et image de Delta},
the BV-algebras $HH^*(H_*(G);H_*(G))$ and $H^{*+\text{dim} G}(LBG;\mathbb{F}_2)$ are not isomorphic.

The BV-algebra $HH^*(H_*(SO(3)),H_*(SO(3)))$ is explicitly computed in the proof of Theorem
\ref{iso BV with Hochschild when no steenrod operation}
and is isomorphic to the tensor product of algebras
$\Lambda(x_{-2},x_{-1})\otimes \mathbb{F}_2[y_2,y_3]$
with $\Delta(x_{-2}y_3)=1$, $\Delta(x_{-2}y_2)=0$, $\Delta(x_{-1}y_2)=1$, $\Delta(x_{-1}y_3)=0$,
and $\Delta$ is trivial on $\Lambda(x_{-2},x_{-1})\otimes 1$ and on
 $1\otimes \mathbb{F}_2[y_2,y_3]$.
 The BV-algebra $H^{*+3}(LBSO(3);\mathbb{F}_2)\cong \Lambda(u_{-2},u_{-1})\otimes \mathbb{F}_2[v_2,v_3]$ is explicited by Theorem~\ref{thm: BV-algebra sur algebre commutative libre pour SO3}.
 
Let $\varphi:\Lambda(x_{-2},x_{-1})\otimes \mathbb{F}_2[y_2,y_3]
\rightarrow \Lambda(u_{-2},u_{-1})\otimes \mathbb{F}_2[v_2,v_3]$ be any morphism of graded algebras.
Since $\Lambda(x_{-2},x_{-1})\otimes \mathbb{F}_2[y_2,y_3]$ and
$\Lambda(u_{-2},u_{-1})\otimes \mathbb{F}_2[v_2,v_3]$ are of the same finite dimension in each degree, $\varphi$ is an isomorphism iff $\varphi$ is surjective.
By Lemma~\ref{surjectivity morphism of algebras}, $\varphi$ is surjective iff $Q(\varphi)$
is surjective.
Therefore if $\varphi$ is an isomorphism of algebras iff
$$\varphi(x_{-2})=u_{-2},$$
$$\varphi(x_{-1})=u_{-1}+\varepsilon u_{-1}u_{-2}v_2,$$
$$\varphi(y_2)=v_2+a u_{-2}v_2^2+bu_{-1}u_{-2}v_2v_3+cu_{-1}v_3,$$
$$\varphi(y_3)=v_3+ \alpha u_{-2}v_2v_3+\beta u_{-1}u_{-2}v_3^2+\gamma u_{-1}u_{-2}v_2^3+\delta u_{-1}v_2^2$$
where $\varepsilon$, $a$, $b$, $c$, $\alpha$, $\beta$, $\gamma$, $\delta$ are 8 elements
of $\mathbb{F}_2$.
Since $(u_{-2})^2=(u_{-1}+\varepsilon u_{-1}u_{-2}v_2)^2=0$,
the above 4 formulas define always a morphism $\varphi$ of algebras.

By the Poisson rule, a morphism of algebras between Gerstenhaber algebras is a morphism
of Gerstenhaber algebras iff the brackets are compatible on the generators.

Note that modulo $2$, in a BV-algebra, for any elements $z$ and $t$,
$\{z+t,z+t\}=\{z,z\}+\{t,t\}$ and $\{z,z\}=\Delta(z^2)$.
Therefore it is easy to check
that
$
\varphi(\{x_{-2},x_{-2}\})=0=\{\varphi (x_{-2}),\varphi(x_{-2})\}
$,
$
\varphi(\{x_{-1},x_{-1}\})=0=\{\varphi (x_{-1}),\varphi(x_{-1})\}
$,
$
\varphi(\{y_2,y_2\})=0=\{\varphi (y_2),\varphi(y_2)\}
$ and 
$
\varphi(\{y_3,y_3\})=0=\{\varphi (y_3),\varphi(y_3)\}.
$

Note that $\Delta\varphi(x_{-1})=\varepsilon u_{-2}$, $\Delta\varphi(x_{-2})=0$, 
$\Delta\varphi(y_2)=(b+c)(u_{-2}v_3+u_{-1}v_2)$
and $\Delta\varphi(y_3)=\alpha u_{-1}v_3+\alpha v_2+(\alpha+\gamma)u_{-2}v_2^2
+\alpha u_{-1}u_{-2}v_2v_3$.

Therefore
$
\varphi(\{x_{-2},y_2\})=0
$,
$
\{\varphi (x_{-2}),\varphi(y_2)\}=(1+c)u_{-1}+(b+c)u_{-1}u_{-2}v_2
$,

$
\varphi(\{x_{-1},y_2\})=1
$,
$
\{\varphi (x_{-1}),\varphi(y_2)\}=1+(1+\varepsilon)u_{-2}v_2+(\varepsilon c+1+b+c)u_{-1}u_{-2}v_3
$,

$
\varphi(\{x_{-2},x_{-1}\})=0=\{\varphi (x_{-2}),\varphi(x_{-1})\}
$,

$
\varphi(\{x_{-2},y_3\})=1
$,
$
\{\varphi (x_{-2}),\varphi(y_3)\}=1+(1+\alpha)u_{-2}v_2+(1+\alpha)u_{-1}u_{-2}v_3
$,

$
\varphi(\{x_{-1},y_3\})=0
$,
$
\{\varphi (x_{-1}),\varphi(y_3)\}=
(1+\alpha+\varepsilon+\alpha)u_{-1}v_2
+(\varepsilon+1+\alpha+\varepsilon)u_{-2}v_3+(\varepsilon\delta+\alpha+\gamma+\varepsilon\alpha)u_{-1}u_{-2}v_2^2
$,

$
\varphi(\{y_2,y_3\})=0
$,
\begin{multline*}
\{\varphi (y_2),\varphi(y_3)\}=
\Delta\varphi(y_2)\varphi(y_3)+\Delta(\varphi(y_2)\varphi(y_3))+\varphi(y_2)\Delta\varphi(y_3)\\
=(b+c)(u_{-2}v_3^2+u_{-1}v_2v_3+(\alpha+\delta)u_{-1}u_{-2}v_2^2v_3)\\
+\Delta\left((a+\alpha)u_{-2}v_2^2v_3+(b+c\alpha+\beta)u_{-1}u_{-2}v_2v_3^2+\delta u_{-1}v_2^3\right)+\varphi(y_2)\Delta\varphi(y_3)\\
=(a+\alpha+\delta+\alpha)v_2^2+(a+\alpha+\delta+\alpha+\gamma+a\alpha)u_{-2}v_2^3\\
+((b+c)(\alpha+\delta)+a+\alpha+\delta+\alpha+a\alpha+b\alpha+c\alpha+c\gamma)u_{-1}u_{-2}v_2^2v_3\\
+(b+c+\alpha+c\alpha)u_{-1}v_2v_3
+(b+c+b+c\alpha+\beta)u_{-2}v_3^2.
\end{multline*}
Therefore, by symmetry of the Lie brackets, $\varphi$ is a morphism
of Gerstenhaber algebras iff
$\varepsilon=b=c=\alpha=1$, $\beta=0$ and $a=\gamma=\delta$.
Conclusion: we have found two isomorphisms of Gerstenhaber algebras between
$H^{*+3}(LBSO(3);\mathbb{F}_2)$ and $HH^*(H_*(SO(3)),H_*(SO(3)))$.
\end{proof}
\section{Review of~\cite{Chataur-Menichi:stringclass} with signs corrections}\label{review of chataurmenichi}
In this section, we review the results of Chataur and the second author in~\cite{Chataur-Menichi:stringclass}.
And we correct a sign mistake.

\noindent{\bf Integration along the fibre in homology with corrected sign.} Let $F \to E \stackrel{p}{\to} B$ be an oriented fibration with $B$ path-connected; that is, the homology $H_*(F; \K)$ is concentrated in degree less than or equal to $n$, 
$\pi_1(B)$ acts on $H_n(F; \K)$ trivially and $H_n(F; \K)\cong \K$. In what follows, we write $H_*(X)$ for $H_*(X; \K)$. 
We choose a generator $\omega$ of $H_n(F; \K)$, which is called 
an orientation class. Then the integration along the fibre $p_!^\omega : H_*(B) \to H_{*+n}(E)$ is defined by the composite 
$$
H_s(B) \stackrel{\eta}{\to} H_s(B)\otimes H_n(F) = E_{s, n}^2 \twoheadrightarrow E_{s, n}^\infty = F^s/F^{s-1}= F^s \subset H_{s+n}(E),  
$$
where $\eta$ sends the $x \in H_s(B)$ to the element $(-1)^{sn}x\otimes \omega \in H_s(B)\otimes H_n(F)$ 
and $\{F^l\}_{l\geq 0}$ denotes the filtration of the Leray-Serre spectral sequence 
$\{E_{*,*}^r, d^r\}$ of the fibration $F \to E \stackrel{p}{\to} B$.
This Koszul sign $(-1)^{sn}$ does not appear in the usual definition of integration along the fibre recalled in~\cite[2.2.1]{Chataur-Menichi:stringclass}.

\noindent{\bf Products:} Let $F' \to E' \stackrel{p'}{\to} B'$ be another oriented fibration with orientation class $\omega' \in H_{n'}(F')$.
We will choose $\omega \otimes \omega' \in H_{n+n'}(F\times F')$ as an orientation class of the fibration 
$F\times F' \to E\times E' \stackrel{p\times p'}{\to} B\times B'$. 
By~\cite[3 Theorem, page 493]{Spanier:livre},
the cross product $\times$ induces a morphism of spectral sequences between the tensor product
of the Serre spectral sequences associated to $p$ and $p'$ and the Serre spectral sequence
associated to $p\times p'$.
Therefore the interchange on $H_*(B) \otimes H_n(F)\otimes H_*(B') \otimes H_{n'}(F')$ between the orientation class $\omega \in H_n(F)$ and elements in $H_*(B')$ yields the formula given (without proof) in~\cite[section 2.3]{Chataur-Menichi:stringclass}
$$
(p\times p')_!^{\omega \times \omega'}(a\times b) = (-1)^{\vert\omega'\vert\vert a \vert}p_!^\omega(a)\times p'^{\omega'}_!(b). 
\eqnlabel{add-1}.
$$
Note that with the usual definition of integration along the fibre recalled in~\cite[2.2.1]{Chataur-Menichi:stringclass},
the Koszul sign $(-1)^{\vert\omega'\vert\vert a \vert}$ must be replaced by the awkward sign
$(-1)^{\vert\omega\vert\vert b\vert}$.
Therefore there is a sign mistake in~\cite[section 2.3]{Chataur-Menichi:stringclass}.

\noindent{\bf Integration along the fibre in cohomology with corrected sign.}
Let $F\buildrel{incl}\over\hookrightarrow E
\buildrel{p}\over\twoheadrightarrow B$ be an oriented fibration
with orientation $\tau:H^n(F)\rightarrow\K$.
By definition, $p^!_\tau:H^{s+n}(E)\rightarrow H^s(B)$ is the composite
$$
H^{s+n}(E)\twoheadrightarrow E_\infty^{s,n}\subset E_2^{s,n}=H^s(B)\otimes H^n(F)
\buildrel{id\otimes \tau}\over\rightarrow H^s(B)
$$
where $(id\otimes \tau)(b\otimes f)=(-1)^{n\vert b\vert} b\tau(f)$.
This Koszul sign $(-1)^{n\vert b\vert}$ does not appear in the usual definition of integration along the fibre recalled in~\cite[p. 268]{Bredon:Sheaf}.

By~\cite[IV.14.1]{Bredon:Sheaf},
$$
p^!_\tau(H^*(p)(\beta)\cup \alpha) = (-1)^{\vert \beta\vert n}\beta\cup p^!_\tau(\alpha)
$$
for $\alpha\in H^*(E)$ and $\beta\in H^*(B)$.
This means that the degree $-n$ linear map $p^!_\tau:H^*(E)\rightarrow H^{*-n}(B)$
is a morphism of left $H^*(B)$-modules 
in the sense that $f(xm)=(-1)^{\vert f \vert \vert x\vert} xf(m)$
as quoted in \cite[p. 44]{Felix-Halperin-Thomas:ratht}. 

\noindent{\bf Example: trivial fibrations.}
Let $\omega\in H_n(F; \K)$ be a generator.
Define the orientation $\tau:H^n(F)\rightarrow\K$ as the image of $\omega$
by the natural isomorphism of the homology into its double dual,
$\psi:H_n(F; \K)\rightarrow\text{Hom}(H^n(F; \K),\K)$.
Explicitly, $\tau(f)=(-1)^{n\vert f\vert}<f,\omega>$ where $<\;,\;>$ is the Kronecker bracket.

Let $p_1:B\times F\twoheadrightarrow B$ be the projection on the first factor.
Then for any $f\in H^*(F)$ and $b\in H^*(B)$, $p^!_{1\tau}(b\times f)=(-1)^{\vert f\vert\vert b\vert} b\tau(f)$.
Let $p_2:F\times B\twoheadrightarrow B$ be the projection on the second factor.
Since $p_2$ is the composite of $p_1$ and the exchange homeomorphism, by naturality
of integration along the fibre,
$$p^!_{2\tau}(f\times b)=p^!_{1\tau}((-1)^{\vert f\vert\vert b\vert}b\times f)=b\tau(f)=(-1)^{n\vert f\vert}<f,\omega>b.$$

\noindent{\bf Main coTheorem of~\cite{Chataur-Menichi:stringclass} with signs.}
The main theorem of~\cite{Chataur-Menichi:stringclass} states that $H_*(LX)$ is a $d$-dimensional (non-unital non
co-unital) homological conformal field theory: that is  $H_*(\mathcal{L}X)$ is an algebra over the tensor product of graded linear props
$$
\bigoplus_{F_{p+q}}\text{det}H_1(F,\partial_{in};\mathbb{Z})^{\otimes d}\otimes_\mathbb{Z} H_*(BDiff^+(F,\partial);\K).
$$
See \cite[Sections 3 and 11]{Chataur-Menichi:stringclass} for the definition of this prop. The prop $\text{det}H_1(F,\partial_{in};\mathbb{Z})$ manages the degree shift
and the sign of each operation.
In~\cite{Chataur-Menichi:stringclass}, Chataur and the second author did not pay attention to this prop $\text{det}H_1(F,\partial_{in};\mathbb{Z})$
(\cite[p. 120]{BGNX:Stringstackspublie} neither, it seems).
Therefore, in order to get the signs correctly,
we need to review all the results of~\cite{Chataur-Menichi:stringclass} by taking this prop into account. 
Explicitly, we have maps
$$
\nu_*(F_{q+p}):\text{det}H_1(F_{q+p},\partial_{in};\mathbb{Z})^{\otimes d}\otimes_\mathbb{Z} H_*(BDiff^+(F_{q+p},\partial))\otimes H_*(LX)^{\otimes q}\rightarrow H_*(LX)^{\otimes p}
$$
$s\otimes a\otimes v\mapsto \nu_*(F_{q+p})^{s\otimes a}(v)$.

Therefore (Compare with \cite[Section 6.3]{Chataur-Menichi:stringclass}), its dual $H^*(LX)$ is an algebra over the opposite
prop
$$
\bigoplus_{F_{p+q}}\text{det}H_1(F,\partial_{in};\mathbb{Z})^{op\otimes d}\otimes_\mathbb{Z} H_*(BDiff^+(F,\partial))^{op}.
$$
which is isomorphic to the prop
$$
\bigoplus_{F_{p+q}}\text{det}H_1(F,\partial_{out};\mathbb{Z})^{\otimes d}\otimes_\mathbb{Z} H_*(BDiff^+(F,\partial)).
$$
since $\text{det}H_1(F_{p+q},\partial_{out};\mathbb{Z})=\text{det}H_1(F_{q+p},\partial_{in};\mathbb{Z})$
and $Diff^+(F_{p+q},\partial)=Diff^+(F_{q+p},\partial)$.
Explicitly, the degree $0$ map
$$
\nu^*(F_{p+q}):\text{det}H_1(F_{q+p},\partial_{in};\mathbb{Z})^{\otimes d}\otimes_\mathbb{Z} H_*(BDiff^+(F_{q+p},\partial))\otimes H^*(LX)^{\otimes p}\rightarrow H^*(LX)^{\otimes q}
$$
send the element
$s\otimes a\otimes \alpha$ to 
$$\nu^*(F_{p+q})^{s\otimes a}(\alpha):= ^t(\nu_*(F_{q+p})^{s\otimes a})(\alpha)=
(-1)^{\vert\alpha\vert(\vert s\vert+\vert a\vert)}\alpha\circ \nu_*(F_{q+p})^{s\otimes a}.$$
Note that here, we have defined the transposition of a map $f$ as
$$ ^t f(\alpha)=(-1)^{\vert\alpha\vert\vert f\vert}\alpha\circ f.$$

This means the following five propositions.
\begin{prop}\label{prop:equivalent}(Compare with~\cite[Proposition 24]{Chataur-Menichi:stringclass})
Let $F$ and $F'$ be two cobordisms with same incoming boundary and same outgoing boundary.
Let $\phi:F\rightarrow F'$ be an orientation preserving diffeomorphism, fixing the boundary
(i. e. an equivalence between the two cobordisms $F$ and $F'$).
Let $c_\phi:Diff^+(F,\partial)\rightarrow Diff^+(F',\partial)$ be the isomorphism
of groups, mapping $f$ to $\phi\circ f\circ\phi^{-1}$.
Then 
for $s\otimes a\in \text{det}H_1(F,\partial_{out};\mathbb{Z})^{\otimes d}\otimes_\mathbb{Z} H_*(BDiff^+(F,\partial))$,
$$\nu^{*s\otimes a}(F)=\nu^{*\text{det}H_1(\phi,\partial_{out};\mathbb{Z})^{\otimes d}(s)\otimes H_*(Bc_\phi)(a)}(F').$$
\end{prop}
\begin{rem}\label{equivalencesamecobordism}
In Proposition~\ref{prop:equivalent}, suppose that $F=F'$.
By a variant of~\cite[Proposition 19]{Chataur-Menichi:stringclass},
$H_1(\phi,\partial_{out};\mathbb{Z})$ is of determinant $+1$.
Since the natural surjection $Diff^+(F,\partial))\buildrel{\simeq}\over\rightarrow
\pi_0(Diff^+(F,\partial))$ is a homotopy equivalence~\cite{Earle-Schatz} and
$\pi_0(c_\phi)$ is the conjugation by the isotopy class of $\phi$, $H_*(Bc_\phi)$ is
the identity. So the conclusion of Proposition~\ref{prop:equivalent} is just
 $\nu^{*s\otimes a}(F)=\nu^{*s\otimes a}(F)$.
\end{rem}
Using Proposition~\ref{prop:equivalent}, it is enough to define the operation
$\nu^*(F)$ for a set of representatives $F$ of oriented classes of cobordisms
(therefore the direct sum over a set $\oplus_F$ in the above definition of the prop has a meaning).
Conversely, if $\nu^*(F)$ is defined for a cobordism $F$ then using
Proposition~\ref{prop:equivalent}, we can define $\nu^*(F')$ for any equivalent cobordism $F'$ using an equivalence of cobordism $\phi:F\rightarrow F'$.
Two equivalences of cobordism $\phi,\phi':F\rightarrow F'$ define the same operation $\nu^*(F')$ since
$\text{det}H_1(\phi)\circ \text{det}H_1(\phi')^{-1}=\text{det}H_1(\phi\circ\phi'^{-1})
=Id$ and $H_*(Bc_\phi)\circ H_*(Bc_{\phi'})^{-1}=H_*(Bc_{\phi\circ{\phi'}^{-1}})=Id$
by Remark~\ref{equivalencesamecobordism}.
\begin{prop}\label{prop:disjointunion}(Compare with~~\cite[Proposition 30 Monoidal]{Chataur-Menichi:stringclass})
Let $F$ and $F'$ be two cobordisms. 
For $s\otimes a\in \text{det}H_1(F,\partial_{out};\mathbb{Z})^{\otimes d}\otimes_\mathbb{Z} H_*(BDiff^+(F,\partial))$ and $t\otimes b\in \text{det}H_1(F',\partial_{out};\mathbb{Z})^{\otimes d}\otimes_\mathbb{Z} H_*(BDiff^+(F',\partial))$
$$
\nu^{*(s\otimes t)\otimes (a\otimes b)}(F\coprod F')
 = 
(-1)^{ \vert t\vert \vert a\vert }
\nu^{*s\otimes a}(F)\otimes  \nu^{*t\otimes b}(F').
$$
\end{prop}
\begin{prop}\label{prop:gluing} (Compare with~\cite[Proposition 31 Gluing]{Chataur-Menichi:stringclass})
Let $F_{p+q}$ and $F_{q+r}$ be two composable cobordisms.
Denote by $ F_{q+r}\circ F_{p+q}$ the cobordism obtained by gluing.
For $s_1\otimes  m_1\in \text{det}H_1(F_{p+q},\partial_{out};\mathbb{Z})^{\otimes d}\otimes_\mathbb{Z} H_*(BDiff^+(F_{p+q},\partial))$ and $s_2\otimes  m_2\in \text{det}H_1(F_{q+r},\partial_{out};\mathbb{Z})^{\otimes d}\otimes_\mathbb{Z} H_*(BDiff^+(F_{q+r},\partial))$
$$
\nu^{*s_2\otimes m_2}(F_{ q+r})
\circ
\nu^{*s_1\otimes m_1}( F_{p+q})
 = 
(-1)^{ \vert m_2\vert \vert s_1\vert }
\nu^{*(s_2\circ s_1)\otimes ( m_2\circ m_1)}(F_{ q+r}\circ   F_{p+q}).
$$
Here
$$
\circ: H_*(BDiff^+(F_{q+r},\partial))\otimes H_*(BDiff^+(F_{p+q},\partial))\rightarrow H_*(BDiff^+(F_{ q+r}\circ F_{p+q},\partial))
$$ mapping $m_2\otimes m_1$ to $m_2\circ m_1$ is induced by the gluing of $F_{p+q}$ and $F_{q+r}$.
\end{prop}
As noted by~\cite{Hepworth-Lahtinen:stringclassifying} with their notion of
$h$-graph cobordism, \cite{Chataur-Menichi:stringclass} never used the smooth structure
of the cobordisms. So in fact, our cobordisms are topological.
Therefore the cobordism $ F_{q+r}\circ F_{p+q}$ obtained by gluing
is canonically defined~\cite[1.3.2]{Kock:Frob2TQFT}.
Note that by~\cite{Earle-Schatz} and~\cite{Hamstrom:homeosurface}, the inclusion
$
Diff^{+}(F,\partial)\buildrel{\approx}\over\hookrightarrow Homeo^+(F,\partial)
$
is a homotopy equivalence since $\pi_0(Diff^{+}(F,\partial))\cong\pi_0(Homeo^+(F,\partial))$~\cite[p. 45]{Farb-Margalit:primermcg}.

\begin{prop}\label{prop:identity}(Compare with~\cite[Corollary 28 i) identity]{Chataur-Menichi:stringclass}) 
Let $id_1\in \text{det}H_1(F_{0,1+1},\partial_{out};\mathbb{Z})$
and $id_1\in H_0(BDiff^+(F_{0,1+1},\partial))$
be the identity morphisms of the object $1$ in the two props.
Then 
$$\nu^{*id_1^{\otimes d}\otimes id_1}(F_{0,1+1})=Id_{H^*(LX)}.$$
\end{prop}
\begin{prop}\label{prop:symmetry}(Compare with~\cite[Corollary 28 ii) symmetry]{Chataur-Menichi:stringclass})
Let $C_\phi$ be the twist cobordism of $S^1\coprod S^1$.
Let $\tau\in \text{det}H_1(C_\phi,\partial_{out};\mathbb{Z})$,
$\tau\in H_0(BDiff^+(C_\phi,\partial))$ and $\tau\in \text{End}(H^*(LX)^{\otimes 2})$
be the exchange isomorphisms of the three props. Then
$$
\nu^{*\tau^{\otimes d}\otimes \tau}(C_\phi)=\tau.$$
\end{prop}
Let $F$ be a cobordism.
Let $\iota_F$ be the generator of $H_0(BDiff^+(F,\partial))$ which is represented by the connected component of $BDiff^+(F,\partial)$.
We may write $\iota$ instead of $\iota_F$ for simplicity.
If $\chi(F)=0$ then $H_1(F,\partial_{out};\mathbb{Z})=\{0\}$
has an unique orientation class which correspond to the generator
$1\in \text{det}H_1(F,\partial_{out};\mathbb{Z})=\Lambda^{-\chi(F)}H_1(F,\partial_{out};\mathbb{Z})=\mathbb{Z}$.

The identity morphim $id_1$ and the exchange isomorphism $\tau$ of the prop
$\text{det}H_1(F,\partial_{out};\mathbb{Z})$ correspond to these unique
orientation classes of $H_1(F_{0,1+1},\partial_{out};\mathbb{Z})$
and $H_1(C_\phi,\partial_{out};\mathbb{Z})$.

The identity morphim $id_1$ and the exchange isomorphism $\tau$ of the prop
$H_*(BDiff^+(F,\partial))$ are just $\iota_{F_{0,1+1}}$ and $\iota_{C_\phi}$.

\section{Commutativity and associativity of the dual to the Loop coproduct}
\begin{thm}\label{Associativity-Commutativity}
Let $d\geq 0$.
Let $H^*$ (upper graded) be an algebra over the (lower graded) prop
$$\text{det}H_1(F,\partial_{out};\mathbb{Z})^{\otimes d}\otimes_\mathbb{Z} H_0(BDiff^+(F,\partial)).$$
Let $s\in\text{det}H_1(F_{0,2+1},\partial_{out};\mathbb{Z})^{\otimes d}$
be a chosen orientation.
Let $$Dlcop:=\nu^{*s\otimes\iota}(F_{0,2+1}).$$
Let $m$ be 
the product
defined by 
$$m(a\otimes b)=(-1)^{d(i-d)}Dlcop(a\otimes b)$$ 
for $a\otimes b \in H^i\otimes H^j$.
Let $\mathbb{H}^*:=H^{*+d}$.
 Then  $(\mathbb{H}^*,m)$ is a graded associative and commutative algebra.
\end{thm}
\begin{proof}
Using Propositions~\ref{prop:disjointunion},~\ref{prop:gluing} and~\ref{prop:identity},
$$
Dlcop\circ (Dlcop\otimes 1)=
\nu^{*s\circ (s\otimes id_1)\bigotimes \iota\circ(\iota\otimes id_1)}(F_{0,2+1}\circ (F_{0,2+1}\coprod F_{0,1+1}))\text{ and}
$$
$$
Dlcop\circ (1\otimes Dlcop)=
\nu^{*s\circ (id_1\otimes s)\bigotimes \iota\circ(id_1\otimes\iota)}
(F_{0,2+1}\circ (F_{0,1+1}\coprod F_{0,2+1})).
$$
The cobordisms $F_{0,2+1}\circ (F_{0,2+1}\coprod F_{0,1+1})$ and
$F_{0,2+1}\circ (F_{0,1+1}\coprod F_{0,2+1})$ are equivalent.
When you identified them,
$\iota\circ (\iota\otimes id_1)=\iota\circ (id_1\otimes \iota)$.
Also $F_{0,2+1}\circ C_\phi=F_{0,2+1}$ and $\iota\circ\tau=\iota$.
 
Let $\beta \in\text{det}H_1(F_{0,2+1},\partial_{out};\mathbb{Z})$ the generator such
that $\beta^{\otimes d}=s$.
The compositions of the $\mathbb{Z}$-linear prop $\text{det}H_1(F,\partial_{out};\mathbb{Z})$ are isomorphisms. Therefore, they send generators to generators.
Moreover $\text{det}H_1(F,\partial_{out};\mathbb{Z}):=\Lambda^{-\chi(F)} H_1(F,\partial_{out};\mathbb{Z})$ is an abelian group on
a single generator of lower degree $-\chi(F)$.
So $\beta\circ (\beta\otimes id_1)=\varepsilon_{ass} \beta\circ (id_1\otimes\beta)$
and $\beta\circ\tau= \varepsilon_{com}\beta$
for given signs $\varepsilon_{ass}$ and $\varepsilon_{com}\in\{-1,1\}$.
Therefore
$$
s\circ (s\otimes id_1)=
\beta^{\otimes d}\circ (\beta\otimes id_1)^{\otimes d}=
(-1)^{\frac{d(d-1)}{2}\vert \beta\vert^2}
\left(\beta\circ(\beta\otimes id_1)\right)^{\otimes d}
=
\varepsilon_{ass}^d s\circ (id_1\otimes s)
$$
and
$
s\circ \tau=\beta^{\otimes d}\circ\tau^{\otimes d}=(\beta\circ\tau)^{\otimes d}=
(\varepsilon_{com}\beta)^{\otimes d}=\varepsilon_{com}^{ d}\beta^{\otimes d}
=\varepsilon_{com}^{ d}s
$.

Therefore, by proposition~\ref{prop:equivalent}
$$Dlcop\circ (Dlcop\otimes 1)=
\varepsilon_{ass}^d Dlcop\circ (1\otimes Dlcop)$$
and $Dlcop\circ \tau=\varepsilon_{com}^d Dlcop$.
This means that for $a,b,c\in H^*(LX)$,
$$
m(m(a\otimes b)\otimes c)=\varepsilon_{ass}^d (-1)^{d}m(a\otimes m(b\otimes c)
$$
and 
$
m(b\otimes a)=\varepsilon_{com}^d(-1)^{(\vert a\vert-d)(\vert b\vert-d)+d}m(a\otimes b))
$
since $$m(m(a\otimes b)\otimes c)
=(-1)^{d\vert b\vert+d}Dlcop\circ (Dlcop\otimes 1)(a\otimes b\otimes c)$$
and $$m(a\otimes m(b\otimes c))=(-1)^{d(\vert a\vert+\vert b\vert)}Dlcop (a\otimes Dlcop(b\otimes c))
=(-1)^{d\vert b\vert}Dlcop\circ (1\otimes Dlcop)(a\otimes b\otimes c).$$

In~\cite[Proof of Proposition 21]{Godin:higherstring}, Godin has shown geometrically
that $\varepsilon_{ass}=-1$ for the prop $\text{det}H_1(F,\partial_{in};\mathbb{Z})$.
To determine the signs $\varepsilon_{ass}$ and $\varepsilon_{com}$ for
the prop $\text{det}H_1(F,\partial_{out};\mathbb{Z})$, we prefer to use
our computations of $m$.

Consider a particular connected compact Lie group $G$ of a particular dimension $d$
and a particular field $\K$ of characteristic different from $2$
such that $H^*( BG;\K)$ is polynomial,
for example $G=(S^1)^d$ or $\K=\mathbb{Q}$.
Then $H^*(LBG;\mathbb{Q})$ is an algebra over our prop
and we can apply (2) of Theorem~\ref{dlcop given by formulas}
or Corollary~\ref{cor:unital p odd}.
Taking  $a=x_1\dots x_N$, $b=1$ and $c=x_1\dots x_N$,
we obtain 
$1=\varepsilon_{ass}^d(-1)^{d}$ and $1=\varepsilon_{com}^d(-1)^d$.
So if we have chosen $d$ odd, $\varepsilon_{ass}=\varepsilon_{com}=-1$
and $m$ is associative and graded commutative.
\end{proof} 
\begin{rem}\label{rem:HCFT}
When $d$ is even, 
the $d$-th power of the prop $\text{det}H_1(F,\partial_{in};\mathbb{Z})$ is isomorphic to the $d$-th power of the trivial prop with a degree 
shift $\chi(F)$.

More precisely, let $\mathcal{P}$ the prop
such that
$$
\mathcal{P}(p,q):=\bigoplus_{F_{p+q}}s^{-\chi(F_{p+q})}\mathbb{Z},
$$
$s^{-\chi(F')}1\circ s^{-\chi(F)}1=s^{-\chi(F'\circ F)}1$
and
$s^{-\chi(F)}1\otimes s^{-\chi(F')}1=s^{-\chi(F\coprod F')}1$.
This prop $\mathcal{P}$ is the the trivial prop with a degree 
shift $\chi(F)$.

For any cobordism $F$, let
$\Theta_F:s^{-\chi(F)}\mathbb{Z}\rightarrow \text{det}H_1(F,\partial_in;\mathbb{Z})$
be an chosen isomorphism.
Then $\Theta_F^{\otimes d}:\mathcal{P}^{\otimes d}\rightarrow \text{det}H_1(F,\partial_{in};\mathbb{Z})^{\otimes d}$ is an isomorphim of props if $d$ is even.
This prop $\mathcal{P}^{\otimes d}$ is the $d$-th power of the trivial prop with a degree 
shift $\chi(F)$ and is not isomorphic to the trivial prop with a degree shift
$-d\chi(F)$.
\end{rem}
\begin{proof}
The following upper square commutes always, while the following lower square commutes if $d$ is even.

\xymatrix{
(s^{-\chi(F')}\mathbb{Z})^{\otimes d}\otimes (s^{-\chi(F)}\mathbb{Z})^{\otimes d}
\ar[r]^-{\Theta_{F'}^{\otimes d}\otimes\Theta_{F}^{\otimes d}}\ar[d]^\tau
&\text{det}H_1(F',\partial_{in};\mathbb{Z})^{\otimes d}
\otimes \text{det}H_1(F,\partial_{in};\mathbb{Z})^{\otimes d}\ar[d]^\tau\\
(s^{-\chi(F')}\mathbb{Z}\otimes s^{-\chi(F)}\mathbb{Z})^{\otimes d}
\ar[r]_-{\left(\Theta_{F'}\otimes\Theta_{F}\right)^{\otimes d}}\ar[d]_{\circ^{\otimes d}}
&\left(\text{det}H_1(F',\partial_{in};\mathbb{Z})
\otimes \text{det}H_1(F,\partial_{in};\mathbb{Z})\right)^{\otimes d}\ar[d]^{\circ^{\otimes d}}\\
(s^{-\chi(F'\circ F)}\mathbb{Z})^{\otimes d}
\ar[r]_-{\left(\Theta_{F'\circ F}\right)^{\otimes d}}
&\text{det}H_1(F'\circ F,\partial_{in};\mathbb{Z})^{\otimes d}
}

Replacing $\circ$ by the tensor product $\otimes$ of props, we have proved that
$\Theta_F^{\otimes d}$ is an  isomorphism of props if $d$ is even.
\end{proof}

Observe that the dual of the loop coproduct $Dlcop$ on $H^*(LX)$ satisfies the same commutative and associative formulae as those of  the Chas-Sullivan loop product on the loop homology of $M$. See \cite[Remark 3.6]{tamanoi:capproducts} or~\cite[Proposition 2.7]{KuriMeniNaito:DerivedEMSS}.
So we wonder if the prop $\text{det}H_1(F,\partial_{out};\mathbb{Z})$
is isomorphic to the prop $\text{det}H_1(F,\partial_{in};\mathbb{Z})$.
\begin{cor} \label{cor:shifted_alg}
Let $X$ be a simply connected space such that $H_*(\Omega X;\K)$ is finite dimensional.
 The shifted cohomology ${\mathbb H}^*(LX) := H^{*+d}(LX)$ 
is a graded commutative, associative algebra endowed with the product $m$
defined by 
$$m(a\otimes b)=(-1)^{d(i-d)}Dlcop(a\otimes b)$$ 
for $a\otimes b \in H^i(LX)\otimes H^j(LX)$. 
\end{cor}

\section{The Batalin-Vilkovisky identity}\label{section:BV-identity}
For any simple closed curve $\gamma$ in a cobordism $F$,
let us denote by $\overline{\gamma}$ the image of the Dehn twist $T_\gamma$ by the
hurewicz map $\Theta$ 
$$
\pi_0(Diff^+(F,\partial))\maprightud{\partial^{-1}}{\cong} 
\pi_1(BDiff^+(F,\partial))\maprightud{\Theta}{} H_1(BDiff^+(F,\partial)).
$$
In this section, we prove the following theorem.
\begin{thm}\label{dualHCFTisBV}
Let $H^*$ be an algebra over the prop
$$\text{det}H_1(F,\partial_{out};\mathbb{Z})^{\otimes d}\otimes_\mathbb{Z} H_*(BDiff^+(F,\partial)).$$
Consider the the graded associative and commutative algebra $(\mathbb{H}^*,m)$
given by Theorem~\ref{Associativity-Commutativity}.
Let $\alpha$ be a closed curve in the cylinder $F_{0,1+1}$ parallel to one of the boundary components.
Let $$
\Delta=\nu^{*id_1\otimes\overline{\alpha}}(F_{0,1+1}).$$
Then 
 $(\mathbb{H}^*,m,\Delta)$ is a Batalin-Vilkovisky algebra.
\end{thm}
In the case $d=0$, Wahl~\cite[Rem 2.2.4]{thesedewahl} or Kupers~\cite[4.1, page 158]{Kupers:stringop}
give an incomplete proof that we complete. Moreover, we pay attention to signs.

The shifted cohomology algebra $({\mathbb H}^*,m)$ equipped with the operator $\Delta$
is a BV-algebra if and only if $\Delta\circ\Delta=0$ and if the Batalin-Vilkovisky identity holds;  
that is,  for any elements $a$, $b$ and $c$ in 
${\mathbb H}^*$, 
\begin{eqnarray*}
\Delta(a\cdot b\cdot c)&=& \Delta(a\cdot b)\cdot c + (-1)^{\parallel a \parallel}a\cdot \Delta (b\cdot c) + 
(-1)^{\parallel b \parallel \parallel a \parallel+ \parallel b \parallel}b\cdot\Delta(a\cdot c) \\ 
&&- \Delta(a)\cdot b\cdot c -(-1)^{\parallel a \parallel}a\cdot \Delta(b)\cdot c 
- (-1)^{\parallel a \parallel+ \parallel b \parallel}a\cdot b\cdot\Delta(c), 
\end{eqnarray*}
where $\alpha\cdot \beta = m(\alpha \otimes \beta)$ and $\parallel \alpha \parallel$ stands for the degree of 
an element $\alpha$ in ${\mathbb H}^*$, namely $\parallel \alpha \parallel = \vert \alpha \vert -d$.  

Since $BDiff^+(F_{0,1+1})$ is $B\mathbb{Z}$,
 $\overline{\alpha}\circ\overline{\alpha}\in H_2(BDiff^+(F_{0,1+1}))=\{0\}$.
 Therefore
$\Delta\circ\Delta=\pm \nu^{*id_1\otimes\overline{\alpha}\circ\overline{\alpha}}(F_{0,1+1})=0
$

The BV-identity will arise up to signs from the lantern relation
(~\cite[Rem 2.2.4]{thesedewahl} or \cite[4.1, page 158]{Kupers:stringop}):
\begin{prop} \cite{Johnson:homeosurface}\cite[Section 5.1]{Farb-Margalit:primermcg} \label{prop:lantern}
Let $a_1, ..., a_4$ and $x, y, z$ be the simple closed curves described in \cite[Figure 6.89, page 161]{Kupers:stringop}. 
Then one has 
$$
T_{a_1}T_{a_2}T_{a_3}T_{a_4} = T_xT_yT_z
$$
in the mapping class group of $F_{0, 4}$, where $T_\gamma$ denotes the Dehn twist around 
a simple closed curve $\gamma$ in the surface.  
\end{prop} 

In order to prove Theorem \ref{thm:B-V_algebra}, 
we  represent each term of the B-V identity in terms of elements of the prop  with 
a HCF theoretical way: this means using the horizontal (coproduct) composition $\otimes$ and the vertical composition $\circ$ on the prop.
We start by the most complicated element $b\cdot \Delta(a\cdot c)$.

By Propositions~\ref{prop:disjointunion},~\ref{prop:gluing},~\ref{prop:identity}
and~\ref{prop:symmetry},
\begin{eqnarray*}
Dlcop\circ\left[Id\otimes (\Delta\circ Dlcop)\right]\circ (\tau\otimes Id)=\\
\nu^{*s\otimes\iota}(F_{0,2+1})\circ
\left[\nu^{*id_1\otimes id_1}(F_{0,1+1})\otimes
(\nu^{*id_1\otimes \overline{\alpha}}(F_{0,1+1})\circ \nu^{*s\otimes\iota}(F_{0,2+1}))
\right]\\
\circ (\nu^{*\tau\otimes \tau}(C_\phi)\otimes\nu^{*id_1\otimes id_1}(F_{0,1+1}))=\\
\pm\nu^{*s\circ [id_1\otimes s]\circ (\tau\otimes id_1)
\bigotimes \iota\circ [id_1\otimes (\overline{\alpha}\circ\iota)]\circ (\tau\otimes id_1)}(F_{0,2+1}\circ(F_{0,1+1}\coprod F_{0,2+1})
\circ(C_\phi\coprod F_{0,1+1}))
\end{eqnarray*}
Here $\pm$ is the Koszul sign $(-1)^{\vert s\vert\vert\overline{\alpha}\vert}=(-1)^d$, since only $s$ and $\overline{\alpha}$ have positive degrees.

We choose $s'=s\circ (s\otimes id_1)$.
In the proof of Theorem~\ref{Associativity-Commutativity}, we saw
$
s\circ (s\otimes id_1)=
(-1)^d s\circ (id_1\otimes s)
$
and
$
s\circ \tau=(-1)^{ d}s
$. Therefore
$$
s\circ (id_1\otimes s)\circ (\tau\otimes id_1)=
(-1)^d s\circ (s\otimes id_1)\circ (\tau\otimes id_1)=
(-1)^d s\circ \left[(s\circ\tau)\otimes (id_1\circ id_1)\right]=
s'.
$$
Since $\iota\circ [id_1\otimes (\overline{\alpha}\circ\iota)]\circ (\tau\otimes id_1)$ coincides with $\overline{z}$ by Proposition~\ref{prop_structure}, we have proved that
$$
Dlcop\circ\left(Id\otimes (\Delta\circ Dlcop)\right)\circ (\tau\otimes Id)=(-1)^d \nu^{*s'\otimes\overline{z}}(F_{0, 3+1}).
$$
Similar computations shows that
\begin{align*}
Dlcop\circ\left(Id\otimes (\Delta\circ Dlcop)\right)=\\
\pm\nu^{*s\circ [id_1\otimes s]
\bigotimes \iota\circ [id_1\otimes (\overline{\alpha}\circ\iota)]}(F_{0,2+1}\circ(F_{0,1+1}\coprod F_{0,2+1}))=
\nu^{*s'\otimes\overline{x}}(F_{0, 3+1}),\\
Dlcop\circ\left((\Delta\circ Dlcop)\otimes Id\right)=\\
\pm\nu^{*s\circ [s\otimes id_1]
\bigotimes \iota\circ [(\overline{\alpha}\circ\iota)\otimes id_1]}(F_{0,2+1}\circ(F_{0,2+1}\coprod F_{0,1+1}))=
(-1)^d \nu^{*s'\otimes\overline{y}}(F_{0, 3+1}),\\
\Delta\circ Dlcop\circ (Dlcop\circ Id)=\\
\nu^{*s\circ [s\otimes id_1]
\bigotimes \overline{\alpha}\circ\iota\circ(\iota\otimes id_1)}(F_{0,2+1}\circ(F_{0,2+1}\coprod F_{0,1+1}))=
\nu^{*s'\otimes\overline{a_4}}(F_{0, 3+1}),\\
Dlcop\circ (\Delta\otimes Dlcop)=\\
\pm\nu^{*s\circ [id_1\otimes s]
\bigotimes \iota\circ [\overline{\alpha}\otimes\iota]}(F_{0,2+1}\circ(F_{0,1+1}\coprod F_{0,2+1}))=
\nu^{*s'\otimes\overline{a_1}}(F_{0, 3+1}), \\
Dlcop\circ (Id\otimes Dlcop)\circ(Id\otimes \Delta\otimes Id)=\\
\nu^{*s\circ [id_1\otimes s]
\bigotimes \iota\circ (id_1\otimes\iota)\circ (id_1\otimes\overline{\alpha}\otimes id_1)}(F_{0,2+1}\circ(F_{0,1+1}\coprod F_{0,2+1}))=
(-1)^d \nu^{*s'\otimes\overline{a_2}}(F_{0, 3+1})\\ \ \text{and} \  \
  Dlcop\circ (Dlcop\otimes \Delta)=\\
  \nu^{*s\circ [s\otimes id_1]
\bigotimes \iota\circ [\iota\otimes \overline{\alpha}]}(F_{0,2+1}\circ(F_{0,1+1}\coprod F_{0,2+1}))=
  \nu^{*s'\otimes\overline{a_3}}(F_{0, 3+1}).
\end{align*}
Therefore using the definition of the product $m$, straightforward computations show that 
\begin{eqnarray*} 
\Delta((a\cdot b)\cdot c)
&=(-1)^{d\vert b\vert+d}
\nu^{*s'\otimes\overline{a_4}}(F_{0, 3+1})(a\otimes b\otimes c)\\
\Delta(a)\cdot b\cdot c 
 &=(-1)^{d\vert b\vert+d} \nu^{*s'\otimes\overline{a_1}}(F_{0, 3+1})(a\otimes b\otimes c) \\
(-1)^{\parallel a\parallel}a\cdot \Delta(b)\cdot c 
&=(-1)^{d\vert b\vert+d}\nu^{*s'\otimes\overline{a_2}}(F_{0, 3+1})(a\otimes b\otimes c)\\
(-1)^{\parallel a\parallel +\parallel b\parallel}a\cdot b\cdot\Delta(c)
&=(-1)^{d\vert b\vert+d} \nu^{*s'\otimes\overline{a_3}}(F_{0, 3+1})(a\otimes b\otimes c)  \\
\Delta(a\cdot b)\cdot c
&=(-1)^{d\vert b\vert+d}\nu^{*s'\otimes\overline{y}}(F_{0, 3+1})(a\otimes b\otimes c) \\
(-1)^{\parallel a\parallel}a\cdot \Delta (b\cdot c)
&=(-1)^{d\vert b\vert+d}\nu^{*s'\otimes\overline{x}}(F_{0, 3+1})(a\otimes b\otimes c) \\
(-1)^{\parallel b\parallel \parallel a\parallel  + \parallel b\parallel}b\cdot \Delta(a\cdot c)  
&=(-1)^{d\vert b\vert+d}\nu^{*s'\otimes\overline{z}}(F_{0, 3+1})(a\otimes b\otimes c).
\end{eqnarray*}
The lantern relation gives rise to the equality 
\begin{align*}
\nu^{*s'\otimes\overline{a_4}}(F_{0, 3+1}) + \nu^{*s'\otimes\overline{a_1}}(F_{0, 3+1}) +
\nu^{*s'\otimes\overline{a_2}}(F_{0, 3+1}) + \nu^{*s'\otimes\overline{a_3}}(F_{0, 3+1})\\ =
 \nu^{*s'\otimes\overline{x}}(F_{0, 3+1})
 +\nu^{*s'\otimes\overline{y}}(F_{0, 3+1})  +\nu^{*s'\otimes\overline{z}}(F_{0, 3+1})
 \end{align*}
 since the hurewicz map is a morphism of groups. 
Thus 
\begin{align*}
\Delta(a\cdot b\cdot c) +\Delta(a)\cdot b\cdot c
+ (-1)^{\parallel a\parallel}a\cdot \Delta(b)\cdot c 
+(-1)^{\parallel a\parallel +\parallel b\parallel}a\cdot b\cdot\Delta(c) \\
=
\Delta(a\cdot b)\cdot c + (-1)^{\parallel a\parallel}a\cdot \Delta (b
\cdot c)+
(-1)^{\parallel b\parallel \parallel a\parallel  + \parallel b\parallel}b\cdot 
\Delta(a\cdot c).
\end{align*}
\begin{cor}\label{thm:B-V_algebra}
Let $G$ be a connected compact Lie group of dimension $d$.
Consider the graded associative and commutative algebra $(\mathbb{H}^*(LBG),m)$
given by Corollary~\ref{cor:shifted_alg}.
Let $\Delta$ be the operator
induced by the action of the circle on $LBG$ (See our definition in section~\ref{Operateur Delta})).
Then 
the shifted cohomology ${\mathbb H}^*(LBG)$ carries the structure of a Batalin-Vilkovisky algebra. 
\end{cor}
\begin{proof}
By Proposition~\ref{Delta en cohomologie dual du Delta en homologie} and by~\cite[Proposition 60]{Chataur-Menichi:stringclass}),
$$
\Delta=\nu^{*id_1\otimes\overline{\alpha}}(F_{0,1+1}).$$
\end{proof}
\section{Seven prop structure equalities on the homology of mapping class groups proving the BV identity}
Recall that for any simple closed curve $\gamma$ in a cobordism $F$, 
we write $\overline{\gamma}$ for the image of the Dehn twist $T_\alpha$ by the hurewicz map $\Theta$
$$
\pi_0(Diff^+(F,\partial))\maprightud{\partial^{-1}}{\cong} 
\pi_1(BDiff^+(F,\partial))\maprightud{\Theta}{} H_1(BDiff^+(F,\partial)).
$$
Here $\partial$ is the connecting homomorphism associated to the universal principal
fibration.

Let $\alpha$ be a closed curve in the cylinder $F_{0,1+1}$ parallel to one of the boundary components.
Let $a_1, ..., a_4$ and $x, y, z$ be the simple closed curves in $F_{0,3+1}$ described in \cite[Figure 6.89, page 161]{Kupers:stringop}. 
In what follows, we denote by $\circ$ the vertical product in the prop 
$$\bigoplus_F H_*(BDiff^+(F,\partial); \K)$$ which acts (up to signs) on $H^{*+\dim G}(LBG; \K)$.  
The goal of this section is to show the following equalities
needed in the proof of the BV-identity, given in section~\ref{section:BV-identity}. 
\begin{prop}\label{prop_structure}
$
\overline{z}=\iota\circ [id_1\otimes (\overline{\alpha}\circ\iota)]\circ [\tau\otimes id_1]
$,
$\overline{x}= \iota\circ [id_1\otimes (\overline{\alpha}\circ\iota)]$,
$\overline{y}= \iota\circ [(\overline{\alpha}\circ\iota)\otimes id_1]$,
$\overline{a_4}= \overline{\alpha}\circ\iota\circ(\iota\otimes id_1)$,
$\overline{a_1}=\iota\circ [\overline{\alpha}\otimes\iota]$,
$\overline{a_2}=\iota\circ (id_1\otimes\iota)\circ (id_1\otimes\overline{\alpha}\otimes id_1)$
and $ \overline{a_3}=\iota\circ [\iota\otimes \overline{\alpha}]$.
\end{prop}

Let $\widetilde{F}$ denote the group $Diff^+(F,\partial)$ (or the mapping class group of a surface $F$ with boundary $\partial$). 
Recall that $\iota_F$ or simply $\iota$ denote the generator of $H_0(B\widetilde{F})$ which is represented by the connected component of $B\widetilde{F}$.
\begin{prop}\label{prop:gluinghurewicz}
Let $F$ and $F'$ be two cobordisms.
In i) and ii), suppose that $F$ and $F'$ are gluable.
Let $\circ:\widetilde{F}\times \widetilde{F'}\rightarrow \widetilde{F\circ F'}$
be the map induced by gluing on diffeomorphisms.
Let $id_F\in \widetilde{F}$ be the identity map of $F$.
For $D$ in $\pi_0(\widetilde{F})$ and $D'$ in $\pi_0(\widetilde{F'})$,

i) $\Theta \partial^{-1}(id_F\circ D') = \iota_F \circ \Theta\partial^{-1} D'$

ii) $\Theta \partial^{-1}(D\circ id_{F'}) = \Theta\partial^{-1} D\circ \iota_{F'}$. 

iii) $\Theta \partial^{-1}(id_F\sqcup D') = \iota_F \otimes \Theta\partial^{-1} D'$
\end{prop}

\begin{proof} 

We consider the diagram:

\xymatrix{
&\pi_0(\widetilde{F})\times \pi_0(\widetilde{F'}) \ar[d]^-{\varphi}_-{\cong}
\\
\pi_0(\widetilde{F'})\ar[ur]^{i_2}\ar[r]_{\pi_0(i_2)}
&\pi_0(\widetilde{F}\times \widetilde{F'})\ar[r]^{\pi_0(\circ)}
&\pi_0(\widetilde{F\circ F'})\\
\pi_1(B(\widetilde{F'}))\ar[d]^{\Theta} \ar[u]^{\cong}_{\partial}\ar[r]_{\pi_1(B(i_2))}\ar[dr]_{\pi_1(i_2)}
&\pi_1(B(\widetilde{F}\times \widetilde{F'}))\ar[ddr]^{\Theta} \ar[u]^{\cong}_{\partial}\ar[d]^{\pi_1(\xi)}_\cong\ar[r]_{\pi_1(B(\circ))}
& \pi_1(B\widetilde{F\circ F'})\ar[ddr]^{\Theta}\ar[u]^{\cong}_{\partial}\\
H_1(B\widetilde{F'}) \ar[d]_{k_2}\ar[dr]_{H_1(i_2)}
&\pi_1(B\widetilde{F}\times B\widetilde{F'})\ar[d]^{\Theta}  \\
H_0(B\widetilde{F})\otimes H_1(B\widetilde{F'})\ar[r]_{\kappa}
&H_1(B\widetilde{F}\times B\widetilde{F'})
& H_1(B(\widetilde{F}\times \widetilde{F'}))\ar[l]^{H_1(\xi)}_\cong\ar[r]_{H_1(B(\circ))}
&H_1(B\widetilde{F\circ F'}) 
}

Here $\varphi$ is the natural isomorphism, $\kappa$ is the K\"unneth map, $\xi:B(\widetilde{F}\times \widetilde{F'})\buildrel{\approx}\over\rightarrow B(\widetilde{F})\times B(\widetilde{F'})$ is the canonical homotopy equivalence,
$k_2$ is the isomorphism defined by $k_2(x) = \iota_F\otimes x$
and $i_2$ denotes various inclusions on the second factor.
Note that by the definition of the prop structure, the bottom line
coincides with $\circ: H_0(B\widetilde{F})\otimes H_1(B\widetilde{F'})\rightarrow H_1(B\widetilde{F\circ F'})$.
The commutativity of the diagram shows i).

Replacing the $i_2$'s and $k_2$ by inclusions on the first factor, we obtain ii).
Replacing $\circ:\widetilde{F}\times \widetilde{F'}\rightarrow \widetilde{F\circ F'}$ by 
the map
$
\widetilde{F}\times \widetilde{F'}\rightarrow \widetilde{F\coprod F'}
$, $(D,D')\mapsto D\sqcup D'$, we obtain iii).

\end{proof}

\begin{proof}[Proof of Proposition \ref{prop_structure}]
Let $F = (F_{0, 1+1}\coprod F_{0, 2+1})\circ (C_\phi\coprod F_{0, 1+1})$.
We can identify $F_{0,3+1}$ with
$F_{0,2+1}\circ (F_{0, 1+1}\coprod F_{0, 1+1})\circ F$.
Let $emb_2:F_{0, 1+1}\hookrightarrow F_{0,3+1}$ be the second embedding due to this
identification.
The composite of the curve $\alpha$ and of $emb_2$,
$S^1\buildrel{\alpha}\over\rightarrow F_{0, 1+1}\buildrel{emb_2}\over\hookrightarrow F_{0,3+1}$,
coincides with the curve $z$.
Taking the same tubular neighborhood around $\alpha$ and $z$, the Dehn twists of $\alpha$
and $z$, $T_\alpha$ and $T_z$, coincide on this tubular neighborhood.
Outside of this tubular neighborhood, $T_\alpha$ and $T_z$ coincide with the identity maps
of $F_{0, 1+1}$ and of $F_{0,3+1}$, $id_{F_{0, 1+1}}$ and $id_{F_{0, 3+1}}$.
Therefore
$$
T_z=id_{F_{0, 2+1}}\circ (id_{F_{0, 1+1}}\sqcup T_\alpha)\circ id_F.
$$
By virtue of 
Proposition \ref{prop:gluinghurewicz} i), ii) and then iii), we have
\begin{eqnarray*}
\overline{z}=\Theta\partial^{-1} T_z
=\Theta\partial^{-1}\left(id_{F_{0, 2+1}}\circ (id_{F_{0, 1+1}}\sqcup T_\alpha)\circ id_F\right)\\
=\iota_{F_{0, 2+1}} \circ \Theta\partial^{-1}\left((id_{F_{0, 1+1}}\sqcup T_\alpha)\circ id_F\right)\\
=\iota_{F_{0, 2+1}} \circ \Theta\partial^{-1}\left(id_{F_{0, 1+1}}\sqcup T_\alpha\right)\circ \iota_F\\
=\iota_{F_{0, 2+1}} \circ \left(\iota_{F_{0, 1+1}}\otimes \Theta\partial^{-1}T_\alpha\right)\circ \iota_F
=\iota_{F_{0, 2+1}}\circ [id_1\otimes \overline{\alpha}]\circ \iota_F
\end{eqnarray*}
The prop structure on the 0th homology gives
$\iota_F=[id_1\otimes \iota_{F_{0, 2+1}}]\circ [\tau\otimes id_1]
$.
Finally, the prop structure on the homology of mapping class groups gives
$$
\overline{z}=\iota_{F_{0, 2+1}}\circ [id_1\otimes \overline{\alpha}]\circ [id_1\otimes \iota_{F_{0, 2+1}}]\circ [\tau\otimes id_1]
=\iota_{F_{0, 2+1}}\circ [id_1\otimes (\overline{\alpha}\circ\iota_{F_{0, 2+1}})]\circ [\tau\otimes id_1].
$$
By similar fashion, we have the other six  equalities.
\end{proof}


\section{The cohomological BV-operator $\Delta$}\label{Operateur Delta}
The goal of this section is to give our definition of the BV-operator $\Delta$ in cohomology and to compare it to others definitions given in the literature.

Let $\Gamma : S^1\times LX \to LX$ be the $S^1$-action map.  
Then in this paper the Batalin-Vilkovisky operator $\Delta : H^*(LX) \to H^{*-1}(LX)$ is defined~\cite[Proposition 3.3]{Kuri:moduleadjoint} by 
$\Delta := \int_{S^1} \circ \Gamma^*$, where  $\int_{S^1} : H^*(S^1\times LX) \to H^{*-1}(LX)$ denotes the integration along the fibre of the 
trivial fibration $S^1\times LX\twoheadrightarrow LX$.

By our example in section~\ref{review of chataurmenichi} (see also up to the sign~\cite[Proof of Proposition 3.3]{Kuri:moduleadjoint}),
$\int_{S^1} f\times b=(-1)^{\vert f\vert}<f,[ S^1] > b$. Up to some signs, this is the slant with $[S^1]$ (Compare~\cite[Definition 1]{Kishi-Kono:freetwisted}).

Therefore for any $\beta\in H^*(LX)$, the image of $\beta$ by $\Delta$, $\Delta(\beta)$, is the unique element
such that (see~\cite{tamanoi:capproducts} up to the sign $-$ )
$$
\Gamma^*(\beta)=1\times \beta - \{S^1\}\times \Delta(\beta)
$$
where $\{S^1\}$ is the fundamental class in cohomology defined by
$<\{S^1\},[S^1]>=1$.

So finally, we have proved that with our definition of integration along the fibre,
since we define the BV-operator $\Delta$ using integration along the fibre as~\cite[Proposition 3.3]{Kuri:moduleadjoint}, our $\Delta$ is exactly the opposite of the one defined by~\cite{tamanoi:capproducts} or~\cite[p. 648 line 4]{Kishi-Kono:freetwisted}.

In particular, observe that $\Delta$ satisfies $\Delta^2=0$ and is a derivation on the cup product on $H^*(LX)$~\cite[Proposition 4.1]{tamanoi:capproducts}. 

In section~\ref{section:BV-identity}, we will need another characterisation of our BV-operator $\Delta$:
\begin{prop}\label{Delta en cohomologie dual du Delta en homologie}
The BV-operator $\Delta := \int_{S^1} \circ \Gamma^*$ is the dual (=transposition) of the
composite
$$
H_*(LX)\buildrel{[S^1]\times -}\over\rightarrow H_{*+1}(S^1\times LX)\buildrel{\Gamma_*}\over\rightarrow H_{*+1}(LX).
$$
\end{prop}
\begin{proof}
For any space $X$, let $\mu_X:H^*(X;\K)\rightarrow H_*(X;\K)^\vee$ be the map
sending $\alpha$ to the form on $H_*(X;\K)$, $<\alpha,->$. Here $<-,->$ is the
Kronecker bracket. By the universal coefficient theorem for cohomology,
$\mu_X$ is an isomorphism. Consider the two squares
$$\xymatrix{
H^*(LX)\ar[r]^{\Gamma^*}\ar[d]_{\mu_{LX}}
&H^*(S^1\times LX)\ar[r]^{\int_{S^1}}\ar[d]_{\mu_{S^1\times LX}}
&H^{*-1}(LX)\ar[d]^{\mu_{LX}}\\
H_*(LX)^\vee\ar[r]_-{(\Gamma_*)^\vee}
&H_*(S^1\times LX)^\vee\ar[r]_{([S^1]\times -)^\vee}
&H_{*-1}(LX)^\vee.
}
$$
The left square commutes by naturality of $\mu_X$.
For any $\alpha\in H^*(S^1)$ and $\beta\in H^*(LX)$ and $y\in H_*(LX)$,
\begin{align*}
(\mu_{LX}\circ \int_{S^1})(\alpha\times\beta)(y)
=\mu_{LX}\left((-1)^{\vert\alpha\vert\vert [S^1]\vert}<\alpha, [S^1]>\beta\right)(y)
\\=(-1)^{\vert\alpha\vert\vert [S^1]\vert}<\alpha, [S^1]><\beta,y>
\end{align*}
and
\begin{align*}
([S^1]\times -)^\vee\left(\mu_{S^1\times LX}(\alpha\times\beta)\right)(y)
=(-1)^{\vert\alpha\times\beta\vert\vert [S^1]\vert}\mu_{S^1\times LX}(\alpha\times\beta)\circ ([S^1]\times -)(y)\\
=(-1)^{\vert\alpha\vert\vert [S^1]\vert+\vert\beta\vert\vert [S^1]\vert}<\alpha\times\beta, [S^1]\times y>.
\end{align*}
Since $<\alpha\times\beta, [S^1]\times y>=(-1)^{\vert\beta\vert\vert [S^1]\vert}
<\alpha, [S^1]><\beta,y>$, the right square commutes also.
\end{proof}
\section{Hochschild cohomology computations}
\begin{prop}\label{unite image delta dans Hochschild}
Let $A$ be a graded (or ungraded) algebra equipped with an isomorphism of
$A$-bimodules $\Theta:A\buildrel{\cong}\over\rightarrow A^\vee$ between $A$ and its dual of any degree $\vert\Theta\vert$. Denote by $tr:=\Theta(1)$ the induced graded trace on $A$.
Let $a\in Z(A)$ be an element of the center of $A$.
Let $d:A\rightarrow A$ be a derivation of $A$.
Obviously $\overline{a}\in\mathcal{C}^0(A,A)=\text{Hom}(\K, A)$ defined by $\overline{a}(1)=a$
and $d\circ s^{-1}\in\mathcal{C}^1(A,A)=\text{Hom}(s\overline{A}, A)$ are two Hochschild cocycles.
Then in the Batalin-Vilkovisky algebra
$HH^*(A,A)\cong HH^{*+\vert\Theta\vert}(A,A^\vee)$,

1) $\Delta([\overline{a}])=0$,

2) $\Delta([d\circ s^{-1}])$ is equal to $[\overline{a}]$ the cohomology class of $\overline{a}$
if and only if for any $a_0\in A$,
$$
(-1)^{1+\vert d\vert} tr\circ d(a_0)=tr(aa_0).
$$

3) In particular, the unit belongs to the image of
$\Delta$ if and only if there exists a derivation $d:A\rightarrow A$ of degree $0$
commuting with the trace: $tr\circ d(a_0)=tr(a_0)$ for any element $a_0$ in $A$.
\end{prop}
\begin{proof}
By definition of $\Delta$, the following diagram commutes up to the sign $(-1)^{\vert\Theta\vert}$ for any $p\geq 0$.
$$
\xymatrix{
\mathcal{C}^p(A,A)\ar[r]^{\mathcal{C}^p(A,\Theta)}\ar[d]_\Delta
&\mathcal{C}^p(A,A^\vee)\ar[r]^{Ad}
&\mathcal{C}_p(A,A)^\vee\ar[d]_{B^\vee}\\
\mathcal{C}^{p-1}(A,A)\ar[r]_{\mathcal{C}^{p-1}(A,\Theta)}
&\mathcal{C}^{p-1}(A,A^\vee)\ar[r]_{Ad}
&\mathcal{C}_{p-1}(A,A)^\vee.
}
$$
Taking $p=0$ we obtain 1).

\noindent The image of the cocycle $d\circ s^{-1}\in\mathcal{C}^1(A;A)$ by $Ad\circ\mathcal{C}^*(A;\Theta)$
is the form $\widehat{\Theta}(d)$ on $\mathcal{C}_1(A;A)=A\otimes s\overline{A}$
defined by (Compare with~\cite[Proof of Proposition 20]{menichi:stringtopspheres})
$$
\widehat{\Theta}(d)(a_0[sa_1])=(-1)^{\vert sa_1\vert\vert a_0\vert}
(\Theta\circ d)(a_1)(a_0)=(-1)^{\vert sa_1\vert\vert a_0\vert}tr (d(a_1)a_0).
$$
For any $a_0\in A$,
$$
(-1)^{\vert\Theta\vert +1+\vert d\vert}B^\vee(\widehat{\Theta}(d))(a_0)
=(\widehat{\Theta}(d)\circ B)(a_0[])=\widehat{\Theta}(d)(1[sa_0])=tr\circ d(a_0).
$$
The image of the cocycle $\overline{a}\in\mathcal{C}^0(A;A)$ by $Ad\circ\mathcal{C}^*(A;\Theta)$
is the form on $A$, mapping $a_0$ to $(\Theta\circ\overline{a})([])(a_0)=\Theta (a)(a_0)=tr(aa_0)$.

\noindent Therefore $\Delta(d\circ s^{-1})=a$
if and only if for any
$a_0\in A$, $(-1)^{\vert\Theta\vert +1+\vert d\vert}tr\circ d(a_0)=(-1)^{\vert\Theta\vert}tr(aa_0)$.
Since there is no coboundary in $\mathcal{C}^0(A,A)$, this proves 2).
\end{proof}
\begin{ex}\label{ex: unit belongs to image of Delta}
a) Let $A=\Lambda x_{-d}$ be the exterior algebra on a generator of lower degree
$-d\in\mathbb{Z}$. If $d\geq 0$ then $A=H^*(S^d;\mathbb{F})$.
Denote by $1^\vee$ and $x^\vee$ the dual basis of $A^\vee$.
The trace on $A$ is $x^\vee$.
Let $d:A\rightarrow A$ be the linear map such that $d(1)=0$ and $d(x)=x$.
Since $d(x\wedge x)=0$ and $dx\wedge x+x\wedge dx=2x\wedge x=2\times 0=0$, even over a field of characteristic different from $2$, $d$ is a derivation commuting with the trace.
Therefore by Theorem~\ref{unite image delta dans Hochschild}, $1\in \text{Im } \Delta$
in $HH^*(A;A)$. When $\mathbb{F}=\mathbb{F}_2$, compare with~\cite[Proposition 20]{menichi:stringtopspheres}.

b)Let $V$ be a graded vector space. Let $A:=\Lambda (V)$ be the graded exterior algebra
on $V$. If $V$ is in non-positive degrees, then $A$ is just the cohomology algebra
of a product of spheres.
Let $x_1$, \dots, $x_N$ be a basis of $V$. The trace of $A$ is $(x_1\dots x_N)^\vee$.
Let $d_1$ be the derivation on $\Lambda x_1$ considered in the previous example.
Then $d:=d_1\otimes id$ is a derivation on
$\Lambda x_1\otimes \Lambda(x_2,\dots,x_N)\cong \Lambda V$.
Obviously $d$ commutes with the trace.
So $1\in \text{Im } \Delta$.

c) Let $A=F[x]/x^{n+1}$, $n\geq 1$ be the truncated polynomial algebra on a generator $x$
of even degree different from $0$. If $x$ is of upper degree $2$ then
$A=H^*(\mathbb{CP}^n;\mathbb{F})$.
The trace of $A$ is $(x^n)^\vee$.
Let $d:A\rightarrow A$ be the unique derivation of $A$ such that $d(x)=x$
(The case $n=1$ was considered in example a)).
Then $d(x^i)=ix^i$. For degree reason, $d$ is a basis of the derivations of degree $0$ of $A$. Then $\lambda d$ commutes with the trace if and only if $\lambda n=1$ in $\mathbb{F}$.
Therefore $1\in \text{Im } \Delta$ in $HH^*(A;A)$ if and only $n$ is invertible
in $\mathbb{F}$(Compare with~\cite{Westerland:stringhomologyspheres} modulo 2 and with~\cite{Yang:BV} otherwise).
\end{ex}
\begin{thm}\label{thm: Hochschild cohomology exterior algebra}
Let $V$ be a graded vector space (non-negatively lower graded or concentrated in upper degree $\geq 1$)
such that in each degre, $V$ is of finite dimension.

i) Let $A={\bf S}(V),0$ be the free strictly commutative graded algebra on $V$:
$A=\Lambda V^{odd}\otimes \mathbb{F}[V^{even}]$ is the graded tensor product on the exterior algebra on $V^{odd}$, the odd degree elements and on $V^{even}$ the even degree elements~\cite[p. 46]{Felix-Halperin-Thomas:ratht}.
Then the Hochschild cohomology of $A$, $HH^*(A,A)$, is isomorphic as Gerstenhaber algebras
to $A\otimes {\bf S}(s^{-1}V^\vee)$. For $\varphi$ a linear form on $V$ and $v\in V$,
$
\{1\otimes s^{-1}\varphi,v\otimes 1\}=(-1)^{\vert\varphi\vert}\varphi(v).
$
The Lie bracket is trivial on $(A\otimes 1)\otimes (A\otimes 1)$ and on
$(1\otimes {\bf S}(s^{-1}V^\vee))\otimes (1\otimes {\bf S}(s^{-1}V^\vee))$.

ii) Suppose that $\mathbb{F}$ is a field of characteristic $2$. Then we can extend i) in the following way:
Let $U$ and $W$ are two graded vector spaces such that $U\oplus W=V$.
(i. e. we don't assume anymore that $U=V^{odd}$ and $W=V^{even}$).
Let $A=\Lambda U\otimes \mathbb{F}[W]$.
Then $HH^*(A,A)$ is isomorphic as Gerstenhaber algebras
to $A\otimes \mathbb{F}(s^{-1}U^\vee)\otimes \Lambda(s^{-1}W^\vee)$
and the Lie bracket is the same as in i).

iii) Suppose that $V$ is concentrated in odd degres or that $\K$ is a field of characteristic $2$.
Let $A=\Lambda V$ be the exterior algebra on $V$.
Then the BV-algebra extending the Gerstenhaber algebra $HH^*(A,A)\cong A\otimes \K[s^{-1}V^\vee]$
has trivial BV-operator $\Delta$ on $A$ and on $\K[s^{-1}V^\vee]$.
\end{thm}
\begin{proof}
i) Recall that the Bar resolution $B(A,A,A)=A\otimes TsA\otimes A\buildrel{\simeq}\over\twoheadrightarrow A$
is a resolution of $A$ as $A\otimes A^{op}$-modules.
When $A={\bf S}(V),0$, there is another smaller resolution
$(A\otimes \Gamma(sV)\otimes A,D)\buildrel{\simeq}\over\twoheadrightarrow A$.
Here $\Gamma(sV)$ is the free divided power graded algebra on $sV$ and $D$ is the unique derivation
such that 
$D(\gamma^k(sv))=v\otimes \gamma^{k-1}(sv)\otimes 1-1\otimes \gamma^{k-1}(sv)\otimes v$
~\cite{MenichiL:cohrfl}.
Since $\Gamma(sV)$ is the invariants of $T(sV)$ under the action of the permutation groups,
there is a canonical inclusion of graded algebras~\cite[p. 278]{Halperin:unieal}
$$
i:\Gamma(sV)\hookrightarrow T(sV)\hookrightarrow T(sA).
$$
This map $i$ maps $\gamma^k(sv)$ to $[sv|\dots|sv]$.
Since both $(A\otimes \Gamma(sV)\otimes A,D)$ and $B(A,A,A)$ are $A\otimes A$-free resolutions of $A$,
the inclusion of differential graded algebras
$$
A\otimes i\otimes A:(A\otimes \Gamma(sV)\otimes A,D)\buildrel{\simeq}\over\hookrightarrow B(A,A,A)
$$
is a quasi-isomorphism.
So by applying the functor $\text{Hom}_{A\otimes A}(-,A)$,
$\text{Hom}(i,A):\mathcal{C}^*(A,A)\buildrel{\simeq}\over\twoheadrightarrow (\text{Hom}(\Gamma (sV),A),0)$
is a quasi-isomorphism of complexes.
The differential on $\text{Hom}_{A\otimes A}((A\otimes \Gamma(sV)\otimes A,D),(A,0))$
is zero since
$$
f\circ D(\gamma^{k_1}(sv_1)\dots \gamma^{k_r}(sv_r))=0.
$$
The inclusion $i:\Gamma(sV)\hookrightarrow T(sA)$
is a morphism of graded coalgebras with respect to the diagonal~\cite[p. 279]{Halperin:unieal}
$$
\Delta[sa_1|\dots|sa_r]=\sum_{p=0}^r [sa_1|\dots|sa_p]\otimes [sa_{p+1}|\dots|sa_r].
$$
Therefore the quasi-isomorphism of complexes $\text{Hom}(i,A):\mathcal{C}^*(A,A)\buildrel{\simeq}\over\twoheadrightarrow (\text{Hom}(\Gamma (sV),A),0)$ is also a morphism of graded algebras with respect to the cup product on the Hochschild cochain complex $\mathcal{C}^*(A,A)$
and the convolution product on $\text{Hom}(\Gamma (sV),A)$.

The morphism of commutative graded algebras
$j:A\otimes \Gamma(sV)^\vee\rightarrow \text{Hom}(\Gamma (sV),A)$
mapping $a\otimes \phi$ to the linear map $j(a\otimes \phi)$ from $\Gamma (sV)$ to $A$
defined by $j(a\otimes \phi)(\gamma)=\phi(\gamma)a$
is an isomorphim.
By~\cite[(A.7)]{Halperin:unieal}, the canonical map $(sV)^\vee\rightarrow \Gamma(sV)^\vee$ extends to an isomorphism
of graded algebras $k:{\bf S}(sV)^\vee\buildrel{\cong}\over\rightarrow \Gamma(sV)^\vee$.
The composite
$\Theta:(sV)^\vee\buildrel{s^\vee}\over\rightarrow V^\vee \buildrel{s^{-1}}\over\rightarrow s^{-1}(V^\vee) $,
 mapping $x$ to $\Theta(x)=(-1)^{\vert x\vert}s^{-1}(x\circ s)$, is a chosen isomorphism between $(sV)^\vee$ and $s^{-1}(V^\vee)$. Note that $\Theta^{-1}$ is the opposite of the
composite $(s^{-1})^\vee\circ s$.
Finally, the composite
\begin{equation*}
A\otimes {\bf S}(s^{-1}(V^\vee))\buildrel{A\otimes {\bf S}(\Theta)}\over\rightarrow A\otimes {\bf S}((sV)^\vee)
\buildrel{A\otimes k}\over\rightarrow A\otimes (\Gamma(sV))^\vee\buildrel{j}\over\
\rightarrow \text{Hom}(\Gamma (sV),A)
\end{equation*}
is an isomorphism of graded algebras. So we have obtained an explicit isomorphism of graded algebras
$l:HH^*(A,A)\buildrel{\cong}\over\rightarrow A\otimes {\bf S}(s^{-1}(V^\vee))$.
To compute the bracket, it is sufficient to compute it on the generators on $A\otimes {\bf S}(s^{-1}(V^\vee))$.
For $m\in A$, let $\overline{m}\in\mathcal{C}^0(A,A)=\text{Hom}((sA)^{\otimes 0},A)$ defined by $\overline{m}([])=m$.
Obviously, $l^{-1}(m\otimes 1)$ is the cohomology class of the cocycle $\overline{m}$.
For any linear form $\varphi$ on $V$, let $\overline{\varphi}\in\mathcal{C}^1(A,A)=\text{Hom}(sA,A)$ be the map
defined by
$$
\overline{\varphi}([sv_1v_2\dots v_n])=\sum_{i=1}^n (-1)^{\vert\varphi\vert\vert sv_1v_2\dots v_{i-1}\vert}
\varphi(v_i)v_1\dots \widehat{v_i}\dots v_n.
$$
Since the composite $\overline{\varphi}\circ s$ is a derivation of $A$, $\overline{\varphi}$ is a cocycle.
Since $\overline{\varphi}([sv_1])=(-1)^{\vert\varphi\vert} \varphi(v_1)1$, the composite $\overline{\varphi}\circ i$
is the image of $1\otimes s^{-1}\varphi$ by the composite
$
j\circ (A\otimes k)\otimes (A\otimes {\bf S}(\Theta)):A\otimes {\bf S}(s^{-1}(V^\vee))\rightarrow \text{Hom}(\Gamma (sV),A)
$.
Therefore $l^{-1}(1\otimes s^{-1}\varphi)$ is the cohomology class of the cocycle $\overline{\varphi}$.
By~\cite[p. 48-9]{Felix-Menichi-Thomas:GerstduaiHochcoh},

\noindent a) the Lie bracket is null on $\mathcal{C}^0(A,A)\otimes \mathcal{C}^0(A,A)$,

\noindent b) the Lie bracket restricted to
$\{\quad,\quad\}:\mathcal{C}^1(A,A)\otimes \mathcal{C}^0(A,A)\rightarrow \mathcal{C}^0(A,A)$
is given by $\{g,\overline{a}\}=\overline{g(sa)}$ for any $g:sA\rightarrow A$
and $a\in A$,

\noindent c) the Lie bracket restricted to
$\{\quad,\quad\}:\mathcal{C}^1(A,A)\otimes \mathcal{C}^1(A,A)\rightarrow \mathcal{C}^1(A,A)$
is given by
$$\{f,g,\}([sa])=
f\circ s\circ g\circ s(a)-(-1)^{(\vert f\vert+1)(\vert g\vert+1)}g\circ s\circ f\circ s(a).$$
By a), the Lie bracket is trivial on $(A\otimes 1)\otimes (A\otimes 1)$.
By b), for $ \varphi\in V^\vee$ and $v\in V$,
$$\{1\otimes s^{-1}\varphi,v\otimes 1\}=(-1)^{\vert\varphi\vert}\varphi(v)1\otimes 1.$$
Let $\varphi$ and $\varphi'$ be two linear forms on $V$. Then
$$
\overline{\varphi}\circ s \circ \overline{\varphi'}\circ s ([v_1\dots v_n])=
\sum_{1\leq j<i\leq n}\left( (-1)^{\vert\varphi\vert \vert\varphi'\vert}\varepsilon_{ij}(\varphi,\varphi')
+\varepsilon_{ij}(\varphi',\varphi)\right) v_1\dots \widehat{v_j}\dots \widehat{v_i}\dots v_n
$$
where $\varepsilon_{ij}(\varphi,\varphi')=(-1)^{\vert\varphi\vert\vert sv_1\dots v_{i-1}\vert+\vert\varphi'\vert\vert sv_1\dots v_{j-1}\vert}\varphi(v_i)\varphi'(v_j)$.
Therefore $\overline{\varphi}\circ s \circ \overline{\varphi'}\circ s-(-1)^{\vert\varphi\vert\vert\varphi'\vert}\overline{\varphi'}\circ s \circ \overline{\varphi}\circ s=0$.
So by c), the Lie bracket
$\{1\otimes s^{-1}\varphi,1\otimes s^{-1}\varphi'\}=0.$

iii) By 1) of Proposition~\ref{unite image delta dans Hochschild}, $\Delta([\overline{m}])=0$ and so $\Delta$ is trivial 
on all $m\otimes 1\in A\otimes 1$.
Let $x_1$, \dots, $x_N$ be a basis of $V$. The trace of $A$ is $(x_1\dots x_N)^\vee$.
Therefore the trace vanishes on elements of wordlength strictly less than $N$.
For any $\varphi\in V^{\vee}$, the derivation $\overline{\varphi}\circ s$ decreases wordlength by $1$.
So $tr\circ  \overline{\varphi}\circ s=0$. By 2) of Proposition~\ref{unite image delta dans Hochschild},
$\Delta(1\otimes s^{-1}\varphi)=0$.
Since the Lie bracket is trivial on
$(1\otimes \K[s^{-1}V^\vee])\otimes (1\otimes \K[s^{-1}V^\vee])$,
$\Delta$ is trivial on $1\otimes \K[s^{-1}V^\vee]$.

ii) The proof is the same as in i): for example, $\Gamma(sV)$ is the graded tensor product of the free divided power
algebra on $sU$ and of the exterior algebra on $sW$.
\end{proof}
\begin{rem}
Suppose that $V$ is concentrated in degree $0$.
We have obtained a quasi-isomorphism of differential graded algebras
$$
\mathcal{C}^*({\bf S}(V),{\bf S}(V))\buildrel{\simeq}\over\twoheadrightarrow
({\bf S}(V)\otimes\Lambda(s^{-1}V^\vee),0).
$$
In particular, the differential graded algebra $\mathcal{C}^*({\bf S}(V),{\bf S}(V))$
is formal.

It is easy to see that if $V$ is of dimension 1 then the inclusion
$$({\bf S}(V)\otimes\Lambda(s^{-1}V^\vee),0)\hookrightarrow \mathcal{C}^*({\bf S}(V),{\bf S}(V))$$ is a quasi-isomorphism of differential graded Lie algebras.
In particular, the differential graded Lie algebra $\mathcal{C}^*({\bf S}(V),{\bf S}(V))$
is formal. Kontsevich formality theorem says that over a field $\mathbb{F}$ of characteristic zero, the differential graded Lie algebra $\mathcal{C}^*({\bf S}(V),{\bf S}(V))$
is formal even if $V$ is not of dimension $1$~\cite[Theorem 2.4.2 (Tamarkin)]{Keller:Deformation}.
\end{rem}

\section{Triviality of the loop product when $H^*(BG)$ is polynomial} 
This section is independent of the rest of the paper.
Recall the dual of the loop coproduct introduced by Sullivan for manifolds
$
H^*(LM)\otimes H^*(LM)\rightarrow H^{*+d}(LM)
$
is (almost) trivial~\cite{tamanoi-2007}.
In this section, we prove that the loop product for classifying spaces of Lie groups
$H_*(LBG)\otimes H_*(LBG)\rightarrow H_{*+d}(LBG)$
is trivial if the inclusion of the fibre in cohomology
$j^*:H^*(LBG; \K)\twoheadrightarrow H^*(G; \K)$ is surjective
(Theorem~\ref{thm:BG}).
We also explain that this condition $j^*:H^*(LBG; \K)\twoheadrightarrow H^*(G; \K)$
surjective is equivalent to our hypothesis
$H^*(BG)$ polynomial (Theorem~\ref{thm:iff}). 
\begin{thm}
\label{thm:BG}
Let $BG$ be the classifying space of a connected Lie group $G$. Suppose that the map induced in cohomology
$H^*(LBG; \K)\twoheadrightarrow H^*(G; \K)$ is surjective. Then the loop product on $H_*(LBG; \K)$ is trivial 
while the loop coproduct is injective.  
\end{thm}
This result is a generalization of  \cite[Theorem D]{Felix-Thomas:stringtopGorenstein} in which it is assumed 
that  the underlying field is of characteristic zero.  
If $\text{Char}\K \neq 2$, the triviality of the loop product was first proved by
Tamanoi~\cite[Theorem 4.7 (2)]{Tamanoi:stabletrivial}.
The second author and David Chataur conjecture that the loop coproduct
on $H_*(LBG)$ has always a counit. Assuming that the loop coproduct
on $H_*(LBG)$ has a counit, obviously the loop coproduct is injective
and it follows from~\cite[Theorem 4.5 (i)]{Tamanoi:stabletrivial}
that the loop product on $H_*(LBG)$ is trivial.

The injectivity described in Theorem \ref{thm:BG} follows from 
a consideration of the Eilenberg-Moore spectral sequences associated with appropriate pullback diagrams. 
In fact, the induced maps $Comp^!$ and $H(q)$ in the cohomology are epimorphisms; see Proposition \ref{prop:key2}. 

Let $
\Omega X\buildrel{i}\over\hookrightarrow LX\twoheadrightarrow X
$
be the free loop fibration. The following proposition is a key to proving Theorem \ref{thm:BG}. 

\begin{prop} \label{prop:key2}
Let $X$ be a simply-connected space.
Suppose that $i^* : H^*(LX) \to H^*(\Omega X)$ induced by the inclusion is surjective. 
Then one has \\
{\em (1)} the map $H^*(q)$ induced by the inclusion 
$q : LX\times_X LX \to LX \times LX$ is an epimorphism.  \\
{\em (2)} Suppose that $G$ is a connected Lie group. Then, for the map $Comp: LBG\times_{BG} LBG \to LBG$, $Comp^!$ is an epimorphism.  \\
\end{prop}
\noindent
{\it Proof of Theorem \ref{thm:BG}.}  
By Proposition \ref{prop:key2} (1) and (2), we see that the dual to the loop coproduct $Dlcop:=Comp^!\circ H^*(q)$ 
on $H^*(LBG)$ is surjective.  
Since $q^!$ is $H^*(LBG\times LBG)$-linear and decreases the degrees, $q^!\circ H^*(q)=0$.
By Proposition \ref{prop:key2} (1), $H^*(q)$ is an epimorphism.
Therefore $q^!$ is trivial and the dual of the loop product
$Dlp:=H^*(q^!)\circ H^*(Comp)$ on $H^*(LBG)$ is also trivial.
\hfill\qed
\medskip

\noindent
{\it Proof of Proposition \ref{prop:key2}.}
Consider the two Eilenberg-Moore spectral sequences associated
to the free loop fibration mentioned above 
and to the pull-back diagram
$$
\xymatrix{
LX\times_X LX\ar[r]^q\ar[d]_p
& LX\times LX\ar[d]_{p\times p}\\
X\ar[r]^\Delta
& X\times X
}
$$
Since $H^*(LX)$ is a free $H^*(X)$-module by Leray-Hirsch theorem, these two Eilenberg-Moore spectral
sequences are concentrated on the $0$-th column.
So the two morphisms of graded algebras
$$
H^*(i)\otimes_{H^*(X)} \eta:
H^*(LX)\otimes_{H^*(X)} \K
\buildrel{\cong}\over\rightarrow H^*(\Omega X)
$$
and
$$
H^*(q)\otimes_{H^*(X)^{\otimes 2}} H^*(p):
(H^*(LX)\otimes H^*(LX))\otimes_{H^*(X)^{\otimes 2}}H^*(X)
\buildrel{\cong}\over\rightarrow H^*(LX\times_X LX)
$$
are isomorphisms. In particular, 
$H^*(q)$ is an epimorphism and we have an isomorphism of graded vector spaces
between $H^*(LX\times_X LX)$ and $H^*(LX)\otimes H^*(\Omega X)$.

Consider the Leray-Serre spectral sequence 
$\{\widehat{E}_r^{*,*}, \widehat{d}_r \}$ of the homotopy fibration 
$\Omega X \stackrel{j}{\to} LX\times_{X} LX \stackrel{Comp}{\to} LX$. 
Since $H^*(LX\times_XLX)$ is isomorphic to
$H^*(LX)\otimes H^*(\Omega X)$, by~\cite[III.Lemma 4.5 (2)]{Mimura-Toda:topliegroups},
$\{\widehat{E}_r^{*,*}, \widehat{d}_r \}$ collapses at the 
$E_2$-term.
Then for $X=BG$, the integration along the fibre 
$Comp^! : H^*(LBG\times_{BG} LBG) \to H^{*-\dim G}(LBG)$
is surjective.
\hfill\qed

%
%
Let $G$ be a connected Lie group and $\K$ a field of arbitrary characteristic. Let ${\mathcal F} : G \stackrel{j}{\to} LBG \to BG$ be the free loop 
fibration. 
\begin{thm}\label{thm:iff}
The induced map $j^* : H^*(LBG; \K) \to H^*(G; \K)$ is surjective if and only if $H^*(BG; \K)$ is a polynomial algebra. 
\end{thm}
\begin{proof}
The "if" part follows from the usual EMSS argument. In fact, suppose that 
 $H^*(BG; \K)\cong \K[V]$. Then the EMSS for the universal bundle  
${\mathcal F}' : G \to EG \to BG$ allows one to deduce that $H^*(G; \K)\cong \Delta(sV)$. By using the EMSS for the fibre square (\cite[Proof of Theorem 1.2]{Kono-Kuri:modulesplitting} or~\cite[Proof of Theorem 1.6]{Kuri:moduleadjoint})
$$
\xymatrix@C15pt@R15pt{
LBG \ar[r]\ar[d] & BG^I \ar[d] \\
BG \ar[r]_{\Delta} & BG \times BG, 
}
$$
we see that $H^*(LBG; \K) \cong H^*(BG; \K)\otimes \Delta(sV)$ as an $H^*(BG)=\K[V]$-algebra. This implies that the Leray-Serre spectral sequence (LSSS) 
for ${\mathcal F}$ collapses at the $E_2$-term and hence 
$j^*$ is surjective. See the beginning of section~\ref{cup product main theorem}
for an alternative proof which uses module derivations.
 
Suppose that $j^*$ is surjective. We further assume that $\text{Char}\K=2$. 
By the argument in \cite[Remark 1.4]{Kuri:moduleadjoint} or \cite[Proof of Theorem 2.2]{Iwase:adjoint}, 
we see that the Hopf algebra $A=H^*(G; \K)$ is cocommutative and so primitively generated; that is,   
the natural map $\iota : P(A) \to Q(A)$ is surjective. By~\cite[Lemma 4.3]{Kuri:moduleadjoint}, this yields that 
$H^*(G; \K) \cong \Delta (x_1, ..., x_N)$,  
where $x_i$ is primitive for any $1\leq i \leq N$. 
The same argument as in the proof of \cite[Chapter 7, Theorem 2.26(2)]{Mimura-Toda:topliegroups} allows us to deduce 
that  each $x_i$ is transgressive in the LSSS $\{E_r, d_r\}$ for ${\mathcal F}'$. 
To see this more precisely, we recall that 
the action of $G$ on $EG$ gives rise to a morphism of spectral sequence 
$$
\{\mu_r^*\} : \{E_r, d_r\} \to \{E_r\otimes H^*(G; \K), d_r\otimes 1\}
$$
for which $\mu_2^*=1\otimes\mu^* : H^*(BG;\K)\otimes H^*(G;\K)\to 
H^*(BG;\K)\otimes H^*(G;\K)\otimes H^*(G;\K)$, where $\mu : G \times G \to G$ denotes  
the multiplication on $G$; see \cite[Chapter 7, Section 2]{Mimura-Toda:topliegroups}. 

Suppose that there exists an integer $i$ such that $x_j$ is transgressive for $j < i$ but not 
$x_i$. Then we see that for some $r < \deg x_i +1$, 
$d_r(x_i) \neq 0$ and $d_p(x_i) = 0$ if $p <r$.  
We write 
$$
d_r(x_i) = \sum_lb_l\otimes x_{l_1}\cdots x_{l_{s_l}}, 
$$
where each $b_l$ is a non-zero element of $H^*(BG; \K)$ and $1 \leq l_u \leq N$ for any 
$l$ and $u$. The equality $\mu_r^*d_r(x_i) = (d_r\otimes 1)\mu_r^*(x_i)$ implies that 
\begin{eqnarray*}
\sum_lb_l\otimes x_{l_1}\cdots x_{l_{s_l-1}}\otimes x_{l_{s_l}} + \cdots 
&=& d_r\otimes 1(1\otimes x_i \otimes1 + 1\otimes 1\otimes x_i) \\
&=&
\sum_lb_l\otimes x_{l_1}\cdots x_{l_{s_l}}\otimes 1, 
\end{eqnarray*}
which is a contradiction. Observe that $x_i$ and $x_{l_u}$ are primitive. 
Thus it follows that $x_i$ is transgressive for any $1\leq i \leq N$.  

In the case where $\text{Char}\K=p\neq 2$, since $j^*$ is surjective by assumption, 
it follows from the argument in \cite[Remark 1.4]{Kuri:moduleadjoint} that 
$H^*(G; {\mathbb Z})$ has no $p$-torsion. Observe that to obtain the result, 
the connectedness of the loop space is assumed. 
By virtue of  \cite[Chapter 7, Theorem 2.12]{Mimura-Toda:topliegroups}, 
we see that $H^*(BG; \K)$ is a polynomial algebra. This completes the proof. 
\end{proof}
The following theorem give another characterisation of our hypothesis 
$H^*(BG)$ polynomial.
\begin{thm}
Let $G$ be a connected Lie group. 
Then the following three conditions are equivalent:

1) $H^*(BG;\K)$ is a polynomial algebra on even degree generators.

2) $BG$ is $\K$-formal and $H^*(BG;\K)$ is strictly commutative.

3) The singular cochain algebra $S^*(BG;\K)$ is weakly equivalent as algebras
to a strictly commutative differential graded algebra $A$.
\end{thm}
Stricly commutative means that  $a^2 = 0$ if $a\in A^{odd}$
($\K$ can be a field of characteristics two).
\begin{proof}
$1\Rightarrow 2$.
Suppose that $H^*(BG;\K)$ is a polynomial algebra.
Then by the beginning of section~\ref{Relation Hochschild cohomology},
$BG$ is $\K$-formal.

$2\Rightarrow 3$. Formality means that we can take $A=(H^*(BG;\K),0)$ in 3).

$3\Rightarrow 1$.
Let $Y$ be a simply connected space such that $S^*(Y;\K)$ is weakly equivalent as algebras to a strictly commutative differential graded algebra $A$.
Let $(\Lambda V,d)$ be a minimal Sullivan model of $A$.
Consider the semifree-$(\Lambda V,d)$ resolution of $(\K,0)$, $(\Lambda V\otimes\Gamma sV,D)$
 given in~\cite[Proposition 2.4]{Halperin:unieal} or~\cite[Lemma 7.2]{MenichiL:cohaf}.
 Then the tensor product of commutative differential graded algebras
 $(\K,0)\otimes_{(\Lambda V,d)} (\Lambda V\otimes\Gamma sV,D)\cong (\Gamma sV,\overline{D})$ has a trivial differential $\overline{D}=0$~\cite[Corollary 2.6]{Halperin:unieal}. 
 Therefore we have the isomorphisms of graded vector spaces
 $$
 H^*(\Omega Y)\cong\text{Tor}^{S^*(Y;\K)}(\K,\K)\cong \text{Tor}^{(\Lambda V,d)}(\K,\K)\cong H_*(\Gamma sV,\overline{D})\cong \Gamma sV.
 $$
 If $H^*(\Omega Y)$ is of finite dimension then the suspension of $V$, $sV$, must
 be concentrated in odd degree and so $V$ must be in even degree and $d=0$,
 i. e. $Y$ is $\K$-formal and $H^*(Y)$ is polynomial in even degree.
\end{proof}
\bibliography{Bibliographie}

\providecommand{\bysame}{\leavevmode\hbox to3em{\hrulefill}\thinspace}
\providecommand{\MR}{\relax\ifhmode\unskip\space\fi MR }
\providecommand{\MRhref}[2]{%
  \href{http://www.ams.org/mathscinet-getitem?mr=#1}{#2}
}
\providecommand{\href}[2]{#2}
\begin{thebibliography}{10}

\bibitem{BGNX:Stringstackspublie}
Kai Behrend, Gr{\'e}gory Ginot, Behrang Noohi, and Ping Xu, \emph{String
  topology for stacks}, Ast\'erisque (2012), no.~343, xiv+169.

\bibitem{Berglund-Borjeson:freeloopmanifold}
A.~{Berglund} and K.~{B{\"o}rjeson}, \emph{{Free Loop Space Homology of Highly
  Connected Manifolds}}, ArXiv e-prints (2015).

\bibitem{Bredon:Sheaf}
Glen~E. Bredon, \emph{Sheaf theory}, second ed., Graduate Texts in Mathematics,
  vol. 170, Springer-Verlag, New York, 1997.

\bibitem{Chas-Sullivan:stringtop}
M.~Chas and D.~Sullivan, \emph{String topology}, preprint: math.GT/9911159,
  1999.

\bibitem{Chataur-Leborgne:loophomologycomplexprojective}
David Chataur and Jean-Fran{\c{c}}ois Le~Borgne, \emph{On the loop homology of
  complex projective spaces}, Bull. Soc. Math. France \textbf{139} (2011),
  no.~4, 503--518.

\bibitem{Chataur-Menichi:stringclass}
David Chataur and Luc Menichi, \emph{String topology of classifying spaces}, J.
  Reine Angew. Math. \textbf{669} (2012), 1--45.

\bibitem{Earle-Schatz}
C.~J. Earle and A.~Schatz, \emph{Teichm\"uller theory for surfaces with
  boundary}, J. Differential Geometry \textbf{4} (1970), 169--185.

\bibitem{Farb-Margalit:primermcg}
Benson Farb and Dan Margalit, \emph{A primer on mapping class groups},
  Princeton Mathematical Series, vol.~49, Princeton University Press,
  Princeton, NJ, 2012.

\bibitem{Felix-Halperin-Thomas:ratht}
Y.~F{\'e}lix, S.~Halperin, and J.-C. Thomas, \emph{Rational homotopy theory},
  Graduate Texts in Mathematics, vol. 205, Springer-Verlag, 2000.

\bibitem{Felix-Menichi-Thomas:GerstduaiHochcoh}
Y.~F{\'e}lix, L.~Menichi, and J.-C. Thomas, \emph{Gerstenhaber duality in
  {H}ochschild cohomology}, J. Pure Appl. Algebra \textbf{199} (2005), no.~1-3,
  43--59.

\bibitem{Felix-Thomas:ratBVstringtop}
Yves F{\'e}lix and Jean-Claude Thomas, \emph{Rational {BV}-algebra in string
  topology}, Bull. Soc. Math. France \textbf{136} (2008), no.~2, 311--327.

\bibitem{Felix-Thomas:stringtopGorenstein}
\bysame, \emph{String topology on {G}orenstein spaces}, Math. Ann. \textbf{345}
  (2009), no.~2, 417--452.

\bibitem{FreedHopkinsTeleman:twistedKtheory3}
Daniel~S. Freed, Michael~J. Hopkins, and Constantin Teleman, \emph{Loop groups
  and twisted {$K$}-theory {III}}, Ann. of Math. (2) \textbf{174} (2011),
  no.~2, 947--1007.

\bibitem{Godin:higherstring}
V.~Godin, \emph{Higher string topology operations}, preprint:
  math.AT/0711.4859, 2007.

\bibitem{Halperin:unieal}
S.~Halperin, \emph{Universal enveloping algebras and loop space homology}, J.
  Pure Appl. Algebra \textbf{83} (1992), 237--282.

\bibitem{Hamstrom:homeosurface}
Mary-Elizabeth Hamstrom, \emph{Homotopy groups of the space of homeomorphisms
  on a {$2$}-manifold}, Illinois J. Math. \textbf{10} (1966), 563--573.

\bibitem{Hepworth:stringcomplexprojective}
R.~A. {Hepworth}, \emph{{String topology for complex projective spaces}}, ArXiv
  e-prints (2009).

\bibitem{Hepworth-Lahtinen:stringclassifying}
Richard Hepworth and Anssi Lahtinen, \emph{On string topology of classifying
  spaces}, Adv. Math. \textbf{281} (2015), 394--507.

\bibitem{Hepworth:stringLie}
Richard~A. Hepworth, \emph{String topology for {L}ie groups}, J. Topol.
  \textbf{3} (2010), no.~2, 424--442.

\bibitem{Iwase:adjoint}
Norio Iwase, \emph{Adjoint action of a finite loop space}, Proc. Amer. Math.
  Soc. \textbf{125} (1997), no.~9, 2753--2757.

\bibitem{Johnson:homeosurface}
Dennis~L. Johnson, \emph{Homeomorphisms of a surface which act trivially on
  homology}, Proc. Amer. Math. Soc. \textbf{75} (1979), no.~1, 119--125.

\bibitem{Keller:Deformation}
B.~Keller, \emph{Deformation quantization after {K}ontsevich and {T}amarkin},
  D\'eformation, quantification, th\'eorie de {L}ie, Panor. Synth\`eses,
  vol.~20, Soc. Math. France, Paris, 2005, pp.~19--62.

\bibitem{Kishi-Kono:freetwisted}
Daisuke Kishimoto and Akira Kono, \emph{On the cohomology of free and twisted
  loop spaces}, J. Pure Appl. Algebra \textbf{214} (2010), no.~5, 646--653.

\bibitem{Kock:Frob2TQFT}
Joachim Kock, \emph{Frobenius algebras and 2{D} topological quantum field
  theories}, London Mathematical Society Student Texts, vol.~59, Cambridge
  University Press, Cambridge, 2004.

\bibitem{Kono-Kuri:modulesplitting}
Akira Kono and Katsuhiko Kuribayashi, \emph{Module derivations and
  cohomological splitting of adjoint bundles}, Fund. Math. \textbf{180} (2003),
  no.~3, 199--221.

\bibitem{Kupers:stringop}
Alexander Kupers, \emph{String topology operations, master thesis}, Utrecht
  University, The Netherlands, 2011.

\bibitem{Kuri:moduleadjoint}
Katsuhiko Kuribayashi, \emph{Module derivations and the adjoint action of a
  finite loop space}, J. Math. Kyoto Univ. \textbf{39} (1999), no.~1, 67--85.

\bibitem{KuriMeniNaito:DerivedEMSS}
Katsuhiko Kuribayashi, Luc Menichi, and Takahito Naito, \emph{Derived string
  topology and the {E}ilenberg-{M}oore spectral sequence}, Israel J. Math.
  \textbf{209} (2015), no.~2, 745--802.

\bibitem{Lahtinen:higherop}
A.~{Lahtinen}, \emph{{Higher operations in string topology of classifying
  spaces}}, ArXiv e-prints, to appear in Mathematische Annalen. (2015).

\bibitem{McCleary:usersguide}
John McCleary, \emph{A user's guide to spectral sequences}, second ed.,
  Cambridge Studies in Advanced Mathematics, vol.~58, Cambridge University
  Press, Cambridge, 2001.

\bibitem{MenichiL:cohrfl}
L.~Menichi, \emph{The cohomology ring of free loop spaces}, Homology Homotopy
  Appl. \textbf{3} (2001), no.~1, 193--224.

\bibitem{MenichiL:cohaf}
\bysame, \emph{On the cohomology algebra of a fiber}, Algebr. Geom. Topol.
  \textbf{1} (2001), 719--742.

\bibitem{menichi:stringtopspheres}
\bysame, \emph{String topology for spheres}, Comment. Math. Helv. \textbf{84}
  (2009), no.~1, 135--157.

\bibitem{Menichi:BVmorphismdoublefree}
\bysame, \emph{A {B}atalin-{V}ilkovisky algebra morphism from double loop
  spaces to free loops}, Trans. Amer. Math. Soc. \textbf{363} (2011),
  4443--4462.

\bibitem{Milnor-Moore}
John~W. Milnor and John~C. Moore, \emph{On the structure of {H}opf algebras},
  Ann. of Math. (2) \textbf{81} (1965), 211--264.

\bibitem{Milnor-Stasheff}
John~W. Milnor and James~D. Stasheff, \emph{Characteristic classes}, Princeton
  University Press, Princeton, N. J., 1974, Annals of Mathematics Studies, No.
  76.

\bibitem{Mimura-Toda:topliegroups}
Mamoru Mimura and Hirosi Toda, \emph{Topology of {L}ie groups. {I}, {II}},
  vol.~91, American Mathematical Society, Providence, RI, 1991.

\bibitem{Spanier:livre}
Edwin~H. Spanier, \emph{Algebraic topology}, Springer-Verlag, New York, 1981.

\bibitem{Halperin-Stasheff:diff}
James Stasheff and Steve Halperin, \emph{Differential algebra in its own rite},
  Proceedings of the {A}dvanced {S}tudy {I}nstitute on {A}lgebraic {T}opology
  ({A}arhus {U}niv., {A}arhus 1970), {V}ol. {III}, Mat. Inst., Aarhus Univ.,
  Aarhus, 1970, pp.~567--577. Various Publ. Ser., No. 13.

\bibitem{Tamanoi:BVLalg}
H.~Tamanoi, \emph{Batalin-{V}ilkovisky {L}ie algebra structure on the loop
  homology of complex {S}tiefel manifolds}, Int. Math. Res. Not. (2006), 1--23.

\bibitem{tamanoi:capproducts}
Hirotaka Tamanoi, \emph{Cap products in string topology}, Algebr. Geom. Topol.
  \textbf{9} (2009), no.~2, 1201--1224.

\bibitem{Tamanoi:stabletrivial}
\bysame, \emph{Stable string operations are trivial}, Int. Math. Res. Not. IMRN
  (2009), no.~24, 4642--4685.

\bibitem{tamanoi-2007}
\bysame, \emph{Loop coproducts in string topology and triviality of higher
  genus {TQFT} operations}, J. Pure Appl. Algebra \textbf{214} (2010), no.~5,
  605--615.

\bibitem{thesedewahl}
Nathalie Wahl, \emph{Ribbon braids and related operads, ph.d. thesis}, Oxford
  university,http://www.math.ku.dk/~wahl/, 2001.

\bibitem{Westerland:stringhomologyspheres}
C.~Westerland, \emph{String homology of spheres and projective spaces}, Algebr.
  Geom. Topol. \textbf{7} (2007), 309--325.

\bibitem{Yang:BV}
Tian Yang, \emph{A {B}atalin-{V}ilkovisky algebra structure on the {H}ochschild
  cohomology of truncated polynomials}, Topology Appl. \textbf{160} (2013),
  no.~13, 1633--1651.

\end{thebibliography}
\bibliographystyle{amsplain}
\end{document}